\documentclass[fleqn,preprint,3p,a4paper]{elsarticle}

\usepackage{amssymb}
\usepackage{amsmath}
\usepackage{amsthm}
\usepackage{dcolumn}
\usepackage{endnotes}
\usepackage{tabularx}
\usepackage[matrix,arrow]{xy}
\usepackage{wasysym}

\newtheorem{theorem}{Theorem}[section]
\newtheorem{proposition}[theorem]{Proposition}
\newtheorem{lemma}[theorem]{Lemma}
\newtheorem{corollary}[theorem]{Corollary}

\theoremstyle{definition}
\newtheorem{definition}[theorem]{Definition}
\newtheorem{example}[theorem]{Example}
\newtheorem{remark}[theorem]{Remark}
\newtheorem{question}[theorem]{Question}

\newcommand{\ir}{{\mathsf{Irr}}}

\newcommand{\cl}{{\rm cl}}
\newcommand{\ii}{{\rm int}}

\newcommand{\ua}{\mathord{\uparrow}}
\newcommand{\da}{\mathord{\downarrow}}

\newcommand{\mk}{\mathord{\mathsf{K}}}

\journal{Annals of Pure and Applied Logic}

\begin{document}

\begin{frontmatter}



\title{On H-sober spaces and H-sobrifications of $T_0$ spaces \tnoteref{t1}}
\tnotetext[t1]{This research was supported by the National Natural Science Foundation of China (11661057) and NSF of Jiangxi Province (20192ACBL20045)}

\author[X. Xu]{Xiaoquan Xu}
\ead{xiqxu2002@163.com}
\address[X. Xu]{School of Mathematics and Statistics,
Minnan Normal University, Zhangzhou 363000, China}

\begin{abstract}
In this paper, we provide a uniform approach to $d$-spaces, sober spaces and well-filtered spaces, and develop a general framework for dealing with all these spaces. The concepts of irreducible subset systems (R-subset systems for short), H-sober spaces and super H-sober spaces for a general R-subset system H are introduced. It is proved that the product space of a family of $T_0$ spaces is H-sober iff each factor space is H-sober, and if H has a natural property (called property M), then the super H-sobriety is a special type of H-sobriety, and hence the product space of a family of $T_0$ spaces is super H-sober iff each factor space is super H-sober. Let $\mathbf{Top}_0$ be the category of all $T_0$ spaces with continuous mappings. For a $T_0$ space $X$ and an H-sober space $Y$, we show that the function space $\mathbf{Top}_0(X, Y)$ equipped with the topology of pointwise convergence is H-sober. Furthermore, if H has property M and $Y$ is a super H-sober space, then the function space $\mathbf{Top}_0(X, Y)$ equipped with the topology of pointwise convergence is super H-sober. One immediate corollary is that for a $T_0$ space $X$ and a well-filtered space $Y$, the function space $\mathbf{Top}_0(X, Y)$ equipped with the topology of pointwise convergence is well-filtered. For an R-subset system H having property M, the Smyth power space of a H-sober space is not H-sober in general. But for the super {\rm H}-sobriety, we prove that a $T_0$ space $X$ is super H-sober iff its Smyth power space $P_S(X)$ is super H-sober. A direct construction of the H-sobrifcations and super H-sobrifications of $T_0$ spaces is given. So the category of all H-sober spaces is reflective in $\mathbf{Top}_0$, and the category of all super H-sober spaces is also reflective in $\mathbf{Top}_0$ if H has property M. It is shown that the H-sobrification preserves finite products of $T_0$ spaces, and the super H-sobrification preserves finite products of $T_0$ spaces if H has property M.

\end{abstract}





\end{frontmatter}


\section{Introduction}

In non-Hausdorff topology and domain theory, the $d$-spaces, well-filtered spaces and sober spaces form three of the most important classes (see [2-21, 23, 24, 26-36]). A space $X$ is a $d$-space provided for any directed subset $D$ of $X$ (under the specialization order of $X$), there exists a unique point $x\in X$ such that $\overline{D}=\overline{\{x\}}$; $X$ is sober provided for any irreducible closed subset $A$ of $X$, there is a unique point $a\in X$ such that $A=\overline{\{a\}}$. A $T_0$ space $Y$ is well-filtered if for any filtered family $\mathcal{K}\subseteq \mathord{\mathsf{K}}(X)$ and any open set $U$, $\bigcap\mathcal{K}{\subseteq} U$ implies $K{\subseteq} U$ for some $K{\in}\mathcal{K}$. So the $d$-spaces and the sober spaces are in the same type, and the well-filtered spaces are in a different type. More precisely, a $T_0$ space $X$ is well-filtered iff its Smyth power space $P_S(X)$ is a $d$-space (see \cite{xi-zhao-MSCS-well-filtered} or \cite{xuxizhao}). So the well-filtered spaces are in a "higher" level. Let $\mathbf{Top}_0$ be the category of all $T_0$ spaces with continuous mappings and $\mathbf{Sob}$ that of all sober spaces with continuous mappings. Denote the category of all $d$-spaces with continuous mappings and that of all well-filtered spaces with continuous mappings respectively by $\mathbf{Top}_d$ and $\mathbf{Top}_w$. It is well-known that $\mathbf{Sob}$ is reflective in $\mathbf{Top}_0$ (see \cite{redbook, Jean-2013}). Using $d$-closures, Wyler \cite{Wyler} proved that $\mathbf{Top}_d$ is reflective in $\mathbf{Top}_0$. Later, Ershov \cite{Ershov_1999} showed that the $d$-completion (i.e., the $d$-reflection) of $X$ can be obtained by adding the closure of directed sets onto $X$ (and then repeating this process by transfinite induction). In \cite{Keimel-Lawson}, using Wyler's method, Keimel and Lawson proved that for a full subcategory $\mathbf{K}$ of $\mathbf{Top}_0$ containing $\mathbf{Sob}$, if $\mathbf{K}$ has certain properties, then $\mathbf{K}$ is reflective in $\mathbf{Top}_0$. They showed that $\mathbf{Top}_d$ and some other categories have such properties. For quite a long time, it is not known whether $\mathbf{Top}_w$ is reflective in $\mathbf{Top}_0$. Recently, following Keimel and Lawson's method, a positive answer to this problem was given in \cite{wu-xi-xu-zhao-19}. Following Ershov's method, a construction of the well-filtered reflection of $T_0$ spaces was presented in \cite{Shenchon}. In \cite {xu20}, for a full subcategory $\mathbf{K}$ of $\mathbf{Top}_0$ containing $\mathbf{Sob}$, the author  provided a direct approach to $\mathbf{K}$-reflections of $T_0$ spaces (see also \cite{ZhangLi}).

In this paper, we provide a uniform approach to $d$-spaces, sober spaces and well-filtered spaces, and develop a general framework for dealing with all these spaces. Based on the concept of subset system in \cite{Wright78} and the Topological Rudin Lemma in \cite{Klause-Heckmann}, we introduce the concepts of \emph{irreducible subset system} (\emph{R}-\emph{subset system} for short), \emph{H}-\emph{sober space} and \emph{super H}-\emph{sober space} for a general R-subset system H. It is proved that the product space of a family of $T_0$ spaces is H-sober iff each factor space is H-sober, and if H has a natural property (called \emph{property M}), then the super H-sobriety is a special type of H-sobriety, and hence the product space of a family of $T_0$ spaces is super H-sober iff each factor space is super H-sober. As a direct corollary, we get that the product space of a family of $T_0$ spaces is well-filtered iff each factor space is well-filtered. For a $T_0$ space $X$ and an H-sober space $Y$, we show that the function space $\mathbf{Top}_0(X, Y)$ of all continuous functions $f : X \longrightarrow Y$ equipped with the topology of pointwise convergence (i.e., the relative product topology) is H-sober. Furthermore, if H has property M and $Y$ is a super H-sober space, then the function space $\mathbf{Top}_0(X, Y)$ equipped with the topology of pointwise convergence is super H-sober. One immediate corollary is that for a $T_0$ space $X$ and a well-filtered space $Y$, the function space $\mathbf{Top}_0(X, Y)$ equipped with the topology of pointwise convergence is well-filtered.

For an R-subset system H having property M, the Smyth power space of a H-sober space is not H-sober in general. But for the super {\rm H}-sobriety, we prove that a $T_0$ space $X$ is super H-sober iff its Smyth power space $P_S(X)$ is super H-sober. In particular, we have that a $T_0$ space $X$ is well-filtered iff $P_S(X)$ is a $d$-space iff $P_S(P_S(X))$ is a $d$-space (that is, $P_S(X)$ is well-filtered).

Let $\mathbf{H}$-$\mathbf{Sober}$ be the category of all H-sober spaces with continuous mappings and $\mathbf{SH}$-$\mathbf{Sober}$  that of all super H-sober spaces with continuous mappings. In order to investigate the reflections of $\mathbf{H}$-$\mathbf{Sober}$ and $\mathbf{SH}$-$\mathbf{Sober}$  in $\mathbf{Top}_0$, we give a direct construction of the H-sobrifcations and super H-sobrifications of $T_0$ spaces, and consequently we get that  $\mathbf{H}$-$\mathbf{Sober}$ is reflective in $\mathbf{Top}_0$, and $\mathbf{SH}$-$\mathbf{Sober}$ is also reflective in $\mathbf{Top}_0$ if H has property M. It is proved that the H-sobrification preserves finite products of $T_0$ spaces, and the super H-sobrification preserves finite products of $T_0$ spaces if H has property M.

Applying directly the above results to the R-subset systems of  $\mathcal D$ and $\mathcal R$, where for each $T_0$ space $X$, $\mathcal D(X)$ is the set of all directed subsets of $X$ and $\mathcal R(X)$ is the set of all countable irreducible subsets of $X$, we immediately get the following known results: (1) $\mathbf{Top}_d$, $\mathbf{Sob}$ and $\mathbf{Top}_w$ are all reflective in $\mathbf{Top}_0$; and (2) the $d$-space reflection and the well-filtered reflection preserve finite products of $T_0$ spaces.

\section{Preliminary}

In this section, we briefly  recall some fundamental concepts and notations that will be used in the paper. Some basic properties of directed sets, irreducible sets, sober spaces, compact saturated sets, well-filtered spaces and Smyth power spaces are presented. For further details, we refer the reader to \cite{redbook, Jean-2013,  Schalk}.

For a poset $P$ and $A\subseteq P$, let
$\mathord{\downarrow}A=\{x\in P: x\leq  a \mbox{ for some }
a\in A\}$ and $\mathord{\uparrow}A=\{x\in P: x\geq  a \mbox{
	for some } a\in A\}$. For  $x\in P$, we write
$\mathord{\downarrow}x$ for $\mathord{\downarrow}\{x\}$ and
$\mathord{\uparrow}x$ for $\mathord{\uparrow}\{x\}$. A subset $A$
is called a \emph{lower set} (resp., an \emph{upper set}) if
$A=\mathord{\downarrow}A$ (resp., $A=\mathord{\uparrow}A$). Define $A^{\uparrow}=\{x\in P : x \mbox{ is an upper bound of } A \mbox{ in }P\}$. Dually, define $A^{\downarrow}=\{x\in P : x \mbox{ is a lower bound of } $A$ \mbox{ in } $P$\}$. The set $A^{\delta}=(A^{\ua})^{\da}$ is called the \emph{cut} generated by $A$. Let $P^{(<\omega)}=\{F\subseteq P : F \mbox{~is a nonempty finite set}\}$, $P^{(\leqslant\omega)}=\{F\subseteq P : F \mbox{~is a nonempty countable set}\}$ and $\mathbf{Fin} ~P=\{\uparrow F : F\in P^{(<\omega)}\}$.
 For a nonempty subset $A$ of $P$, define $\mathrm{max}(A)=\{a\in A : a \mbox{~ is a maximal element of~} A\}$ and $\mathrm{min}(A)=\{a\in A : a \mbox{~ is a minimal element of~} A\}$.

 The category of all sets with mappings is denoted by $\mathbf{Set}$. For a set $X$ and $A, B\subseteq X$, $A\subset B$ means that $A\subseteq B$ but $A\neq B$, that is, $A$ is a proper subset of $B$. Let $|X|$ be the cardinality of $X$ and $\omega=|\mathbb{N}|$, where $\mathbb{N}$ is the set of all natural numbers.

 Let $\mathbf{Poset}$ denote the category of all posets with order-preserving mappings. For $P\in $\emph{ob}($\mathbf{Poset}$) and $D\subseteq P$,  $D$ is called \emph{directed} if every two
elements in $D$ have an upper bound in $D$. The set of all directed sets of $P$ is denoted by $\mathcal D(P)$. $I\subseteq P$ is called an \emph{ideal} of $P$ if $I$ is a directed lower subset of $P$. Let $\mathrm{Id} (P)$ be the poset of all ideals of $P$ with the order of set inclusion. Dually, we define the concept of \emph{filters} and denote the poset of all filters of $P$ by $\mathrm{Filt}(P)$. The ideals $\da x$ and filters $\ua x$ are called \emph{principal ideals} and \emph{principal filters} respectively. A poset $P$ is called \emph{chain-complete}, if every chain of $P$ has a least upper bound in $P$. More general, $P$ is called a
\emph{directed complete poset}, or \emph{dcpo} for short, provided that $\bigvee D$ exists in $P$ for any
$D\in \mathcal D(P)$.

\begin{lemma}\emph{(\cite{Mar76})}\label{Markowsky} If $D$ is an infinite directed set, then there exists a transfinite sequences $D_\alpha, \alpha< |D|$, of
directed subsets of $D$ having the following properties:
\begin{enumerate}[\rm (1)]
\item for each $\alpha$, if $\alpha$ is finite, so is $D_\alpha$, while if $\alpha$ is infinite $|D_\alpha|=|\alpha|$ (thus for all $\alpha, |D_\alpha|<|D|$).
    \item if $\alpha<\beta<|D|, D_\alpha\subset D_\beta$ (that is, $D_\alpha$ is a proper subset of $D_\beta$).
    \item $D=\bigcup_{\alpha<|D|} D_\alpha$.
\end{enumerate}
\end{lemma}

\begin{corollary}\emph{(\cite{Mar76})}\label{dcpo=chain comp} For a poset $P$, $P$ is a dcpo if and only if $P$ is chain-complete.
\end{corollary}

\begin{lemma}\label{D countable chain}
Let $P$ be a poset and $D$ a countable directed subset of $P$. Then there exists a countable chain $C\subseteq D$ such that $\da D=\da C$. Hence, $\bigvee C$ exists and $\bigvee C=\bigvee D$ whenever $\bigvee D$ exists.
\end{lemma}
\begin{proof}
	If $|D|<\omega$, then $D$ contains a largest element $d$, so let $C=\{d\}$, which satisfies the requirement.
	
	Now assume $|D|=\omega$ and let $D=\{d_n:n<\omega\}$. We use induction on $n\in\omega$ to  define $C=\{c_n:n<\omega\}$.
	More precisely, let $c_0=d_0$ and let $c_{n+1}$ be an upper bound of $\{d_{n+1},c_0, c_1,c_2\ldots,c_n\}$ in $D$. It is clear that $C$ is a chain and $\da D=\da C$.
\end{proof}

As in \cite{redbook}, the \emph{upper topology} on a poset $Q$, generated
by the complements of the principal ideals of $Q$, is denoted by $\upsilon (Q)$. A subset $U$ of $Q$ is \emph{Scott open} if
(i) $U=\mathord{\uparrow}U$ and (ii) for any directed subset $D$ for
which $\bigvee D$ exists, $\bigvee D\in U$ implies $D\cap
U\neq\emptyset$. All Scott open subsets of $Q$ form a topology.
This topology is called the \emph{Scott topology} on $Q$ and
denoted by $\sigma(Q)$. The space $\Sigma~\!\! Q=(Q,\sigma(Q))$ is called the
\emph{Scott space} of $Q$. The upper sets of $Q$ form the (\emph{upper}) \emph{Alexandroff topology} $\alpha (Q)$.

For $X\in$ \emph{ob}($\mathbf{Top}_0$), we use $\leq_X$ to represent the \emph{specialization order} of $X$, that is, $x\leq_X y$ if{}f $x\in \overline{\{y\}}$). In the following, when a $T_0$ space $X$ is considered as a poset, the order always refers to the specialization order if no other explanation. Let $\mathcal O(X)$ (resp., $\Gamma (X)$) be the set of all open subsets (resp., closed subsets) of $X$. Define $\mathcal S_c(X)=\{\overline{{\{x\}}} : x\in X\}$ and $\mathcal D_c(X)=\{\overline{D} : D\in \mathcal D(X)\}$. A space $X$ is called a $d$-\emph{space} (or \emph{monotone convergence space}) if $X$ (with the specialization order) is a dcpo
 and $\mathcal O(X) \subseteq \sigma(X)$ (cf. \cite{redbook, Wyler}). Clearly, for a dcpo $P$, $\Sigma~\!\! P$ is a $d$-space. The category of all $d$-spaces with continuous mappings is denoted by $\mathbf{Top}_d$.

One can directly get the following result (cf. \cite{xu-shen-xi-zhao1}).

\begin{proposition}\label{d-spacecharac1} For a $T_0$ space $X$, the following conditions are equivalent:
\begin{enumerate}[\rm (1)]
	        \item $X$ is a $d$-space.
            \item For each $D\in \mathcal D(X)$, there exists a \emph{(}unique\emph{)} point $x\in X$ such that $\overline{D}=\overline{\{x\}}$.
\end{enumerate}
\end{proposition}

\begin{lemma}\label{d-space max point}
Let $X$ be a $d$-space and $A$ a nonempty closed subset of $X$. Then $A=\da \mathrm{max}(A)$, and hence $\mathrm{max}(A)\neq\emptyset$.
\end{lemma}
\begin{proof}
 For $x\in A$, by Zorn's Lemma there is a maximal chain $C_x$ in $A$ with $x\in C_x$. Since $X$ is a $d$-space, $c_x=\bigvee C_x$ exists and $c_x\in A$. By the maximality of $C_x$, we have $c_x\in \mathrm{max}(A)$. Clearly, $x\leq c_x$. Therefore, $A\subseteq \da \mathrm{max}(A)\subseteq \da A=A$, and hence $A=\da \mathrm{max}(A)$.
\end{proof}

\begin{remark}\label{C up=cl C up} Let $X$ be a $T_0$ space, $C\subseteq X$ and $x\in X$. Then the following four conditions are equivalent:
\begin{enumerate}[\rm (1)]
   \item $x\in C^{\ua}$;
   \item $C\subseteq \da x$;
   \item $\overline{C}\subseteq \da x$;
   \item $x\in \overline{C}^{\ua}$.
\end{enumerate}
Therefore, $\bigcap_{c\in C}\ua c=C^{\ua}=\overline{C}^{\ua}=\bigcap_{b\in \overline{C}}\ua b$, and $ C^{\delta}=\bigcap \{\da x : C\subseteq \da x\}=\bigcap \{\da x : \overline{C}\subseteq \da x\}=\overline{C}^{\delta}.$ Furthermore, $\bigvee C$ exists in $X$ if{}f $\bigvee \overline{C}$ exists in $X$, and $\bigvee C =\bigvee \overline{C}$ if they exist in $X$.
\end{remark}

For a $T_0$ space $X$ and a nonempty subset $A$ of $X$, $A$ is \emph{irreducible} if for any $\{F_1, F_2\}\subseteq \Gamma (X)$, $A \subseteq F_1\cup F_2$ implies $A \subseteq F_1$ or $A \subseteq  F_2$.  Denote by $\ir(X)$ (resp., $\ir_c(X)$) the set of all irreducible (resp., irreducible closed) subsets of $X$. Clearly, every subset of $X$ that is directed under $\leq_X$ is irreducible.

The following two lemmas on irreducible sets are well-known (cf. \cite{Klause-Heckmann}).

\begin{lemma}\label{irrsubspace}
Let $X$ be a space and $Y$ a subspace of $X$. Then the following conditions are equivalent for a
subset $A\subseteq Y$:
\begin{enumerate}[\rm (1)]
	\item $A$ is an irreducible subset of $Y$.
	\item $A$ is an irreducible subset of $X$.
	\item ${\rm cl}_X A$ is an irreducible closed subset of $X$.
\end{enumerate}
\end{lemma}

\begin{lemma}\label{irrimage}
	If $f : X \longrightarrow Y$ is continuous and $A\in\ir (X)$, then $f(A)\in \ir (Y)$.
\end{lemma}

\begin{remark}\label{subspaceirr}  If $Y$ is a subspace of a space $X$ and $A\subseteq Y$, then by Lemma \ref{irrsubspace}, $\ir (Y)=\{B\in \ir(X) : B\subseteq Y\}\subseteq \ir (X)$ and  $\ir_c (Y)=\{B\in \ir(X) : B\in \Gamma(Y)\}\subseteq \ir (X)$. If $Y\in \Gamma(X)$, then $\ir_c(Y)=\{C\in \ir_c(X) : C\subseteq Y\}\subseteq \ir_c (X)$.
\end{remark}

\begin{lemma}\label{irrprod}\emph{(\cite{Shenchon})}
	Let	$\{X_i : i\in I\}$ be a family of $T_0$ spaces and $X=\prod_{i\in I}X_i$ the product space. If  $A$ is an irreducible subset of $X$, then $\cl_X(A)=\prod_{i\in I}\cl_{X_i}~(p_i(A))$, where $p_i : X \longrightarrow X_i$ is the $i$th projection for each $i\in I$.
\end{lemma}

\begin{lemma}\label{prodirr}
	Let	$X=\prod_{i\in I}X_i$ be the product space of a family $\{X_i : i\in I\}$ of $T_0$ spaces and and $A_i\subseteq X_i$ for each $i\in I$. Then the following two conditions are equivalent:
\begin{enumerate}[\rm (1)]
	\item $\prod_{i\in I}A_i\in \ir (X)$.
	\item $A_i\in \ir (X_i)$ for each $i\in I$.
\end{enumerate}
\end{lemma}
\begin{proof}  (1) $\Rightarrow$ (2): By Lemma \ref{irrimage}.

(2) $\Rightarrow$ (1): Let $A=\prod_{i\in I}A_i$. For $U, V\in \mathcal O(X)$, if $A\cap U\neq\emptyset\neq A\cap V$, then there exist $I_1, I_2\in I^{(<\omega)}$ and $(U_i, V_j)\in \mathcal O(X_i)\times \mathcal O(X_j)$ for all $(i, j)\in I_1\times I_2$ such that $\bigcap_{i\in I_1}p_i^{-1}(U_i)\subseteq U$, $\bigcap_{j\in I_2}p_j^{-1}(V_j)\subseteq V$ and $A\cap \bigcap_{i\in I_1}p_i^{-1}(U_i)\neq\emptyset\neq A\cap \bigcap_{j\in I_2}p_i^{-1}(V_j)$. Let $I_3=I_1\cup I_2$. Then $I_3$ is finite. For $i\in I_3\setminus I_1$ and $j\in I_3\setminus I_2$, let $U_i=X_i$ and $V_j=X_j$. Then for each $i\in I_3$, we have $A_i\cap U_i\neq\emptyset\neq A_i\cap V_i$, and whence $A_i\cap U_i\cap V_i\neq\emptyset$ by $A_i\in \ir (X_i)$. It follows that $A\cap \bigcap_{i\in I_1}p_i^{-1}(U_i)\cap \bigcap_{j\in I_2}p_i^{-1}(V_j)\neq\emptyset$, and consequently, $A\cap U\cap V\neq \emptyset$. Thus $A\in \ir (X)$.

\end{proof}

By Lemma \ref{irrprod} and Lemma \ref{prodirr}, we obtain the following corollary.

\begin{corollary}\label{irrcprod} Let $\{X_i : i\in I\}$ be a family of $T_0$ spaces and $X=\prod_{i\in I}X_i$ the product space. If  $A\in \ir_c(X)$, then $A=\prod_{i\in I}p_i(A)$ and $p_i(A)\in \ir_c (X_i)$ for each $i\in I$.
\end{corollary}

A topological space $X$ is called \emph{sober}, if for any  $F\in\ir_c(X)$, there is a unique point $a\in X$ such that $F=\overline{\{a\}}$. The category of all sober spaces with continuous mappings is denoted by $\mathbf{Sob}$. Let $\mathrm{OFilt(\mathcal O(X))}=\sigma (\mathcal O(X))\bigcap \mathrm{Filt}(\mathcal O(X))$. A set $\mathcal{U}\subseteq \mathcal O(X)$ is called an \emph{open filter} if $\mathcal U\in \mathrm{OFilt(\mathcal O(X))}$. For a $K\in \mk (X)$, let $\Phi (K)=\{U\in \mathcal O(X) : K\subseteq U\}$. Then $\Phi (K)\in \mathrm{OFilt(\mathcal O(X))}$ and $K=\bigcap \Phi (K)$. Obviously, $\Phi : \mk (X) \longrightarrow \mathrm{OFilt(\mathcal O(X))}, K\mapsto \Phi (K)$, is an order embedding.

The single most important result about sober spaces is the Hofmann-Mislove Theorem (see \cite[Theorem 2.16]{Hofmann-Mislove} or \cite[Theorem II-1.20 and Theorem II-1.21]{redbook}).

\begin{theorem}\label{Hofmann-Mislove theorem} \emph{(Hofmann-Mislove Theorem)} For a $T_0$ space $X$, the following conditions are equivalent:
\begin{enumerate}[\rm (1)]
            \item $X$ is a sober space.
            \item  For any $\mathcal F\in \mathrm{OFilt}(\mathcal O(X))$, there is a $K\in \mk (X)$ such that $\mathcal F=\Phi (K)$.
            \item  For any $\mathcal F\in \mathrm{OFilt}(\mathcal O(X))$, $\mathcal F=\Phi (\bigcap \mathcal F)$.
\end{enumerate}
\end{theorem}

By the Hofmann-Mislove Theorem, $\Phi : \mk (X) \longrightarrow \mathrm{OFilt(\mathcal O(X))}$ is an order isomorphism if and only if $X$ is sober.

For locally compact well-filtered spaces, we have the following well-known result (see, e.g., \cite{redbook, Jean-2013, Kou}).

\begin{theorem}\label{SoberLC=CoreC}  For a $T_0$ space $X$, the following conditions are equivalent:
\begin{enumerate}[\rm (1)]
	\item $X$ locally compact and sober.
	\item $X$ is locally compact and well-filtered.
	\item $X$ is core compact and sober.
\end{enumerate}
\end{theorem}

For a topological space $X$, $\mathcal G\subseteq 2^{X}$ and $A\subseteq X$, let $\Diamond_{\mathcal G} A=\{G\in \mathcal G : G\bigcap A\neq\emptyset\}$ and $\Box_{\mathcal G} A=\{G\in \mathcal G : G\subseteq  A\}$. The symbols $\Diamond_{\mathcal G} A$ and $\Box_{\mathcal G} A$ will be simply written as $\Diamond A$  and $\Box A$ respectively if there is no confusion. The \emph{lower Vietoris topology} on $\mathcal{G}$ is the topology that has $\{\Diamond U : U\in \mathcal O(X)\}$ as a subbase, and the resulting space is denoted by $P_H(\mathcal{G})$. If $\mathcal{G}\subseteq \ir (X)$, then $\{\Diamond_{\mathcal{G}} U : U\in \mathcal O(X)\}$ is a topology on $\mathcal{G}$. The space $P_H(\Gamma(X)\setminus \{\emptyset\})$ is called the \emph{Hoare power space} or \emph{lower space} of $X$ and is denoted by $P_H(X)$ for short (cf. \cite{Schalk}). Clearly, $P_H(X)=(\Gamma(X)\setminus \{\emptyset\}, \upsilon(\Gamma(X)\setminus \{\emptyset\}))$. So $P_H(X)$ is always sober by Proposition \ref{uppertop H-Sober} below (or \cite[Corollary 4.10]{ZhaoHo}). The \emph{upper Vietoris topology} on $\mathcal{G}$ is the topology that has $\{\Box_{\mathcal{G}} U : U\in \mathcal O(X)\}$ as a base, and the resulting space is denoted by $P_S(\mathcal{G})$.

\begin{remark} \label{eta continuous} Let $X$ be a $T_0$ space.
\begin{enumerate}[\rm (1)]
	\item If $\mathcal{S}_c(X)\subseteq \mathcal{G}$, then the specialization order on $P_H(\mathcal{G})$ is the order of set inclusion, and the \emph{canonical mapping} $\eta_{X}: X\longrightarrow P_H(\mathcal{G})$, given by $\eta_X(x)=\overline {\{x\}}$, is an order and topological embedding (cf. \cite{redbook, Jean-2013, Schalk}).
    \item The space $X^s=P_H(\ir_c(X))$ with the canonical mapping $\eta_{X}: X\longrightarrow X^s$ is the \emph{sobrification} of $X$ (cf. \cite{redbook, Jean-2013}).
\end{enumerate}
\end{remark}

For a space $X$, a subset $A$ of $X$ is called \emph{saturated} if $A$ equals the intersection of all open sets containing it (equivalently, $A$ is an upper set in the specialization order). Let $\mathcal S^u(X)$ denote the set of all principal filters, namely, $\mathcal S^u(X)=\{\uparrow x : x\in X\}$. We shall use $\mathord{\mathsf{K}}(X)$ to
denote the set of all nonempty compact saturated subsets of $X$ and endow it with the \emph{Smyth preorder}, that is, for $K_1,K_2\in \mathord{\mathsf{K}}(X)$, $K_1\sqsubseteq K_2$ if{}f $K_2\subseteq K_1$. A space $X$ is called \emph{well-filtered} (resp., \emph{$\omega$-well-filtered}) if it is $T_0$, and for any open set $U$ and any filtered family (resp., any countable filtered family) $\mathcal{K}\subseteq \mathord{\mathsf{K}}(X)$, $\bigcap\mathcal{K}{\subseteq} U$ implies $K{\subseteq} U$ for some $K{\in}\mathcal{K}$. The category of all well-filtered spaces with continuous mappings is denoted by $\mathbf{Top}_w$.
The space $P_S(\mathord{\mathsf{K}}(X))$, denoted shortly by $P_S(X)$, is called the \emph{Smyth power space} or \emph{upper space} of $X$ (cf. \cite{Heckmann, Klause-Heckmann, Schalk}).

\begin{remark} \label{xi continuous} Let $X$ be a $T_0$ space. Then
\begin{enumerate}[\rm (1)]
	\item the specialization order on $P_S(X)$ is the Smyth order, that is, $\leq_{P_S(X)}=\sqsubseteq$.
    \item the \emph{canonical mapping} $\xi_X: X\longrightarrow P_S(X)$, $x\mapsto\ua x$, is an order and topological embedding (cf. \cite{Heckmann, Klause-Heckmann, Schalk}).
        \item $P_S(\mathcal S^u(X))$ is a subspace of $P_S(X)$ and $X$ is homeomorphic to $P_S(\mathcal S^u(X))$.
\end{enumerate}
\end{remark}

\begin{lemma}\label{Scott compact closed} For a poset $P$ and $A\in \Gamma (\Sigma~\!\!P)$, the following conditions are equivalent:
 \begin{enumerate}[\rm (1)]
 \item $\da (\ua x\cap A)\in \Gamma (\Sigma~\!\!P)$ for all $x\in P$.
 \item $\da (K\cap A)=\bigcup_{k\in K}\da (\ua k\cap A)\in \Gamma (\Sigma ~\!\!P)$ for all $K\in \mk (\Sigma ~\!\!P)$.
 \end{enumerate}
If $P$ is a sup semilattice, then condition \emph{(1)} \emph{(}and hence condition \emph{(2))} holds for any $x\in P$ and any $A\in \Gamma (\Sigma~\!\!P)$.
\end{lemma}
\begin{proof} (1) $\Rightarrow$ (2): Let $D\in \mathcal D(P)$ such that $D\subseteq \da (K\cap A)$ and $\bigvee D$ exists. If $\bigvee D\not\in  \da (K\cap A)$, then for each $k\in K$, $\bigvee D\not\in  \da (\ua k\cap A)$, and hence by $\da (\ua k\cap A)\in \Gamma (\Sigma~\!\!X)$, there is a $d_k\in D$ such that $d_k\not\in\da (\ua k\cap A)$, and consequently, $k\in P\setminus \da (\ua d_k \cap A)$ and $\da (\ua d_k \cap A)\in \Gamma (\Sigma~\!\!P)$ by condition (1). By the compactness of $K$ in $\Sigma ~\!\!P$, there exists a finite subset $\{d_{k_1}, ..., d_{k_n}\}\subseteq D$ such that $K\subseteq \bigcup_{i=1}^{n} (X\setminus \da (\ua d_{k_i} \cap A))$. Since $D$ is directed, there is a $d_0$ such that $\ua d_0\subseteq \bigcap_{i=1}^{n}\ua d_{k_i}$. It follows that $K\subseteq P\setminus \da (\ua d_o\cap A)$, which contradicts $d_o\in \da (K\cap A)$. Thus $\bigvee D\in  \da (K\cap A)$, proving $\da (K\cap A)\in \Gamma (\Sigma ~\!\!P)$.

(2) $\Rightarrow$ (1): Trivial.

Now we suppose that $P$ is a sup semilattice. Let $D\in \mathcal D(P)$ such that $D\subseteq \da (\ua x\cap A)$ and $\bigvee D$ exists. If $\bigvee D\not\in  \da (\ua x\cap A)$, then $\bigvee D\vee x=\bigvee\limits_{d\in D} d\vee x\not\in  A$, and hence by $A\in \Gamma (\Sigma~\!\!X)$, there is a $d_x\in D$ such that $d_x\vee x\not\in A$, and consequently, $d_x\not\in \da (\ua x\cap A)$, which contradicts $D\subseteq \da (\ua x\cap A)$. Therefore, $\da (\ua x\cap A)\in \Gamma (\Sigma ~\!\!P)$.
\end{proof}

\begin{corollary}\label{complete lattice Scott compact closed} \emph{(\cite{Xi-Lawson-2017})} Let $L$ be a complete lattice. Then $\da (K\cap A)\in \Gamma (\Sigma ~\!\!L)$ for any $K\in \mk (\Sigma ~\!\!L)$ and any $A\in \Gamma (\Sigma~\!\!L)$.
\end{corollary}

\begin{lemma}\label{Ps functor} $P_S : \mathbf{Top}_0 \longrightarrow \mathbf{Top}_0$ is a covariant functor, where for any $f : X \longrightarrow Y$ in $\mathbf{Top}_0$, $P_S(f) : P_S(X) \longrightarrow P_S(Y)$ is defined by $P_S(f)(K)=\ua f(K)$ for all $ K\in\mk(X)$.
\end{lemma}

\begin{proof} For any continuous mapping $f : X \longrightarrow Y$ to a $T_0$ space $Y$, one can easily deduce that $P_S(X)$ and $P_S(Y)$ are $T_0$.

	{Claim 1:} $P_S(f)\circ\xi_X=\xi_Y\circ f$.
	
	For each $ x\in X$, we have
	$$P_S(f)\circ\xi_X(x)=P_S(f)(\ua x)=\ua f(\ua x)=\ua f(x)=\xi_Y\circ f(x),$$
	that is, the following diagram commutes.
	\begin{equation*}
	\xymatrix{
		X \ar[d]_-{f} \ar[r]^-{\xi_X} &P_S(X)\ar[d]^-{P_S(f)}\\
		Y \ar[r]^-{\xi_Y} &P_S(Y)	}
	\end{equation*}

{Claim 2:} $P_S(f): P_S(X)\longrightarrow P_S(Y)$ is continuous.

For $V\in\mathcal O(Y)$, we have
$$\begin{array}{lll}
P_S(f)^{-1}(\Box V)&=&\{K\in\mk(X): P_S(f)(K)=\ua f(K)\subseteq V\}\\
&=&\{K\in\mk(X): K\subseteq f^{-1}(V)\}\\
&=&\Box f^{-1}(V),
\end{array}$$
which is open in $P_S(X)$. This implies that $P_S(f)$ is continuous.

{Claim 3:} $P_S(id_X)=id_{P_S(X)}$

For each $K\in\mk(X)$, $P_S(id_X)(K)=\ua id_X(K)=\ua K=K$.

{Claim 4:} For any continuous mapping $g : Y \longrightarrow Z$ in $\mathbf{Top}_0$, $P_S(g\circ f)=P_S(g)\circ P_S(f)$.

For any $ K\in\mk(X)$, $P_S(g\circ f)(K)=\ua g\circ f(K)=\ua g(f(K))=\ua g(\ua f(K))=P_S(g)\circ P_S(f)(K)$.

\noindent Thus $P_S : \mathbf{Top}_0 \longrightarrow \mathbf{Top}_0$ is a covariant functor.
\end{proof}

\begin{corollary}\label{Ps retract} Let $X$ and $Y$ be two $T_0$ spaces. If $Y$ is a retract of $X$, then $P_S(Y)$ is a retract of $P_S(X)$.
\end{corollary}

Similarly, we can show that $P_H : \mathbf{Top}_0 \longrightarrow \mathbf{Top}_0$ is a covariant functor (cf. Lemma \ref{Ph functor} below), where for any $f : X \longrightarrow Y$ in $\mathbf{Top}_0$, $P_H(f) : P_H(X) \longrightarrow P_H(Y)$ is defined by $P_H(f)(K)=\overline{f(A)}$ for all $ A\in\Gamma(X)\setminus \{\emptyset\}$.

 \begin{lemma}\label{X-Smyth-irr} Let $X$ be a $T_0$ space and $A\subseteq X$. Then the following four conditions are equivalent:
 \begin{enumerate}[\rm (1)]
	\item $A\in\ir (X)$.
	\item $\xi_X(A)\in \ir (P_S(X))$.
    \item $\Diamond_{\mk (X)}A=\da_{\mk (X)}\xi_X(A)\in \ir (P_S(X))$.
    \item $\xi_X(A)\in \ir (P_S(\mathcal S^u(X)))$.
\end{enumerate}

\noindent Moreover, the following three conditions are equivalent:
 \begin{enumerate}[\rm (a)]
	\item $A\in\ir_c (X)$.
	\item $\Diamond_{\mk (X)}A\in \ir_c(P_S(X))$.
    \item $\xi_X(A)\in \ir_c (P_S(\mathcal S^u(X)))$.
\end{enumerate}
\end{lemma}
\begin{proof} (1) $\Rightarrow$ (2): By Lemma \ref{irrimage}.

(2) $\Leftrightarrow$ (3): Trivial.

(2) $\Rightarrow$ (4): By Remark \ref{subspaceirr} and $P_S(\mathcal S^u(X)))$ is a subspace of $P_S(X)$.

(4) $\Rightarrow$ (1) and (a) $\Leftrightarrow$ (c): Since $x\mapsto \ua x : X \longrightarrow P_S(\mathcal S^u(X))$ is a homeomorphism.

(a) $\Leftrightarrow$ (b): By the equivalences of (1) and (3) (note that $\cl_{\mk (X)} \xi_X(B)=\Diamond_{\mk (X)}\cl_X B$ for any subset $B$ of $X$).

(a) $\Leftrightarrow$ (c):  By the homeomorphism $x\mapsto \ua x : X \longrightarrow P_S(\mathcal S^u(X))$.
\end{proof}

\begin{remark}\label{intersection=closure intersection in Smyth} Let $X$ be a $T_0$ space and $\mathcal A\subseteq \mk (X)$. Then $\bigcap \mathcal A=\bigcap \overline{\mathcal A}$, here the closure of $\mathcal A$ is taken in $P_S(X)$. Clearly, $\bigcap \overline{\mathcal A}\subseteq\bigcap \mathcal A$. On the other hand, for any $K\in \overline{\mathcal A}$ and $U\in \mathcal O(X)$ with $K\subseteq U$ (that is, $K\in \Box U$), we have $\mathcal A\bigcap\Box U\neq\emptyset$, and hence there is a $K_U\in \mathcal A\bigcap\Box U$. Therefore, $K=\bigcap \{U\in \mathcal O(X) : K\subseteq U\}\supseteq\bigcap \{K_U : U\in \mathcal O(X) \mbox{ and } K\subseteq U\}\supseteq\bigcap \mathcal A$. It follows that $\bigcap \overline{\mathcal A}\supseteq\bigcap \mathcal A$. Thus $\bigcap \mathcal A=\bigcap \overline{\mathcal A}$.
\end{remark}

\begin{lemma}\label{irr-induced-opne filter} Suppose that $X$ is a $T_0$ space and $\mathcal A\subseteq \mk (X)$. Then the following conditions are equivalent:
\begin{enumerate}[\rm (1)]
\item $\mathcal A\in \ir (P_S(X))$.
\item $\mathcal F_{\mathcal A}=\bigcup_{K\in \mathcal A} \Phi (K)\in \mathrm{OFilt}(\mathcal O(X))$.
\end{enumerate}
\end{lemma}
\begin{proof} (1) $\Rightarrow$ (2): Clearly, $\mathcal F_{\mathcal A}\in \sigma (\mathcal O(X))$ since $\Phi (K)\in \sigma (\mathcal O(X))$ for all $K\in \mk (X)$. Now we show that $\mathcal F_{\mathcal A}\in \mathrm{Filt}(\mathcal O(X))$. Suppose $U, V\in \mathcal F_{\mathcal A}$. Then $\mathcal A\bigcap\Box U\neq\emptyset$ and $\mathcal A\bigcap\Box V\neq\emptyset$, and hence $\mathcal A\bigcap \Box (U\cap V)=\mathcal A\bigcap \Box U\bigcap\Box V\neq\emptyset$ by $\mathcal A\in \ir (P_S(X))$. Therefore, $U\cap V\in \mathcal F_{\mathcal A}$.

(2) $\Rightarrow$ (1): Suppose that $\mathcal U, \mathcal V\in \mathcal O(P_S(X))$ for which $\mathcal A\cap \mathcal U\neq\emptyset \neq\mathcal A\cap\mathcal V$. Then there exist $\{U_i : i\in I\}\subseteq \mathcal O(X)$ and $\{V_j : j\in J\}\subseteq \mathcal O(X)$ with $\mathcal U=\bigcup_{i\in I}\Box U_i$ and $\mathcal V=\bigcup_{j\in J}\Box V_j$, and consequently, there are $K_1, K_2\in \mathcal A$ and $(i_0, j_0)\in I\times J$ such that $K_1\subseteq U_{i_0}$ and $K_2\subseteq U_{j_0}$. By condition (2), there is $K_3\in \mathcal A$ with $K_3\subseteq U_{i_0}\cap V_{j_0}$, that is, $K_3\in \Box U_{i_0}\cap \Box V_{j_0}\subseteq \mathcal U\cap \mathcal V$, and whence $\mathcal A\cap \mathcal U\cap\mathcal V\neq\emptyset$. Thus $\mathcal A\in \ir (P_S(X))$.
\end{proof}

\begin{remark}\label{K-union-closure} For a $T_0$ space $X$ and $\mathcal A\in \ir (P_S(X))$, $\mathcal F_{\mathcal A}=\mathcal F_{\mathrm{cl} \mathcal A}$. In fact, if $U\in \mathcal O(X)$ and $U\in \mathcal F_{\mathrm{cl} \mathcal A}$, then $\mathrm{cl}\mathcal A\bigcap \Box U\neq\emptyset$, and whence $\mathcal A\bigcap \Box U\neq\emptyset$. It follows $U\in \mathcal F_{\mathcal A}$.
\end{remark}

\begin{lemma}\label{sups in Smyth}\emph{(\cite{redbook})}  Let $X$ be a $T_0$ space. For a nonempty family $\{K_i : i\in I\}\subseteq \mk (X)$, $\bigvee_{i\in I} K_i$ exists in $\mk (X)$ if{}f~$\bigcap_{i\in I} K_i\in \mk (X)$. In this case $\bigvee_{i\in I} K_i=\bigcap_{i\in I} K_i$.
\end{lemma}

\begin{remark}\label{two meets in Smyth} For a nonempty family $\{K_i : i\in I\}\subseteq \mk (X)$, $\bigcap\limits_{i\in I} \ua_{\mk (X)}K_i\neq\emptyset$ in $\mk (X)$ if{}f $\bigcap\limits_{i\in I} K_i\neq\emptyset$. Furthermore, if $\bigcap\limits_{i\in I} K_i\neq\emptyset$, then $\bigcap\limits_{i\in I} \ua_{\mk (X)}K_i=\{K\in \mk (X) : K\subseteq \bigcap\limits_{i\in I} K_i\}$. In fact, if $K\in \mk (X)$ and $K\in \bigcap\limits_{i\in I} \ua_{\mk (X)}K_i$, then $\emptyset \neq K\subseteq \bigcap\limits_{i\in I} K_i$. Conversely, if $\bigcap\limits_{i\in I} K_i\neq\emptyset$, we can select an $x\in \bigcap\limits_{i\in I} K_i\neq\emptyset$. Then $\ua x\in \mk (X)$ and $\ua x\in \bigcap\limits_{i\in I} \ua_{\mk (X)}K_i$. In the case that $\bigcap\limits_{i\in I} K_i\neq\emptyset$, one can directly get that $\bigcap\limits_{i\in I} \ua_{\mk (X)}K_i=\{K\in \mk (X) : K\subseteq \bigcap\limits_{i\in I} K_i\}$.
\end{remark}

\begin{lemma}\label{K union} \emph{(\cite{jia-Jung-2016, Schalk})}  Let $X$ be a $T_0$ space. If $\mathcal K\in\mk(P_S(X))$, then  $\bigcup \mathcal K\in\mk(X)$.
\end{lemma}

\begin{corollary}\label{Smythunioncont} \emph{(\cite{jia-Jung-2016, Schalk})}  For any $T_0$ space $X$, the mapping $\bigcup : P_S(P_S(X)) \longrightarrow P_S(X)$, $\mathcal K\mapsto \bigcup \mathcal K$, is continuous.
\end{corollary}
\begin{proof} For $\mathcal K\in\mk(P_S(X))$, $\bigcup \mathcal K=\bigcup \mathcal K\in\mk(X)$ by Lemma \ref{K union}. For $U\in \mathcal O(X)$, we have $\bigcup^{-1}(\Box U)=\{\mathcal K\in \mk (P_S(X)) : \bigcup \mathcal K\in \Box U\}=\{\mathcal K\in \mk (P_S(X)) : \mathcal K\subseteq \Box U\}=\eta_{P_S(X)}^{-1}(\Box (\Box U))\in \mathcal O(P_S(P_S(X)))$. Thus  $\bigcup : P_S(P_S(X)) \longrightarrow P_S(X)$ is continuous.
\end{proof}

\begin{lemma}\label{many  meets in Smyth}  For a $T_0$ space $X$ and a nonempty family $\mathcal K\subseteq \mk (X)$, the following conditions are equivalent:
\begin{enumerate}[(1)]
\item For any $\mathcal U\in \mathcal O(P_S(X))$, $\bigcap\limits_{K\in \mathcal K}\ua_{\mk (X)}K\subseteq \mathcal U$ implies $\ua_{\mk (X)}K\subseteq \mathcal U$ for some $K\in \mathcal K$.

    \item For any $U\in \mathcal O(X)$, $\bigcap\limits_{K\in \mathcal K}\ua_{\mk (X)}K\subseteq \Box U$ implies $\ua_{\mk (X)}K\subseteq \Box U$ for some $K\in \mathcal K$.
        \item $\bigcap \mathcal K\in \mk (X)$, and for any $U\in \mathcal O(X)$, $\bigcap \mathcal K\subseteq U $ implies $ K\subseteq U$ for some $K\in \mathcal K$.
            \item For any $U\in \mathcal O(X)$, $\bigcap \mathcal K\subseteq U $ implies $ K\subseteq U$ for some $K\in \mathcal K$.
\end{enumerate}
\end{lemma}

\begin{proof} (1) $\Rightarrow$ (2): Trivial.

(2) $\Rightarrow$ (3): By condition (2), $\bigcap\limits_{K\in \mathcal K}\ua_{\mk (X)}K\neq\emptyset$ and $\bigcap\limits_{K\in \mathcal K}\ua_{\mk (X)}K\in \mk (P_S(X))$. By Remark \ref{two meets in Smyth} and Lemma \ref{K union}, $\bigcap \mathcal K=\bigcup \bigcap\limits_{K\in \mathcal K}\ua_{\mk (X)}K\in \mk (X)$. For any $U\in \mathcal O(X)$, if $\bigcap \mathcal K\subseteq U$, then $\bigcap\limits_{K\in \mathcal K}\ua_{\mk (X)}K=\ua_{\mk (X)}\bigcap \mathcal K\subseteq \Box U$, and hence by condition (2), there is some $K\in \mathcal K$ such that $\ua_{\mk (X)}K\subseteq \Box U$, that is, $K\subseteq U$.

(3) $\Leftrightarrow$ (4): We only need to show that under condition (4), $\bigcap \mathcal K\in \mk (X)$. First, $\bigcap \mathcal K\neq \emptyset$ (for otherwise $\bigcap \mathcal K=\emptyset$ implies $K=\emptyset$ for some $K\in \mathcal K$). If $\{V_j : j\in J\}\subseteq \mathcal O(X)$ is an open cover of $\bigcap \mathcal K$, then by condition (4), $K\subseteq \bigcup\limits_{j\in J}U_j$ for some $K\in \mathcal K$. By the compactness of $K$, there is a $J_0\in J^{(<\omega)}$ with $K\subseteq \bigcup\limits_{j\in J_0}U_j$, and whence $\bigcap \mathcal K\subseteq K\subseteq \bigcup\limits_{j\in J_0}U_j$. Thus $\bigcap \mathcal K\in \mk (X)$.

(3) $\Rightarrow$ (1): Suppose that $\mathcal U\in \mathcal O(P_S(X))$ and $\bigcap\limits_{K\in \mathcal K}\ua_{\mk (X)}K\subseteq \mathcal U$. Then there exists $\{U_i : i\in I\}\subseteq \mathcal O(X)$ with $\mathcal U=\bigcup\limits_{i\in I}\Box U_i$. Let $H=\bigcap \mathcal K$. Then by condition (3), $H\in \mk (X)$, and whence by Remark \ref{two meets in Smyth}, $\bigcap\limits_{K\in \mathcal K}\ua_{\mk (X)}K=\ua_{\mk (X)}H$. Therefore, $H\in \mathcal U=\bigcup\limits_{i\in I}\Box U_i$, and hence for some $i\in I$, $H\in \Box U_i$, i.e., $\bigcap \mathcal K=H\subseteq U_j$. By condition (3) again, $K\subseteq U_i$ for some $K\in \mathcal K$, and consequently, $\ua_{\mk (X)}K\subseteq \Box U_i\subseteq\mathcal U$.
\end{proof}

As in \cite{E_20182}, a topological space $X$ is \emph{locally hypercompact} if for each $x\in X$ and each open neighborhood $U$ of $x$, there is  $\ua F\in \mathbf{Fin}~X$ with $x\in\ii\,\ua F\subseteq\ua F\subseteq U$. A space $X$ is called \emph{core compact} if $\mathcal O(X)$ is a \emph{continuous lattice} (see \cite{redbook}).

For a nonempty subset $C$ of a $T_0$ space $X$, it is easy to see that $C$ is compact if{}f $\ua C\in \mk (X)$. Furthermore, we have the following useful result (see, e.g., \cite[pp.2068]{E_2009}).

\begin{lemma}\label{COMPminimalset} Let $X$ be a $T_0$ space and $C\in \mk (X)$. Then $C=\ua \mathrm{min}(C)$ and  $\mathrm{min}(C)$ is compact.
\end{lemma}

A subset system $\Omega$ of a set $X$ is called an \emph{open system} on $X$ provided that (1) $\{\emptyset, X\}\subseteq \Omega$, and (2) $\bigcup_{i\in I}U_i\in \Omega$ for any $\{U_i : i\in I\}\subseteq \Omega$. The sets in $\Gamma(\Omega)=\{X\setminus U : U\in \Omega\}$ are called \emph{closed}. As usual, for $A\subseteq X$, define $\ii_{\Omega} A=\bigcup\{U\in \Omega : U\subseteq A\}$ and $\mathrm{cl}_{\Omega} A=\bigcap\{C\in \Gamma(\Omega) : A\subseteq C\}$. Suppose that $\Omega$ and $\Psi$ are open systems on sets $X$and $Y$ respectively. A mapping $f : (X, \Omega) \longrightarrow (Y, \Psi)$ is called \emph{continuous}, if $f^{-1}(V)\in \Omega$ for any $V\in \Psi$ or, equivalently, if $f^{-1}(C)$ is closed for any closed set $C$ of $(Y, \Psi)$.

\section{Irreducible subset systems and topological Rudin's lemma}
In order to provide a uniform approach to $d$-spaces, sober spaces and well-filtered spaces and develop a general framework for dealing with all these spaces, inspired by the work of Wright, Wagner and Thatcher \cite{Wright78} on inductive posets and inductive closures, we introduce the following two concepts.

\begin{definition}\label{subset system} A covariant functor ${\rm H} : \mathbf{Top}_0 \longrightarrow \mathbf{Set}$ is called a \emph{subset system} on $\mathbf{Top}_0$ provided that the following two conditions are satisfied:
\begin{enumerate}[\rm (1)]
\item $\mathcal S(X)\subseteq {\rm H}(X)\subseteq 2^X$ (the set of all subsets of $X$) for each $X\in$ \emph{ob}($\mathbf{Top}_0$).
\item For any continuous mapping $f : X \longrightarrow Y$ in $\mathbf{Top}_0$, ${\rm H}(f)(A)=f(A)\in {\rm H}(Y)$ for all $A\in H(X)$.
\end{enumerate}
\end{definition}

For a subset system ${\rm H} : \mathbf{Top}_0 \longrightarrow \mathbf{Set}$ and a $T_0$ space $X$, let ${\rm H}_c(X)=\{\overline{A} : A\in {\rm H}(X)\}$. We call $A\subseteq X$ an \emph{H}-\emph{set} if $A\in {\rm H}(X)$. The sets in ${\rm H}_c(X)$ are called \emph{closed H}-\emph{sets}.

\begin{definition}\label{R-subset system} A subset system ${\rm H} : \mathbf{Top}_0 \longrightarrow \mathbf{Set}$ is called an \emph{irreducible subset system}, or an \emph{R-subset system} for short, if ${\rm H}(X)\subseteq \ir (X)$ for all $X\in$ \emph{ob}($\mathbf{Top}_0$). The set of all R-subset systems is denoted by $\mathcal H$. Define a partial order $\leq$ on $\mathcal H$ by ${\rm H}_1\leq {\rm H}_2$ if{}f ${\rm H}_1(X)\subseteq {\rm H}_2(X)$ for all $X\in$ \emph{ob}($\mathbf{Top}_0)$.
\end{definition}

In what follows the capital letter {\rm H} always stands for an R-subset system ${\rm H} : \mathbf{Top}_0 \longrightarrow \mathbf{Set}$. Here are some important
examples of $R$-subset systems:
\begin{enumerate}[\rm (1)]
    \item $\mathcal S$ ($\mathcal S(X)$ is the set of all single point subsets of $X$).
    \item $\mathcal C$ ($\mathcal C(X)$ is the set of all chains of $X$).
    \item $\mathcal C^\omega$ ($\mathcal C^\omega(X)$ is the set of all countable chains of $X$).
    \item $\mathcal D$ ($\mathcal D(X)$ is the set of all directed subsets of $X$).
    \item $\mathcal D^\omega$ ($\mathcal D^\omega(X)$ is the set of all countable directed subsets of $X$).
    \item $\mathcal R$ ($\mathcal R(X)$ is the set of all irreducible subsets of $X$).
    \item $\mathcal R^\omega$ ($\mathcal R^\omega(X)$ is the set of all countable irreducible subsets of $X$).
\end{enumerate}

\begin{remark}\label{R-SS relatons}  $\mathcal S\leq  \mathcal C^\omega\leq\mathcal C\leq \mathcal D\leq \mathcal R$ and $\mathcal C^\omega\leq \mathcal D^\omega \leq \mathcal R^\omega$.
\end{remark}

By Lemma \ref{Ps functor}, we get the following result.

\begin{proposition}\label{subspace H-set} Let $X, Y, Z$ be $T_0$ spaces and $f : X \longrightarrow Y$ a continuous mapping.
\begin{enumerate}[\rm (1)]
\item If $Z$ is a nonempty subspace of $X$, then $H(Z)\subseteq H(X)$.
\item For any $\mathcal K \in H(P_S(X))$, $\{\ua f(K) : K\in \mathcal K \}\in H(P_S(Y))$.
\end{enumerate}
\end{proposition}

Rudin's Lemma plays a crucial role in domain theory and is a useful tool in studying the various aspects of well-filtered spaces and sober spaces (see [6-11, 23, 24, 30-34]). In \cite{Rudin}, Rudin proved her lemma by transfinite methods, using the Axiom of Choice.
In \cite{Klause-Heckmann}, Heckman and Keimel presented the following topological variant of Rudin's Lemma.

\begin{lemma}\label{t Rudin} \emph{(Topological Rudin's Lemma)} \emph{(\cite{Klause-Heckmann})} Let $X$ be a topological space and $\mathcal{A}$ an
irreducible subset of the Smyth power space $P_S(X)$. Then every closed set $C {\subseteq} X$  that
meets all members of $\mathcal{A}$ contains a minimal irreducible closed subset $A$ that still meets all
members of $\mathcal{A}$.
\end{lemma}

\begin{corollary}\label{H-rudin}  Let $H : \mathbf{Top}_0 \longrightarrow \mathbf{Set}$ be an R-subset system, $X$ a $T_0$ space and $\mathcal{A}\in {\rm H}(P_S(X))$. Then every closed set $C {\subseteq} X$  that
meets all members of $\mathcal{A}$ contains a minimal irreducible closed subset $A$ that still meets all
members of $\mathcal{A}$.
\end{corollary}

Applying Lemma \ref{t Rudin} to the Alexandroff topology on a poset $P$, one obtains  the original Rudin's Lemma.

\begin{corollary}\label{rudin} \emph{(Rudin's Lemma)} Let $P$ be a poset, $C$ a nonempty lower subset of $P$ and $\mathcal F\in \mathbf{Fin}~P$ a filtered family with $\mathcal F\subseteq\Diamond C$. Then there exists a directed subset $D$ of $C$ such that $\mathcal F\subseteq \Diamond\da D$.
\end{corollary}

For a $T_0$ space $X$ and $\mathcal{K}\subseteq \mathord{\mathsf{K}}(X)$, let $M(\mathcal{K})=\{A\in \Gamma(X) : K\bigcap A\neq\emptyset \mbox{~for all~} K\in \mathcal{K}\}$ (that is, $\mathcal A\subseteq \Diamond A$) and $m(\mathcal{K})=\{A\in \Gamma(X) : A \mbox{~is a minimal menber of~} M(\mathcal{K})\}$.

By the proof of \cite[Lemma 3.1]{Klause-Heckmann}, we have the following result.

\begin{lemma}\label{t ruding}  Let $X$ be a $T_0$ space and $\mathcal{K}\subseteq \mathord{\mathsf{K}}(X)$. If $C\in M(\mathcal{K})$, then there is a closed subset $A$ of $C$ such that $A\in m(\mathcal{K})$.
\end{lemma}

\begin{definition}\label{R-SS property M} A R-subset system ${\rm H} : \mathbf{Top}_0 \longrightarrow \mathbf{Set}$ is said to satisfy  \emph{property M} if for any $T_0$ space $X$, $\mathcal K\in {\rm H}(P_S(X))$ and $A\in M(\mathcal{K})$, we have $\{\uparrow(K\bigcap A) : K\in \mathcal K\}\in {\rm H}(P_S(X))$.
\end{definition}

\begin{remark}\label{C D property M} The R-subset systems $\mathcal S$, $\mathcal C^\omega$, $\mathcal C$, $\mathcal D^\omega$ and $\mathcal D$ satisfy property M.
\end{remark}

\begin{lemma}\label{R has property M}  The R-subset system $\mathcal R$ satisfies property M.
\end{lemma}

\begin{proof} Suppose that $X$ is a $T_0$ space, $\mathcal K\in \ir (P_S(X))$ and $A\in M(\mathcal{K})$. For any $U, V\in \mathcal O(X)$, if $\{\uparrow(K\bigcap A) : K\in \mathcal K\}\bigcap \Box U \neq \emptyset \neq \{\uparrow(K\bigcap A) : K\in \mathcal K\}\bigcap \Box V$, then there exist $K_1, K_2\in \mathcal K$ such that $\uparrow (K_1\bigcap A)\subseteq U$ and $\uparrow (K_1\bigcap A)\subseteq V$ or, equivalently, $K_1\in \Box ((X\setminus A)\bigcup U)$ and $K_1\in \Box ((X\setminus A)\bigcup V)$. By $\mathcal K\in \ir (P_S(X))$, $\mathcal K\bigcap  \Box ((X\setminus A)\bigcup (U\bigcap V))=\mathcal K\bigcap  \Box ((X\setminus A)\bigcup V)\bigcap \Box ((X\setminus A)\bigcup V)\neq \emptyset$, and hence there is a $K_3\in \mathcal K$ such that $K_3\in \Box ((X\setminus A)\bigcup (U\bigcap V))$. Therefore, $\uparrow (K_3\bigcap A)\subseteq U\bigcap V$, i.e., $\uparrow (K_3\bigcap A)\in \Box U\bigcap \Box V$, and consequently, $\{\uparrow(K\bigcap A) : K\in \mathcal K\}\bigcap \Box U\bigcap \Box V\neq \emptyset$. Thus $\{\uparrow(K\bigcap A) : K\in \mathcal K\}\in \ir (P_S(X))$.
 \end{proof}

\begin{lemma}\label{R property M charact} For an R-subset system ${\rm H} : \mathbf{Top}_0 \longrightarrow \mathbf{Set}$, the following two conditions are equivalent:
\begin{enumerate}[\rm (1)]
\item H satisfies property M.
\item For any continuous mapping $f:X\longrightarrow Y$ in $\mathbf{Top}_0$, $\mathcal K\in H(P_S(X))$ and $A\in M(\mathcal{K})$, $\{\uparrow f(K\bigcap A) : K\in \mathcal K\}\in H(P_S(Y))$.
    \end{enumerate}
\end{lemma}
\begin{proof}  (1) $\Rightarrow$ (2): Since ${\rm H}$ satisfies property M, $\{\uparrow(K\bigcap A) : K\in \mathcal K\}\in {\rm H}(P_S(X))$, and hence by Lemma \ref{Ps functor}, $\{\ua f(K\bigcap A) : K\in \mathcal K\}=\{\ua f(\uparrow(K\bigcap A)) : K\in \mathcal K\}={\rm H}(f)(\{\uparrow(K\bigcap A) : K\in \mathcal K\})\in {\rm H}(P_S(Y))$.

(2) $\Rightarrow$ (1): Applying condition (2) to the identity $id_X : X \longrightarrow X$.

\end{proof}

\section{H-sober spaces}

\begin{definition}\label{def H-sober} Let ${\rm H} : \mathbf{Top}_0 \longrightarrow \mathbf{Set}$ be an $R$-subset system. A $T_0$ space $X$ is called \emph{H}-\emph{sober} if for any $A\in{\rm H}(X)$, there is a (unique) point $x\in X$ such that $\overline{A}=\overline{\{x\}}$ or, equivalently, if ${\rm H}_c(X)=\mathcal S_c(X)$. The category of all {\rm H}-sober spaces with continuous mappings is denoted by $\mathbf{H}$-$\mathbf{Sob}$.
\end{definition}

The $\mathcal D$-sober spaces and the $\mathcal R$-sober spaces are exactly $d$-spaces and sober spaces respectively. Therefore, $\mathcal D$-$\mathbf{Sob}=\mathbf{Top}_d$ and $\mathcal R$-$\mathbf{Sob}=\mathbf{Sob}$.

\begin{definition}\label{H-complete} Let ${\rm H} : \mathbf{Top}_0 \longrightarrow \mathbf{Set}$. A $T_0$ space $X$ is called \emph{H}-\emph{complete}, if for any $A\in {\rm H}(X)$, $\bigvee A$ exists in $X$, or equivalently, for any $B\in {\rm H}_c(X)$, $\bigvee B$ exists in $X$ (see Remark \ref{C up=cl C up}). The $\mathcal R$-complete spaces are also called \emph{irreducible complete} (cf. \cite{xu-shen-xi-zhao1}).
\end{definition}

\begin{definition}\label{H-bound} Let ${\rm H} : \mathbf{Top}_0 \longrightarrow \mathbf{Set}$. A $T_0$ space $X$ is called \emph{H}-\emph{bounded}, if for any $A\in {\rm H}(X)$, $A^{\uparrow}\neq\emptyset$, that is, $A$ has an upper bound in $X$.
\end{definition}

\begin{remark}\label{H-S H-C H-B} We have the following implications (which can not be reversed):
\begin{center}
{\rm H}-sobriety $\Rightarrow$ {\rm H}-completeness $\Rightarrow$ {\rm H}-boundedness.
\end{center}

In fact, if $X$ is an {\rm H}-sober space and $A\in {\rm H}(X)$, then there is an $x\in X$ such that $\overline{A}=\overline{\{x\}}$, and hence $\bigvee A=\bigvee\overline{A}=\bigvee \overline{\{x\}}=x$. So {\rm H}-sobriety $\Rightarrow$ {\rm H}-completeness. Clearly, {\rm H}-completeness $\Rightarrow$ {\rm H}-boundedness.
\end{remark}

Let $L$ be the complete lattice constructed by Isbell in \cite{isbell}. Then $\Sigma L$ is $\mathcal R$-complete, but is not $\mathcal R$-sober (i.e., is non-sober). For a poset $P$ with a largest element $\top$, any \emph{order compatible} topology $\tau$ on $P$ (that is, $\leq_{\tau}$ agrees with the original order on $P$) is ${\rm H}$-bounded, but $(P, \tau)$ may not be ${\rm H}$-complete. For example, let $P$ be any non-dcpo with a largest element $\top$, then $\Sigma~\!\! P$ is ${\rm H}$-bounded for any R-subset system ${\rm H}$, but $\Sigma~\!\! P$ is not $\mathcal D$-complete.

By Corollary \ref{dcpo=chain comp} and Lemma \ref{D countable chain}, we get the following result.

\begin{proposition}\label{D complete or bounded=Cs} Let $P$
be a poset. Then
\begin{enumerate}[\rm (1)]
\item $P$ is $\mathcal D$-complete if{}f $P$ is $\mathcal C$-complete.
\item $P$ is $\mathcal D^\omega$-bounded if{}f $P$ is $\mathcal C^\omega$-bounded.
\item $P$ is $D^\omega$-complete if{}f $P$ is $\mathcal C^\omega$-complete.
\end{enumerate}
\end{proposition}

Clearly, $\mathcal D$-boundedness $\Rightarrow$ $\mathcal C$-boundedness. But we do not know whether the converse holds.

\begin{proposition}\label{D-sober=C-sober} For a $T_0$ space $X$, the following conditions are equivalent:
\begin{enumerate}[\rm (1)]
\item $X$ is $\mathcal D$-sober \emph{(}i.e., a $d$-space\emph{)}.
\item $X$ is $\mathcal C$-sober.
\end{enumerate}
\end{proposition}

\begin{proof} (1) $\Rightarrow$ (2): Trivial.

(2) $\Rightarrow$ (1): We prove it by transfinite induction on the cardinalities of directed subsets of $X$. Suppose the conclusion is false. Let $D$ be a directed subset of $X$ such that: 1) $|D|\leq \gamma$, 2) $\overline{D} \notin \mathcal S_c(X)$, and 3) for all directed sets $E\subset X$ with $|E|<|D|$, $\overline{E} \in \mathcal S_c(X)$. Obviously, $D$ cannot be finite. Let $D=\bigcup_{\alpha<|D|}D_\alpha$, where the $D_\alpha$ are as in Lemma \ref{Markowsky}. For each $\alpha<|D|$, by 3), there is a unique $x_\alpha\in X$ such that $\overline{D_\alpha}=\overline{\{x_\alpha\}}$. For $\alpha<\beta<|D|$, by condition (2) in Lemma \ref{Markowsky}, $\overline{\{x_\alpha\}}=\overline{D_\alpha}\subseteq \overline{D_\beta}=\overline {\{x_\beta\}}$. Therefore, $\{x_\alpha : \alpha<|D|\}\in \mathcal C(X)$. Since $X$ is $\mathcal C$-sober, there is an $x\in X$ such that
$\overline{\{x_\alpha : \alpha<|D|\}}=\overline{\{x\}}$. It follows that
$$\begin{array}{lll}
\overline{D}&=&\overline{\bigcup_{\alpha<|D|}D_\alpha}\\
&=&\overline{\bigcup_{\alpha<|D|}D_\alpha}\\
&=&\overline{\bigcup_{\alpha<|D|}\overline{D_\alpha}}\\
&=&\overline{\bigcup_{\alpha<|D|}\overline{\{x_\alpha\}}}\\
&=&\overline{\bigcup_{\alpha<|D|}\{x_\alpha\}}\\
&=&\overline{\{x\}}.
\end{array}$$
This contradiction proves the implication of (2) $\Rightarrow$ (1).
\end{proof}

Similarly, we have the following result.

\begin{proposition}\label{D omega-sober=C omega-sober} For a $T_0$ space $X$, the following conditions are equivalent:
\begin{enumerate}[\rm (1)]
\item $X$ is $\mathcal D^\omega$-sober.
\item $X$ is $\mathcal C^\omega$-sober.
\end{enumerate}
\end{proposition}

\begin{proof} (1) $\Rightarrow$ (2): Trivial.

(2) $\Rightarrow$ (1): For any $D\in \mathcal D^\omega (X)$, by Lemma \ref{D countable chain}, there exists a countable chain $C\subseteq D$ such that $\da D=\da C$. By (2), there exists an $x\in X$ such that $\overline{C}=\overline{\{x\}}$, and hence $\overline{D}=\overline{\da D}=\overline{\da C}=\overline{C}=\overline{\{x\}}$. Thus $X$ is $\mathcal D^\omega$-sober.
\end{proof}

\begin{definition}\label{Scott H-open system} Let $X$ and $Y$ be two $T_0$ spaces.
\begin{enumerate}[\rm (1)]
\item A subset $U$ of $X$ is called \emph{Scott H}-\emph{open} if
(i) $U=\mathord{\uparrow}U$, and (ii) for any $A\in {\rm H}(X)$ for
which $\bigvee A$ exists, $\bigvee A\in U$ implies $A\cap
U\neq\emptyset$. The collection of all Scott H-open subsets of $X$ is called the \emph{Scott H-open system} on $X$ and is denoted by $\sigma_H(X)$. The sets in $\{X\setminus U : U\in \sigma_H(X)\}$ are called \emph{Scott H-closed}.
\item A mapping $f : X \longrightarrow Y$ is called \emph{Scott H-continuous} if $f :  (X, \sigma_H(X)) \longrightarrow (Y, \sigma_H(X))$ is continuous.
\end{enumerate}
\end{definition}

Clearly, $\upsilon(X)\subseteq \sigma_H(X)$ and $\sigma_{\mathcal D}(X)$ is the usual Scott topology $\sigma (X)$. For the R-subset system $\mathcal R$, $\sigma_{\mathcal R}(X)$ is only an open system but not a topology in general.

\begin{remark}\label{Scott H-closed set}  $C\subseteq X$ is Scott H-closed if{}f (i) $C=\da C$, and (ii) for any $A\in {\rm H}(X)$ for
which $\bigvee A$ exists, $A\subseteq C$ implies $\bigvee A\in C$.
\end{remark}

\begin{proposition}\label{H-sober Scott H-open system} For a $T_0$ space $X$, the following two conditions are equivalent:
\begin{enumerate}[\rm (1)]
\item $X$ is H-sober.
\item $X$ is H-complete, and $\mathcal O(X)\subseteq \sigma_H(X)$.
\end{enumerate}
\end{proposition}
\begin{proof} (1) $\Rightarrow$ (2): By Remark \ref{H-S H-C H-B}, $X$ is {\rm H}-complete. Suppose that $C\in \Gamma (X)$ and $A\in {\rm H}(X)$ for which $\bigvee A$ exists. Then there is $x\in X$ with $\cl_X A=\cl_X\{x\}$, and hence $\bigvee A=x$. If $A \subseteq C$, then $\bigvee A=x\in \cl_X\{x\}=\cl_X A\subseteq C$. Thus $C$ is H-closed in $(X, \sigma_H(X))$.

(2) $\Rightarrow$ (1): Let $A \in {\rm H}(X)$. Then by condition (2), $\bigvee A$ exists and $\cl_X A$ is Scott H-closed, and hence by $A\subseteq \cl_X A$, we have $\bigvee A\in \cl_X A\subseteq \da \bigvee A$. It follows that $\cl_X A=\da \bigvee A=\cl_X \{\bigvee A\}$. Thus $X$ is {\rm H}-sober
\end{proof}

\begin{lemma}\label{Scott H-cont} Let $X$ and $Y$ be two $T_0$ spaces. For a continuous mapping $f: X \longrightarrow Y$, the following two conditions are equivalent:
\begin{enumerate}[\rm (1)]
\item $f$ is Scott H-continuous.
\item For any $A\in H(X)$ for
which $\bigvee A$ exists in $X$, $f(\bigvee A)=\bigvee f(A)$.
\end{enumerate}
\end{lemma}
\begin{proof} (1) $\Rightarrow$ (2): Since $f: X \longrightarrow Y$ is continuous, $f$ is order-preserving and $f(A)\in {\rm H}(Y)$, and hence $f(\bigvee A)$ is an upper bound of $f(A)$. If $y\in Y$ is an upper bound of $f(A)$, that is, $f(A)\subseteq \da y$, then $A\subseteq f^{-1}(\da y)$. Since $f$ is Scott H-continuous and $\da y$ is H-closed in $(Y, \sigma_H(Y))$, $f^{-1}(\da y)$ is H-closed in $(X, \sigma_H(X))$, and consequently, $\bigvee A\in f^{-1}(\da y)$. Therefore, $f(\bigvee A)\leq y$. Thus $f(\bigvee A)=\bigvee f(A)$.

(2) $\Rightarrow$ (1): Suppose $B$ is H-closed in $(Y, \sigma_H(Y))$. Then $f^{-1}(B)$ is a lower subset of $X$ since $f$ is order-preserving and $B=\da B$. For any $A\in {\rm H}(X)$ for which $\bigvee A$ exists in $X$, if $A \subseteq f^{-1}(B)$, then $f(A)\in {\rm H}(Y)$ and $f(A)\subseteq B$. By condition (2),  $f(\bigvee A)=\bigvee f(A)\in B$, and whence $\bigvee A\in  f^{-1}(B)$. Thus $f^{-1}(B)$ is H-closed in $(X, \sigma_H(X))$, proving that $f$ is Scott H-continuous.
\end{proof}

\begin{proposition}\label{uppertop H-Sober} Let $H : \mathbf{Top}_0 \longrightarrow \mathbf{Set}$ be an R-subset system and $P$ a poset. Then the space $(P,\upsilon (P))$ is H-sober if and only if it is H-complete.
\end{proposition}
\begin{proof} If the upper topology $\upsilon (P)$ is {\rm H}-sober, then $(P, \upsilon (P))$ is {\rm H}-complete by Remark \ref{H-S H-C H-B}. Conversely, if $(P, \upsilon (P))$ is {\rm H}-complete, we show that $\upsilon (P)$ is {\rm H}-sober. For $A\in$ H$((P, \upsilon (P)))$, we have $A\in\ir ((P, \upsilon (P)))$. If $\cl_{\upsilon (P)} A=P$, then by {\rm H}-completeness of $(P, \upsilon (P))$ and Remark \ref{C up=cl C up}, $P$ has a largest element $\top$. So $P=\downarrow \top=\overline{\{\top\}}$. If $\cl_{\upsilon (P)}A\neq P$, then there is a nonempty family $\{\da F_i : i\in I\}\subseteq \mathbf{Fin} (P)$ such that $\cl_{\upsilon (P)} A=\bigcap_{i\in I} \da F_i$. For each $i\in I$, $\cl_{\upsilon (P)} A\subseteq \da F_i$, and hence by the irreducibility of $A$, $\cl_{\upsilon (P)}A\subseteq \da x_i$ for some $x_i\in F_i$. Therefore, $\cl_{\upsilon (P)} A=\bigcap_{i\in I} \da x_i\supseteq\bigcap \{\da x : A\subseteq \da x\}=A^{\delta}\supseteq \cl_{\upsilon (P)} {A}$. Since $(P, \upsilon (P))$ is {\rm H}-complete, $\bigvee A$ exists in $P$, and consequently, $\cl_{\upsilon (P)}A=A^{\delta}=\da \bigvee A=\overline{\{\bigvee A\}}$. Thus $\upsilon (P)$ is {\rm H}-sober.
\end{proof}

\begin{corollary}\label{uppertop Sober} For a poset $P$, $(P,\upsilon (P))$ is sober iff it is irreducible complete.
\end{corollary}

\begin{proposition}\label{H-sober charact-H-sets} Let $H : \mathbf{Top}_0 \longrightarrow \mathbf{Set}$ be an R-subset system and $X$ a $T_0$ space. Then the following conditions are equivalent:
\begin{enumerate}[\rm (1)]
	        \item $X$ is H-sober.

            \item  For any $A\in  H(X)$, $\overline{A}\cap\bigcap\limits_{a\in A}\ua a\neq\emptyset$.
            \item  For any $A\in H_c(X)$, $A\cap\bigcap\limits_{a\in A}\ua a\neq\emptyset$.

            \item  For any $A\in H(X)$ and $U\in \mathcal O(X)$, $\bigcap\limits_{a\in A}\ua a\subseteq U$ implies $\ua a \subseteq U$ \emph{(}i.e., $a\in U$\emph{)} for some $a\in A$.
            \item  For any $A\in H_c(X)$ and $U\in \mathcal O(X)$, $\bigcap\limits_{a\in A}\ua a\subseteq U$ implies $\ua a \subseteq U$ \emph{(}i.e., $a\in U$\emph{)} for some $a\in A$.

\end{enumerate}
\end{proposition}
\begin{proof} (1) $\Rightarrow$ (2): If $X$ is {\rm H}-sober and $A\in \mathcal {\rm H}(X)$, then there is an $x\in X$ such that $\overline{A}=\overline{\{x\}}=\da x$, and whence $x\in \overline{A}\cap \bigcap\limits_{a\in A}\ua a$.

(2) $\Leftrightarrow$ (3): Clearly, we have (2) $\Rightarrow$ (3). Conversely, if condition (3) is satisfied, then for $A\in {\rm H}(X)$, $\overline{A}\in {\rm H}_c(X)$, and $\emptyset\neq\overline{A}\cap \bigcap\limits_{b\in \overline{A}}\ua b\subseteq \overline{A}\cap \bigcap\limits_{a\in A}\ua a$.

(2) $\Rightarrow$ (4): If $\ua a\not\subseteq U$ for all $a\in A$, then $A\subseteq X\setminus U$, and hence $\overline{A}\subseteq X\setminus U$. By condition (2), $\emptyset\neq \overline{A}\cap \bigcap\limits_{a\in A}\ua a\subseteq (X\setminus U)\cap U=\emptyset$, a contradiction.

(4) $\Leftrightarrow$ (5): Obviously, (4) $\Rightarrow$ (5). Conversely, if condition (5) holds, then for $A\in {\rm H}(X)$ and $U\in \mathcal O(X)$ with $\bigcap\limits_{a\in A}\ua a\subseteq U$, we have $\overline{A}\in {\rm H}_c (X)$ and $\bigcap\limits_{b\in \overline{A}}\ua b=\bigcap\limits_{a\in A}\ua a\subseteq U$ by Remark \ref{C up=cl C up}. By condition (5), $b\in U$ for some $b\in \overline{A}$, and whence $A\cap U\neq\emptyset$. Condition (4) is thus satisfied.

(5) $\Rightarrow$ (1): Suppose $A\in {\rm H}_c(X)$. Then $A\cap \bigcap\limits_{a\in A}\ua a\neq\emptyset$ (otherwise, by condition (5), $A\cap \bigcap\limits_{a\in A}\ua a=\emptyset \Rightarrow \bigcap\limits_{a\in A}\ua a\subseteq X\setminus A \Rightarrow \ua a\subseteq X\setminus A$ for some $a\in A$, a contradiction). Select an $x\in A\cap \bigcap\limits_{a\in A}\ua a$. Then $A\subseteq \da x=\overline{\{x\}}\subseteq A$, and hence $A=\overline{\{x\}}$. Thus $X$ is {\rm H}-sober.

\end{proof}

Now we give some equational characterizations of H-sober spaces.

\begin{proposition}\label{H-sober charact-bounded} Let $H : \mathbf{Top}_0 \longrightarrow \mathbf{Set}$ be an R-subset system and $X$ a $T_0$ space. Then the following conditions are equivalent:
\begin{enumerate}[\rm (1)]
	        \item $X$ is H-sober.
            \item  $X$ is H-bounded \emph{(}especially, $X$ is H-complete\emph{)}, and $\ua \left(C\cap\bigcap\limits_{a\in A} \ua a\right)=\bigcap\limits_{a\in A}\ua (C\cap \ua a)$ for any $A\in H(X)$ and $C\in \Gamma(X)$.
             \item  $X$ is H-bounded \emph{(}especially, $X$ is H-complete\emph{)}, and $\ua \left(C\cap\bigcap\limits_{a\in A} \ua a\right)=\bigcap\limits_{a\in A}\ua (C\cap \ua a)$ for any $A\in H_c(X)$ and $C\in \Gamma(X)$.
            \item  $X$ is H-bounded \emph{(}especially, $X$ is H-complete\emph{)}, and $\ua \left(C\cap\bigcap\limits_{a\in A} \ua a\right)=\bigcap\limits_{a\in A}\ua (C\cap \ua a)$ for any $A\in H(X)$ and $C\in H_c (X)$.
            \item  $X$ is H-bounded \emph{(}especially, $X$ is H-complete\emph{)}, and $\ua \left(C\cap\bigcap\limits_{a\in A} \ua a\right)=\bigcap\limits_{a\in A}\ua (C\cap \ua a)$ for any $A\in H_c(X)$ and $C\in H_c (X)$.
\end{enumerate}
\end{proposition}
\begin{proof}  (1) $\Rightarrow$ (2): Since $X$ is {\rm H}-sober, $X$ is ${\rm H}$-complete by Remark \ref{H-S H-C H-B}. For $A\in {\rm H} (X)$ and $C\in \Gamma(X)$, clearly, $\ua \left(C\cap\bigcap\limits_{a\in A} \ua a\right)\subseteq\bigcap\limits_{a\in A}\ua (C\cap \ua a)$. Conversely, if $x\not\in \ua \left(C\cap\bigcap\limits_{a\in A}\ua a\right)$, that is, $\da x\cap C\cap\bigcap\limits_{a\in A}\ua a=\emptyset$, then $\bigcap\limits_{a\in A}\ua a\subseteq X\setminus \da x\cap C$, and whence by Proposition \ref{H-sober charact-H-sets}, $a\in X\setminus \da x\cap C$ for some $a\in A$, i.e., $\ua a\bigcap\da x\bigcap C=\emptyset$, and hence $x\not\in \ua (C\cap\ua a)$. Therefore, $x\not\in \bigcap\limits_{a\in A}\ua (C\cap \ua a)$. Thus $\ua \left(C\cap\bigcap\limits_{a\in A} \ua d\right)=\bigcap\limits_{a\in A}\ua (C\cap \ua a)$.

(2) $\Rightarrow$ (3): Since $A\in {\rm H}_c(X)$, there is $B\in {\rm H}(X)$ with $A=\overline{B}$. By condition (2) and Remark \ref{C up=cl C up}, $\bigcap\limits_{a\in A}\ua (C\cap \ua a)\supseteq\ua \left(C\cap\bigcap\limits_{a\in A} \ua a\right)=\ua \left(C\cap\bigcap\limits_{b\in B} \ua b\right)=\bigcap\limits_{b\in B}\ua (C\cap \ua b)\supseteq \bigcap\limits_{a\in A}\ua (C\cap \ua a)$, and hence $\ua \left(C\cap\bigcap\limits_{a\in A} \ua a\right)=\bigcap\limits_{a\in A}\ua (C\cap \ua a)$.

(2) $\Rightarrow$ (4) and (3) $\Rightarrow$ (5) : Trivial.

(4) $\Rightarrow$ (5): The proof is similar to that of (2) $\Rightarrow$ (3).

(5) $\Rightarrow$ (1): For each $A\in {\rm H}_c(X)$, by condition (5), $\emptyset \neq A^{\ua}=\bigcap\limits_{a\in A}\ua a=\bigcap\limits_{a\in A}\ua (A\cap \ua a)=\ua \left(A\cap\bigcap\limits_{a\in A} \ua d\right)$. By Proposition \ref{H-sober charact-H-sets}, $X$ is {\rm H}-sober.
\end{proof}

\begin{theorem}\label{H-sober charact-mapping} Let $H : \mathbf{Top}_0 \longrightarrow \mathbf{Set}$ be an R-subset system and $X$ a $T_0$ space. Then the following conditions are equivalent:
\begin{enumerate}[\rm (1)]
		\item $X$ is H-space.
        \item For every continuous mapping $f:X\longrightarrow Y$ to a $T_0$ space $Y$ and any $A\in H(X)$, $\ua f\left(\bigcap\limits_{a\in A}\ua a\right)=\bigcap\limits_{a\in A}\ua f(\ua a)=\bigcap\limits_{a\in A} \ua f(a)$.
        \item For every continuous mapping $f:X\longrightarrow Y$ to a H-sober space $Y$ and any $A\in H(X)$, $\ua f\left(\bigcap\limits_{a\in A}\ua a\right)=\bigcap\limits_{a\in A}\ua f(\ua a)=\bigcap\limits_{a\in A} \ua f(a)$.
        \item For every continuous mapping $f:X\longrightarrow Y$ to a sober space $Y$ and any $A\in H(X)$, $\ua f\left(\bigcap\limits_{a\in A}\ua a\right)=\bigcap\limits_{a\in A}\ua f(\ua a)=\bigcap\limits_{a\in A} \ua f(a)$.
\end{enumerate}
\end{theorem}
\begin{proof} (1) $\Rightarrow$ (2): First, $f$ is order-preserving. Since $X$ is ${\rm H}$-sober, there exists $x\in X$ with $\overline{A}=\overline{\{x\}}$, and hence $\bigvee A=\bigvee \overline{A}=x$. Therefore, $\ua f\left(\bigcap\limits_{a\in A}\ua a\right)=\ua f(\ua \bigvee A)=\ua f(\bigvee A)=\ua f(x)$. Obviously, $\ua f(\bigvee A)=\ua f(x)\subseteq \bigcap\limits_{a\in A} \ua f(a)$. On the other hand, if $y\in \bigcap\limits_{a\in A} \ua f(a)$, then $a\in f^{-1}(\da y)=f^{-1}(\overline{\{y\}})\in \Gamma (X)$ for all $a\in A$, and hence $\overline{\{x\}}=\overline{A}\subseteq f^{-1}(\da y)$, that is, $y\in \ua f(x)=\ua f(\bigvee A)$. Thus $f(\bigvee A)=\bigvee f(A)$, and whence $\ua f\left(\bigcap\limits_{a\in A}\ua a\right)=\ua f(\bigvee A)=\ua \bigvee f(A)=\bigcap\limits_{a\in A}\ua f(\ua a)=\bigcap\limits_{a\in A} \ua f(a)$.

(2) $\Rightarrow$ (3) $\Rightarrow$ (4): Trivial.

(4) $\Rightarrow$ (1):  Let $\eta_X : X \rightarrow X^s$ ($=P_H(\ir_c(X))$) be the canonical topological embedding from $X$ into its soberification (see Remark \ref{eta continuous}). For $A\in {\rm H}(X)$, by condition (4) we have $\overline{A}\in \bigcap\limits_{a\in A}\ua_{\ir_c(X)} \eta_X (a)=\bigcap\limits_{a\in A}\ua_{\ir_c(X)} \eta_X (\ua a)=\ua_{\ir_c(X)} \eta_X\left(\bigcap\limits_{a\in A}\ua a\right)=\ua_{\ir_c(X)} \eta_X (A^{\ua})$, and whence there is $x\in A^{\ua}$ such that $\overline{\{x\}}\subseteq \overline{A}$. Therefore, $\overline{A}=\overline{\{x\}}$. Thus $X$ is ${\rm H}$-sober.
\end{proof}

\begin{corollary}\label{H-sober charact-mapping-1} Let $H : \mathbf{Top}_0 \longrightarrow \mathbf{Set}$ be an R-subset system and $X$ a $T_0$ space. Then the following conditions are equivalent:
\begin{enumerate}[\rm (1)]
		\item $X$ is H-sober.
        \item $X$ is H-complete, and for every continuous mapping $f:X\longrightarrow Y$ to a $T_0$ space $Y$ and any $A\in H(X)$, $ f(\bigvee A)=\bigvee f(A)$.
        \item $X$ is H-complete, and for every continuous mapping $f:X\longrightarrow Y$ to a H-sober space $Y$ and any $A\in H(X)$, $ f(\bigvee A)=\bigvee f(A)$.
        \item $X$ is H-complete, and for every continuous mapping $f:X\longrightarrow Y$ to a sober space $Y$ and any $A\in H(X)$, $ f(\bigvee A)=\bigvee f(A)$.
\end{enumerate}
\end{corollary}
\begin{proof} (1) $\Rightarrow$ (2): Since $X$ is ${\rm H}$-sober, $X$ is {\rm H}-complete by Remark \ref{H-S H-C H-B}. By the proof of (1) $\Rightarrow$ (2) in Theorem \ref{H-sober charact-mapping}, we have $f(\bigvee A)=\bigvee f(A)$.

(2) $\Rightarrow$ (3) $\Rightarrow$ (4): Trivial.

(4) $\Rightarrow$ (1): For every continuous mapping $f:X\longrightarrow Y$ to a sober space $Y$ and any $A\in {\rm H}(X)$, by condition (4) we have $\ua f\left(\bigcap\limits_{a\in A}\ua a\right)=\ua f(\bigvee A)=\ua \bigvee f(A)=\bigcap\limits_{a\in A} \ua f(a)=\bigcap\limits_{a\in A}\ua f(\ua a)$, and whence by Theorem \ref{H-sober charact-mapping}, $X$ is {\rm H}-sober.
\end{proof}

By Lemma \ref{Scott H-cont} and Corollary \ref{H-sober charact-mapping-1}, we get the following corollary.

\begin{corollary}\label{H-sober charact-mapping-2} Let $H : \mathbf{Top}_0 \longrightarrow \mathbf{Set}$ be an R-subset system and $X$ a $T_0$ space. Then the following conditions are equivalent:
\begin{enumerate}[\rm (1)]
		\item $X$ is H-sober.
        \item $X$ is H-complete, and any continuous mapping $f:X\longrightarrow Y$ to a $T_0$ space $Y$ is Scott H-continuous.
        \item $X$ is H-complete, and any continuous mapping $f:X\longrightarrow Y$ to a H-sober space $Y$  is Scott H-continuous.
        \item $X$ is H-complete, and any continuous mapping $f:X\longrightarrow Y$ to a sober space $Y$  is Scott H-continuous.
\end{enumerate}
\end{corollary}

\begin{definition}\label{Hd set}  Let ${\rm H} : \mathbf{Top}_0 \longrightarrow \mathbf{Set}$ be an R-subset system and $X$ a $T_0$ space.
A subset $A$ of $X$ is called \emph{H}-\emph{sober determined}, if for any continuous mapping $ f:X\longrightarrow Y$
to an {\rm H}-sober space $Y$, there exists a unique $y_A\in Y$ such that $\overline{f(A)}=\overline{\{y_A\}}$. The $\mathcal R$-sober determined sets are shortly called \emph{sober determined} sets. Denote by ${\rm H}^d(X)$ the set of all {\rm H}-sober determined subsets of $X$. The set of all closed {\rm H}-sober determined subsets of $X$ is denoted by ${\rm H}^d_c(X)$.
\end{definition}

Clearly, a subset $A$ of a $T_0$ space $X$ is {\rm H}-sober determined if{}f $\overline{A}$ is {\rm H}-sober determined.

\begin{lemma}\label{sobd=irr} For a $T_0$ space $X$ and $A\subseteq X$, the following two conditions are equivalent:
\begin{enumerate}[\rm (1)]
\item $A$ is sober determined.
\item $A$ is irreducible.
\end{enumerate}
Therefore, $\mathcal R^d(X)=\ir(X)$.
\end{lemma}

\begin{proof} (1) $\Rightarrow$ (2): Suppose $A$ is sober determined. Now we show that $A$ is irreducible. Consider the sobrification $X^s$ ($=P_H(\ir_c(X)$) of $X$ and the canonical topological embedding $\eta_{X}: X\longrightarrow X^s$. Then there is $B\in \ir_c(X)$ such that $\Box_{\ir_c(X)}\overline{A}=\overline{\eta_X (A)}=\overline {\{B\}}=\Box_{\ir_c(X)}B$, and whence $\overline{A}=B$. Thus $A\in \ir (X)$.

(2) $\Rightarrow$ (1): By Lemma \ref{irrimage}.
\end{proof}

\begin{lemma}\label{H Hd ir} Let $H : \mathbf{Top}_0 \longrightarrow \mathbf{Set}$ be an R-subset system and $X$ a $T_0$ space. Then $H(X)\subseteq H^d(X)\subseteq \ir (X)$.
\end{lemma}

\begin{proof} First, if $ f:X\longrightarrow Y$ is a continuous mapping to an {\rm H}-sober space $Y$ and $A\in {\rm H}(X)$, then ${\rm H}(f)(A)=f(A)\in {\rm H}(Y)$, and hence by the {\rm H}-sobriety of $Y$, there exists a unique $y_A\in Y$ such that $\overline{f(A)}=\overline{\{y_A\}}$. Thus $A\in  {\rm H}^d(X)$. Now suppose $B\in {\rm H}^d(X)$. Since ${\rm H}\leq \mathcal R$, we have $B\in R^d(X)=\ir(X)$ by Lemma \ref{sobd=irr}.
\end{proof}

\begin{lemma}\label{Hd image} Let $H : \mathbf{Top}_0 \longrightarrow \mathbf{Set}$ be an R-subset system and $ f:X\longrightarrow Y$ a continuous mapping in $\mathbf{Top}_0$. Then $f(A)\in H^d(Y)$ for all $A\in H^d(X)$.
\end{lemma}

\begin{proof} Let $Z$ is an {\rm H}-sober space and $g:Y\longrightarrow Z$ is a continuous mapping.
Since $g\circ f:X\longrightarrow Z$ is continuous and $A\in {\rm H}^d (X)$, there is $z\in Z$ such that $\overline{g(f(A))}=\overline{g\circ f(A)}=\overline{\{z\}}$. Thus $\overline{f(A)}\in {\rm H}^d (Y)$.
\end{proof}

The following corollary follows directly from Lemma \ref{sobd=irr}, Lemma \ref{H Hd ir} and Lemma \ref{Hd image}.

\begin{corollary}\label{Hd R-subset system} Let $H : \mathbf{Top}_0 \longrightarrow \mathbf{Set}$ be an R-subset system. Then
\begin{enumerate}[\rm (1)]
\item $H^d : \mathbf{Top}_0 \longrightarrow \mathbf{Set}$ is an R-subset system, where for any continuous mapping $ f:X\longrightarrow Y$ in $\mathbf{Top}_0$, $H^d(f) : H^d(X) \longrightarrow H^d(Y)$ is defined by $H^d(f)(A)=f(A)$ for all $A\in H^d (X)$.
    \item $H\leq H^d\leq \mathcal R$.
    \item $\mathcal R^d=\mathcal R$
    \end{enumerate}
    \end{corollary}

\begin{example}\label{Johnstone dcpo} (Johnstone's dcpo)  Let $\mathbb{J}=\mathbb{N}\times (\mathbb{N}\cup \{\infty\})$ with ordering defined by $(j, k)\leq (m, n)$ if{}f $j = m$ and $k \leq n$, or $n =\infty$ and $k\leq m$ (see Figure 1). $\mathbb{J}$ is a well-known dcpo constructed by Johnstone in \cite{johnstone-81}. Then $\mathcal C(\mathbb{J}) \subset \mathcal D(\mathbb{J})$. By Proposition \ref{D-sober=C-sober} and Lemma \ref{H Hd ir},  $D(\mathbb{J})\subseteq \mathcal D^d(\mathbb{J})=\mathcal C^d(\mathbb{J})$. Therefore, $\mathcal C(\mathbb{J}) \subset \mathcal C^d(\mathbb{J})$.
\end{example}

\begin{figure}[ht]
	\centering
	\includegraphics[height=4.5cm,width=4.5cm]{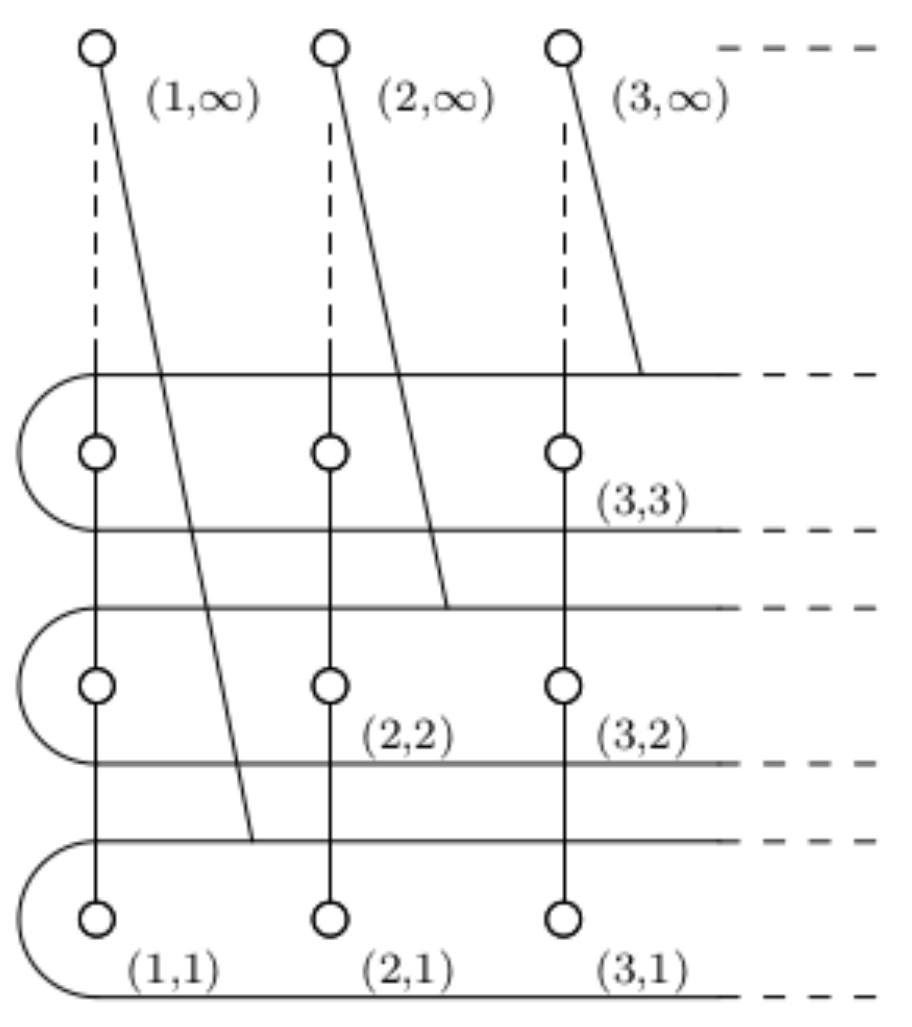}
	\caption{Johnstone's dcpo $\mathbb{J}$}
\end{figure}

\begin{proposition}\label{H-sober=Hd-sober} Let $H : \mathbf{Top}_0 \longrightarrow \mathbf{Set}$ be an R-subset system and $X$ a $T_0$ space. Then the following two conditions are equivalent:
\begin{enumerate}[\rm (1)]
\item $X$ is  H-sober.
    \item X is $H^d$-sober.
    \end{enumerate}
    \end{proposition}
\begin{proof} (1) $\Rightarrow$ (2): Let $A\in {\rm H}^d(X)$. Since $X$ is {\rm H}-sober and  the identity $id_X : X \longrightarrow X$ is continuous, there exists an $x\in X$ such that $\overline{A}=\overline{id_X(A)}=\overline{\{x\}}$. Therefore, $X$ is ${\rm H}^d$-sober.

(2) $\Rightarrow$ (1): By Lemma \ref{H Hd ir}.
\end{proof}

By Corollary \ref{Hd R-subset system} and Proposition \ref{H-sober=Hd-sober}, we get the following result.

\begin{corollary}\label{Hdd=Hd} Let $H : \mathbf{Top}_0 \longrightarrow \mathbf{Set}$ be an R-subset system. Then $d : \mathcal H\longrightarrow \mathcal H$, $H \mapsto H^d$, is a closure operator.
\end{corollary}

\begin{proposition}\label{H-sober closed} Let $X$ be an H-sober space.
\begin{enumerate}[\rm (1)]
\item If $A$ is a nonempty closed subspace of $X$, then $A$ is H-sober.
\item If $U$ is a nonempty saturated subspace of $X$, then $U$ is H-sober.
\end{enumerate}
\end{proposition}
\begin{proof} (1): Let $B\in {\rm H}(A)$. Then $B\in {\rm H}(X)$ (see Remark \ref{subspace H-set}), and hence there is $x\in X$ such that $\cl_X B=\cl_X\{x\}$. As $A\in \Gamma (X)$ and $B\subseteq A$, we have $x\in A$ and $\cl_X B\subseteq A$. It follows that $\cl_A B=\cl_X B=\cl _X\{x\}=\cl_A\{x\}$. Therefore, $A$ is {\rm H}-sober.

(2): Let $C\in {\rm H}(U)$. Then $C\in {\rm H}(X)$, and whence there is $x\in X$ such that $\cl_X C=\cl_X\{x\}$. By $C\subseteq U=\ua U$ and $C\subseteq \da x$, we have $x\in U$ and $\cl_U C=(\cl_X C)\bigcap U=(\cl_X\{x\})\bigcap U=\cl_U\{x\}$. Thus $U$ is {\rm H}-sober.
\end{proof}

\begin{proposition}\label{H-sober retract}
	A retract of an H-sober space is H-sober.
\end{proposition}
\begin{proof}
	It is well-known that a retract of a $T_0$ space is $T_0$ (cf. \cite{Engelking}). Assume $X$ is an ${\rm H}$-sober space and $Y$ a retract of $X$. Then there are continuous mappings $f:X\longrightarrow Y$ and $g:Y\longrightarrow X$ with $f\circ g={\rm id}_Y$. Let $B\in {\rm H}(Y)$. Then $g(B)\in{\rm H}(X)$. Since $X$ is {\rm H}-sober, there is $x\in X$ such that $\overline{g(B)}=\overline{\{x\}}$, and whence $\overline{B}=\overline{fg(B)}=\overline{f\left(\overline{g(B)}\right)}=\overline{f\left(\overline{\{x\}}\right)}=\overline{f\left(\{x\}\right)}=\overline{\{f(x)\}}$. Thus $Y$ is {\rm H}-sober.
\end{proof}

\begin{theorem}\label{H-sober prod}
	Let $\{X_i:i\in I\}$ be a family of $T_0$ spaces. Then the following two conditions are equivalent:
	\begin{enumerate}[\rm(1)]
		\item The product space $\prod_{i\in I}X_i$ is H-sober.
		\item For each $i \in I$, $X_i$ is H-sober.
	\end{enumerate}
\end{theorem}
\begin{proof}	
	(1) $\Rightarrow$ (2):  For each $i \in I$, $X_i$ is a retract of $\prod_{i\in I}X_i$. By Proposition \ref{H-sober retract}, $X_i$ is {\rm H}-sober.
	
	(2) $\Rightarrow$ (1): Let $X=\prod_{i\in I}X_i$. Then $X$ is $T_0$ (see \cite[Proposition 2.3.11]{Engelking}). Suppose $A\in {\rm H}(X)$. Then for each $i \in I$, $p_i(A)\in {\rm H}(X_i)$, and hence there is $x_i\in X_i$ such that $\cl_{X_i} p_i(A)=\cl_{X_i}\{x_i\}$. Let $x=(x_i)_{i\in I}$. Then by Lemma \ref{irrprod} and \cite[Proposition 2.3.3]{Engelking}), we have $\cl_X(A)=\prod_{i\in I}\cl_{X_i} p_i(A)=\prod_{i\in I}\cl_{X_i}\{x_i\}=\cl_X \{x\}$. So $X$ is {\rm H}-sober.
\end{proof}

\begin{theorem}\label{H-sober function space}
	Let $H : \mathbf{Top}_0 \longrightarrow \mathbf{Set}$ be an R-subset system. If $X$ is a $T_0$ space and $Y$ an H-sober space, then the function space  $\mathbf{Top}_0(X, Y )$ of all continuous functions
$f : X\longrightarrow Y$ equipped with the topology of pointwise convergence \emph{(}i.e.,
the relative product topology\emph{)} is H-sober.
\end{theorem}
\begin{proof}	 Let $\mathbf{Top}_0(X, Y)$ be the subspace of product space $Y^X$, that is, $\mathbf{Top}_0(X, Y)$ is endowed with the topology induced by the product topology on $Y^X$. For $A\in {\rm H}(\mathbf{Top}_0(X, Y))$, we have $A\in {\rm H}(Y^X)$, and whence for
each $x\in X$, $p_x(A)\in {\rm H}(Y)$, where $p_x$ is the $x$th projection. As $Y$ is {\rm H}-sober, there is a
unique element $a_x\in Y$ such that $\overline{p_x(A)}=\overline{\{a_x\}}$ for
each $x\in X$. We now show that the function
$f : X \longrightarrow Y$ defined by $f(x)=a_x$ is continuous. Indeed let $x\in X$ and let $V\in \mathcal O(Y)$ with $f(x)=a_x\in V$.
Then $V\bigcap p_x(A)\neq \emptyset$, that is,
there is an element $a\in A$ such that $a(x)\in V$. As $a : X \longrightarrow Y$ is continuous, there is a
$U\in \mathcal O(X)$ with $x\in U$ such that $a(z)\in V$ for every $z\in U$. Since $a\in A$, we have $a(z)\in p_z(A)\subseteq \overline{p_z(A)}=\overline{\{a_z\}}=\overline{\{f(z)\}}$, and hence $f(z)\in V$ for all $z\in U$. Thus $f$
is continuous. Finally, we shows that $\overline{A}=\overline{\{f\}}$ in $\mathbf{Top}_0(X, Y )$ (with the topology
induced by the product topology on $Y^X$). For any subbasic open set $p_x^{-1}(U_x)$ ($x\in X$ and $U_x\in \mathcal O(Y)$) with $f\in p_x^{-1}(U_x)$, since $f(x)=a_x\in \overline{p_x(A)}$, we have $f\in p_x^{-1}(\overline{p_x(A)})\bigcap p_x^{-1}(U_x))=p_x^{-1}(\overline{p_x(A)}\bigcap U_x)$. Therefore,
$\overline{p_x(A)}\bigcap U_x\neq\emptyset$, and hence $A\bigcap p_x^{-1}(U_x)\neq\emptyset$. Thus all basic open sets of $f$
must meet $A$ since $A\in {\rm H}(Y^X)\subseteq \ir (Y^X)$. It follows that $\overline{A}=\overline{\{f\}}$. Thus $\mathbf{Top}_0(X, Y )$, as a subspace of the product space $Y^X$, is {\rm H}-sober.
\end{proof}

 \begin{proposition}\label{H-sober equalizer} Let $X$ be an H-sober space and $Y$ a $T_0$ space. If $f, g: X \longrightarrow Y$ are continuous, then the equalizer
$E(f, g)=\{x\in X : f(x)=g(x)\}$ \emph{(}as a subspace of $X$\emph{)} is H-sober.
\end{proposition}
\begin{proof} Suppose $E(f, g)\neq\emptyset$. Then as a subspace of $X$, $E(f, g)$ is $T_0$. Let $A\in {\rm H}(E(f, g))$. Then $A\in {\rm H}(X)$. By the {\rm H}-sobriety of $X$, there is $x\in X$ such that $\cl_X(A)=\cl_X\{x\}$, and consequently, $\cl_Y f(\{x\})=\cl_Y f(\cl_X\{x\})=\cl_Y f(\cl_X A)=\cl_Y f(A)=\cl_Y g(A)=\cl_Y g(\cl_X A)=\cl_Y g(\cl_X\{x\})=\cl_Y g(\{x\})$. Since $Y$ is $T_0$, we have $f(x)=g(x)$, and hence $x\in E(f, g)$. It follows that $cl_{E(f,~\! g)}A=\cl_X(A)\bigcap E(f, g)=\cl_X\{x\}\bigcap E(f, g)=\cl_{E(f,~ \! g)}\{x\}$, proving that $E(f, g)$ is ${\rm H}$-sober.
\end{proof}

By Theorem \ref{H-sober prod} and  Proposition \ref{H-sober equalizer}, we get the following corollary.

\begin{corollary}\label{H-sober complete} For any R-subset system $H : \mathbf{Top}_0 \longrightarrow \mathbf{Set}$, $\mathbf{H} $-$\mathbf{Sob}$ is complete.
\end{corollary}

As a direct corollary, we  have the following known result (cf. \cite{redbook}).

\begin{corollary}\label{sober d-space complete} $\mathbf{Top}_d$ and $\mathbf{Sob}$ are complete.
\end{corollary}

\section{Super H-sober spaces}

In this section, we shall introduce and investigate a strong type of H-sober spaces --- super H-sober spaces. Several important connections between super H-sober spaces and H-sober spaces will be given.

\begin{definition}\label{super H-sober space} Let ${\rm H} : \mathbf{Top}_0 \longrightarrow \mathbf{Set}$ be an R-subset system and $X$ a $T_0$ space.
$X$ is called \emph{super H}-\emph{sober} provided its Smyth power space $P_S(X)$ is {\rm H}-sober, that is, for any $\mathcal A\in{\rm H}(P_S(X))$, there is a (unique) $K\in \mk (X)$ such that $\cl_{P_S(X)}{\mathcal A}=\cl_{P_S(X)}\{K\}=\da_{\mk (X)}K$. The super $\mathcal R$-sober spaces are shortly called \emph{super sober spaces}. The category of all super {\rm H}-sober spaces with continuous mappings is denoted by $\mathbf{SH}$-$\mathbf{Sob}$.
\end{definition}

\begin{definition}\label{HIP property} Let $X$ be a $T_0$ space.
 \begin{enumerate}[\rm (1)]
 \item $X$ is called \emph{Smyth H}-\emph{complete}, if $\bigcap \mathcal A\in \mk (X)$ for each $\mathcal A\in {\rm H}(P_S(X))$. $\mk (X)$ is called \emph{irreducible complete} if $X$ is Smyth $\mathcal R$-complete, that is, for each $\mathcal A\in \ir (P_S(X))$, $\bigcap \mathcal A\in \mk (X)$ (cf. Lemma \ref{sups in Smyth}).
 \item $X$ is said to have \emph{H}-\emph{intersection property}, $\mathbf{HIP}$ for short, if $\bigcap \mathcal A\neq\emptyset$ for each $\mathcal A\in {\rm H}(P_S(X))$.

\end{enumerate}
\end{definition}

For ${\rm H}=\mathcal D$ (resp., $\mathcal R$), the {\rm H}-intersection property is also called \emph{filtered intersection property} (resp., \emph{irreducible intersection property}), $\mathbf{FIP}$ (resp., $\mathbf{RIP}$) for short (cf. \cite{xu-shen-xi-zhao1}).

\begin{lemma}\label{Ps bounded=HIP}  Let $X$ be a $T_0$ space. Then
\begin{enumerate}[(1)]
\item $X$ is Smyth H-complete if{}f $P_S(X)$ is H-complete.
\item $X$ has $\mathbf{HIP}$ if{}f $P_S(X)$ is H-bounded.
\end{enumerate}
\end{lemma}

\begin{proof} (1): If $X$ is Smyth {\rm H}-complete, then for each $\mathcal A\subseteq {\rm H}(P_S(X))$, $\bigcap \mathcal A\in \mk (X)$, and hence by Lemma \ref{sups in Smyth}, $\bigvee_{\mk (X)} \mathcal A$ exists and $\bigvee_{\mk (X)}=\bigcap \mathcal A$. Conversely, if $P_S(X)$ is {\rm H}-complete, then for each $\mathcal A\subseteq {\rm H}(P_S(X))$, $\bigvee_{\mk (X)} \mathcal A$ exists. By Lemma \ref{sups in Smyth}, $\bigcap \mathcal A\in \mk (X)$ and $\bigvee_{\mk (X)}=\bigcap \mathcal A$.

(2): If $X$ has $\mathbf{HIP}$, then for each $\mathcal A\in {\rm H}(P_S(X))$, $\bigcap \mathcal A\neq\emptyset$. Select an $x\in \bigcap \mathcal A$. Then $\ua x\in \mk (X)$ is an upper bound of $\mathcal A$ in $\mk (X)$. Conversely, if $P_S(X)$ is H-bounded, then for any $\mathcal A\in {\rm H}(P_S(X))$, $\mathcal A$ has an upper bound $K\in \mk (X)$. It follows that $\emptyset \neq K\subseteq \bigcap \mathcal A$, proving that $X$ has $\mathbf{HIP}$.
\end{proof}

By Remark \ref{two meets in Smyth}, Remark \ref{H-S H-C H-B} and Lemma \ref{Ps bounded=HIP}, we have the following implications:
\begin{center}

super {\rm H} sobriety $\Rightarrow$ Smyth {\rm H}-completeness $\Rightarrow$ $\mathbf{HIP}$ $\Rightarrow$ {\rm H}-boundedness.\\

\end{center}

The following example shows that Smyth {\rm H}-completeness does not implies the {\rm H}-completeness in general.

\begin{example}\label{example 1} Let $P$=$\{a_1, a_2\}\cup \mathbb{N}$ with the partial order defined by (1) for each $n\in \mathbb{N}$, $ n<n+1$, $n<a_1$ and $n<a_2$; (2) $a_1$ and $a_2$ are incomparable (see Figure 2). Then the Scott space $\Sigma~\!\! P$ is not $\mathcal D$-complete, that is, $P$ is not a dcpo. But $\mk (\Sigma~\!\! P)=\{\ua n : n\in \mathbb{N}\}\bigcup\{\ua a_1=\{a_1\}, \ua a_2=\{a_2\}, \ua \{a_1, a_2\}=\{a_1, a_2\}\}$ (with the Smyth order) is a dcpo, and hence $\Sigma~\!\! P$ is Smyth $\mathcal D$-complete.
\end{example}

\begin{figure}[ht]
	\centering
	\includegraphics[height=3.5cm,width=2.5cm]{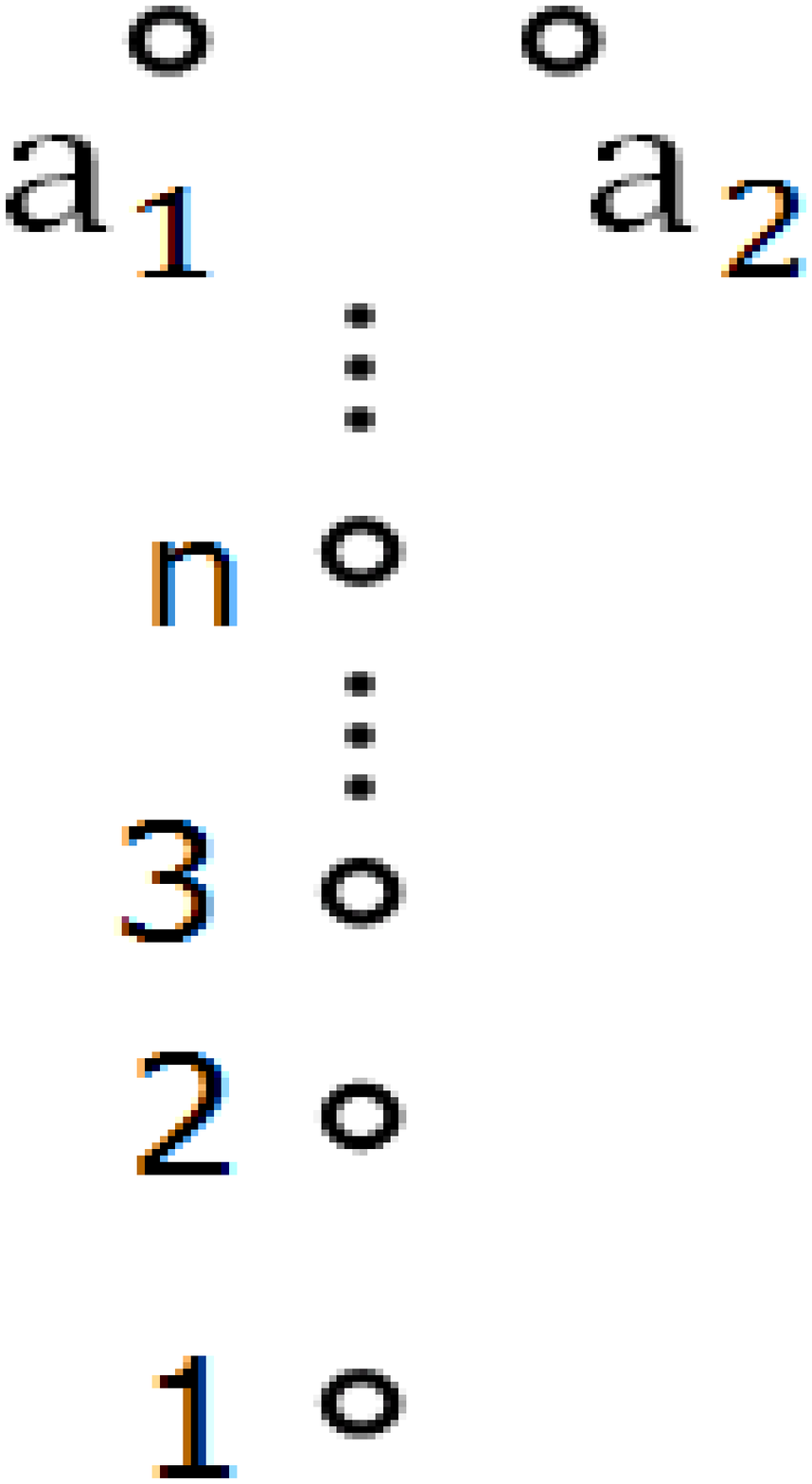}
	\caption{A super $\mathcal D$-complete non-dcpo $T_0$ space}
\end{figure}

By Lemma \ref{many  meets in Smyth} and Proposition \ref{H-sober charact-H-sets}, we have the following result.

\begin{theorem}\label{super H-sober charact H-sets} Let $H : \mathbf{Top}_0 \longrightarrow \mathbf{Set}$ be an R-subset system and $X$ a $T_0$ space. Then the following conditions are equivalent:
\begin{enumerate}[\rm (1)]
	        \item $X$ is super H-sober.

            \item  For any $\mathcal K\in H(P_S(X))$, $\cl_{P_S(X)}\mathcal K\cap\bigcap\limits_{K\in \mathcal K}\ua_{\mk (X)} K\neq\emptyset$.

            \item  For any $\mathcal K\in  H(P_S(X))$ and $\mathcal U\in \mathcal O(P_S(X))$, $\bigcap\limits_{K\in \mathcal K}\ua_{\mk (X)} K\subseteq \mathcal U$ implies $\ua_{\mk (X)} K \subseteq \mathcal U$ \emph{(}i.e., $K\in \mathcal U$\emph{)} for some $K\in \mathcal K$.
            \item  For any $\mathcal K\in  H(P_S(X))$ and $U\in \mathcal O(X)$, $\bigcap\limits_{K\in \mathcal K}\ua_{\mk (X)} K\subseteq \Box U$ implies $\ua_{\mk (X)} K \subseteq \Box U$ \emph{(}i.e., $K\subseteq U$\emph{)} for some $K\in \mathcal K$.

                \item  For any $\mathcal K\in  H(P_S(X))$ and $U\in \mathcal O(X)$, $\bigcap\mathcal K\subseteq U$ implies $K \subseteq U$ for some $K\in \mathcal K$.
                    \item  For any $\mathcal K\in  H(P_S(X))$ and $U\in \mathcal O(X)$, $\bigcap\mathcal K\in \mk (X)$, and $\bigcap\mathcal K\subseteq U$ implies $K \subseteq U$ for some $K\in \mathcal K$.

\end{enumerate}
\end{theorem}

\begin{remark}\label{super H-sober charact remark} In Theorem \ref{super H-sober charact H-sets} we can also chose $\mathcal K\in  H_c(P_S(X))$ (see Remark \ref{intersection=closure intersection in Smyth}).
\end{remark}

\begin{corollary}\label{WF=Smyth d-space} For a $T_0$ space $X$, the following conditions are equivalent:
\begin{enumerate}[\rm (1)]
	        \item $X$ is well-filtered.
            \item $P_S(X)$ is a $d$-space.
            \item  For any $\mathcal K\in  \mathcal D(\mk (X))$, $\cl_{P_S(X)}\mathcal K\cap\bigcap\limits_{K\in \mathcal K}\ua_{\mk (X)} K\neq\emptyset$.
            \item  For any $\mathcal K\in  \mathcal D(\mk (X))$ and $\mathcal U\in \mathcal O(P_S(X))$, $\bigcap\limits_{K\in \mathcal K}\ua_{\mk (X)} K\subseteq \mathcal U$ implies $\ua_{\mk (X)} K \subseteq \mathcal U$ \emph{(}i.e., $K\in \mathcal U$\emph{)} for some $K\in \mathcal K$.
            \item  For any $\mathcal K\in  \mathcal D(\mk (X))$ and $U\in \mathcal O(X)$, $\bigcap\limits_{K\in \mathcal K}\ua_{\mk (X)} K\subseteq \Box U$ implies $\ua_{\mk (X)} K \subseteq \Box U$ \emph{(}i.e., $K\subseteq U$\emph{)} for some $K\in \mathcal K$.
            \item  For any $\mathcal K\in \mathcal D(\mk (X))$ and $U\in \mathcal O(X)$, $\bigcap\mathcal K\in \mk (X)$, and $\bigcap\mathcal K\subseteq U$ implies $K \subseteq U$ for some $K\in \mathcal K$.
\end{enumerate}
\end{corollary}

The equivalence of conditions (1) and (2) was first shown in \cite{xi-zhao-MSCS-well-filtered}. By Proposition \ref{D-sober=C-sober} and Theorem \ref{super H-sober charact H-sets}, we get the following conclusion, which was first obtained by Shen\footnote{C. Shen, An equivalent description of well-filtered spaces, preprint.}.

\begin{corollary}\label{WF=CWF}
A $T_0$ space is well-filtered if and only if $P_S(X)$ is $\mathcal C$-sober, that is, for each chain $\mathcal C$ of compact saturated sets and each open set $U$ with $\bigcap \mathcal C\subseteq U$, there exists $K\in \mathcal C$ such that $K\subseteq U$.
\end{corollary}

\begin{proposition}\label{omega WF=omega CWF} For a $T_0$ space $X$, the following conditions are equivalent:
\begin{enumerate}[\rm (1)]
 \item $X$ is $\omega$-well-filtered.
  \item $X$ is $\mathcal C^\omega$-sober.
  \item For any countable descending chain $K_1\supseteq K_2\supseteq\ldots\supseteq K_n\supseteq\ldots$ of compact saturated subsets of $X$ and $U\in\mathcal O(X)$, $\bigcap_{n\in \mathbb{N}}K_n\subseteq U$ implies $K_{n_0}\subseteq U$ for some $n_0\in \mathbb{N}$.
  \end{enumerate}
\end{proposition}

\begin{proof} (1) $\Leftrightarrow$ (2): By Proposition \ref{D omega-sober=C omega-sober} and Theorem \ref{super H-sober charact H-sets}.

(2) $\Rightarrow$ (3): Trivial.

(3) $\Rightarrow$ (2): For any countable chain $\{G_n : n\in \mathbb{N}\}$ of compact saturated subsets of $X$ and $U\in\mathcal O(X)$ with $\bigcap_{n\in \mathbb{N}}G_n\subseteq U$, let $K_n=\mathrm{min}\{G_1, G_2, ..., G_n\}$ for each $n\in \mathbb{N}$. Then $\{K_n : n\in \mathbb{N}\}$ is a countable descending chain and $\bigcap_{n\in \mathbb{N}}K_n=\bigcap_{n\in \mathbb{N}}G_n\subseteq U$. By condition (3), $K_{n_0}\subseteq U$ for some $n_0\in \mathbb{N}$. Since $\{G_n : n\in \mathbb{N}\}$ is a chain with the set inclusion order, $K_{n_0}=\mathrm{min}\{G_1, G_2, ..., G_{n_0}\}=G_{m_0}$ for some $1\leq m_0\leq n_0$, and hence $G_{m_0}\subseteq U$.

\end{proof}

\begin{definition}\label{R-SS property Q} A R-subset system ${\rm H} : \mathbf{Top}_0 \longrightarrow \mathbf{Set}$ is said to satisfy  \emph{property Q} if for any $\mathcal K\in {\rm H}(P_S(X))$ and any $A\in M(\mathcal{K})$, $A$ contains a closed {\rm H}-set $C$ such that $C\in M(\mathcal{K})$.
\end{definition}

By Lemma \ref{t Rudin}, we have the following corollary.

\begin{corollary}\label{R propert Q} The R-subset system $\mathcal R : \mathbf{Top}_0 \longrightarrow \mathbf{Set}$ has property Q.
\end{corollary}

\begin{theorem}\label{super H-sober is H-sober} Let $H : \mathbf{Top}_0 \longrightarrow \mathbf{Set}$ be an R-subset system and $X$ a $T_0$ space. Consider the following two conditions:
\begin{enumerate}[\rm (1)]
\item $X$ is super H-sober.
\item $X$ is H-sober.
\end{enumerate}
Then \emph{(1)} $\Rightarrow$ \emph{(2)}, and two conditions are equivalent if $H$ has property Q.
\end{theorem}

\begin{proof} (1) $\Rightarrow$ (2): Let $A\in {\rm H}(X)$ and $U\in \mathcal U$ with $\bigcap\limits_{a\in A}\ua a\subseteq U$. Then by Remak \ref{xi continuous}, $\xi_X(A)\in {\rm H}(P_S(X))$, and hence by Theorem \ref{super H-sober charact H-sets}, $\ua a\subseteq U$ for some $a\in A$. Therefore, by Proposition \ref{H-sober charact-H-sets}, $X$ is H-sober.

(2) $\Rightarrow$ (1): Suppose that ${\rm H}$ has property {\rm Q}. Let $\mathcal A\in {\rm H}((P_S(X))$ and $U\in \mathcal O(X)$ with $\bigcap\mathcal A \subseteq U$. If $K\nsubseteq U$ for all $K\in \mathcal A$, then by the property {\rm Q} of {\rm H}, $X\setminus U$ contains an {\rm H}-set $C$ such that $\overline{C}\in M(\mathcal{K})$. Since $X$ is {\rm H}-sober, there exists $x\in X$ such that $\overline{C}=\overline{\{x\}}$, and hence $\overline{\{x\}}\in M(\mathcal{K})$. Therefore, $x\in \bigcap \mathcal A\subseteq U$, which is a contradiction with $\overline{\{x\}}=\overline{C}\subseteq X\setminus U$. It follows from Theorem \ref{super H-sober charact H-sets} that $X$ is super {\rm H}-sober.
\end{proof}

\begin{remark}\label{WF is d-space} By Corollary \ref{WF=Smyth d-space} and Theorem \ref{super H-sober is H-sober}, we get the known result that every well-filtered space is a $d$-space (see, e.g., \cite[Proposition 2.1]{Xi-Lawson-2017}).
\end{remark}

\begin{corollary}\label{SH-sober is subcat of H-sober} For any R-subset system $H : \mathbf{Top}_0 \longrightarrow \mathbf{Set}$,  $\mathbf{SH}$-$\mathbf{Sob}$ is a full subcategory of $\mathbf{H}$-$\mathbf{Sob}$.
\end{corollary}

By Theorem \ref{super H-sober charact H-sets}, Corollary \ref{R propert Q} and Theorem \ref{super H-sober is H-sober}, we have the following result (see \cite[Theorem 3.11]{Klause-Heckmann} and \cite[Lemma 7.20]{Schalk}).

\begin{theorem}\label{Heckman-Keimel theorem}\emph{(Heckmann-Keimel-Schalk Theorem)} For a $T_0$ space $X$, the following conditions are equivalent:
\begin{enumerate}[\rm (1)]
            \item $X$ is sober.
             \item  For any $\mathcal A\in  \ir(P_S(X))$, $\cl_{P_S(X)}\mathcal A\cap\bigcap\limits_{K\in \mathcal A}\ua_{\mk (X)} K\neq\emptyset$.

            \item  For any $\mathcal A\in  \ir(P_S(X))$ and $\mathcal U\in \mathcal O(P_S(X))$, $\bigcap\limits_{K\in \mathcal A}\ua_{\mk (X)} K\subseteq \mathcal U$ implies $\ua_{\mk (X)} K \subseteq \mathcal U$ \emph{(}i.e., $K\in \mathcal U$\emph{)} for some $K\in \mathcal A$.
            \item  For any $\mathcal A\in  \ir(P_S(X))$ and $U\in \mathcal O(X)$, $\bigcap\limits_{K\in \mathcal A}\ua_{\mk (X)} K\subseteq \Box U$ implies $\ua_{\mk (X)} K \subseteq \Box U$ \emph{(}i.e., $K\subseteq U$\emph{)} for some $K\in \mathcal A$.
 \item  For any $\mathcal A\in \ir(P_S(X))$ and $U\in \mathcal O(X)$, $\bigcap\mathcal K\subseteq U$ implies $K \subseteq U$ for some $K\in \mathcal A$.
                \item  For any $\mathcal A\in  \ir(P_S(X))$ and $U\in \mathcal O(X)$, $\bigcap\mathcal A\in \mk (X)$, and $\bigcap\mathcal K\subseteq U$ implies $K \subseteq U$ for some $K\in \mathcal A$.
           \item $X$ is super sober \emph{(}that is, $P_S(X)$ is sober\emph{)}.
\end{enumerate}
\end{theorem}

The following example shows that for ${\rm H}=\mathcal D, \mathcal D^\omega, \mathcal C$ and $\mathcal C^\omega$, {\rm H} does not satisfy property {\rm Q}, and ${\rm H}$-sobriety indeed does not imply super ${\rm H}$-sobriety in general.

\begin{example}\label{d-space not WF}  Let $\mathbb{J}$ be the Johnstone's dcpo (see Example \ref{Johnstone dcpo}). Then $\mathbb{J}$ is a countable dcpo, $\mathcal D=\mathcal D^\omega$ and $\mathcal C=\mathcal C^\omega$. The Johnstone space $\Sigma~\!\mathbb{J}$ is a $\mathcal D$-sober space (i.e., a $d$-space). However, $\Sigma~\!\mathbb{J}$ is not well-filtered (see \cite[Exercise 8.3.9]{Jean-2013}), that is, $\Sigma~\!\mathbb{J}$ is not super $\mathcal D$-sober by Corollary \ref{WF=Smyth d-space}. By Theorem \ref{super H-sober is H-sober}, $\mathcal D$ does not have property {\rm Q}, and hence by Proposition \ref{D-sober=C-sober}, $\mathcal C$ do not satisfy property {\rm Q}.
\end{example}

For the super {\rm H}-sobriety, we have the following result.

\begin{theorem}\label{super H-sober=PS super H-sober} Let $H : \mathbf{Top}_0 \longrightarrow \mathbf{Set}$ be an R-subset system having property M and $X$ a $T_0$ space. Then the following conditions are equivalent:
\begin{enumerate}[\rm (1)]
\item $X$ is super H-sober.
\item $P_S(X)$ is super H-sober.
\end{enumerate}
\end{theorem}
\begin{proof} (1) $\Rightarrow$ (2): Suppose that $\{\mathcal K_i : i\in I\}\in{\rm H}(P_S(P_S(X)))$ and $\mathcal U\in \mathcal O(P_S(X))$ with $\bigcap\limits_{i\in I}\mathcal K_i\subseteq \mathcal U$. If $\mathcal K_i\not\subseteq \mathcal U$ for each $i\in I$, then by Lemma \ref{t Rudin}, $\mk (X)\setminus \mathcal U$ contains a minimal irreducible closed subset $\mathcal C$ that still meets all
$\mathcal K_i$. For each $i\in I$, let $K_i=\bigcup \mathcal \ua_{\mk (X)} (\mathcal C\bigcap \mathcal K_i)$ ($=\bigcup (\mathcal C\bigcap \mathcal K_d$)). Then by Corollary \ref{Smythunioncont} and the property M of {\rm H}, $\{K_i : i\in I\}\in {\rm H}(P_S(X))$, and $K_i\in \mathcal C$ for each $i\in I$ since $\mathcal C=\da_{\mk (X)}\mathcal C$. Let $K=\bigcap\limits_{i\in I} K_i$. Then by Lemma \ref{sups in Smyth}, Proposition \ref{H-sober Scott H-open system} and condition (1), we have $K\in \mk (X)$ and $K=\bigvee_{\mk (X)} \{K_i : i\in I\}\in \mathcal C$. We claim that $K\in \bigcap\limits_{i\in I}\mathcal K_i$. Suppose, on the contrary, that $K\not\in \bigcap\limits_{i\in I}\mathcal K_i$. Then there is $i_0\in I$ such that $K\not\in \mathcal K_{i_0}$. Select a $G\in \mathcal C\bigcap \mathcal K_{i_0}$. Then $K\not\subseteq G$ and we can choose $g\in K\setminus G$. It follows that $g\in K_i=\bigcup (\mathcal A\bigcap \mathcal K_i)$ for all $i\in I$ and $G\not\in \Diamond_{\mk (K)}\overline{\{g\}}$. Thus $\Diamond_{\mk (K)}\overline{\{g\}}\bigcap\mathcal A\bigcap \mathcal K_i\neq\emptyset$ for all $i\in I$. By the minimality of $\mathcal C$, we have $\mathcal C=\Diamond_{\mk (K)}\overline{\{g\}}\bigcap\mathcal C$, and consequently, $G\in \mathcal C\bigcap \mathcal K_{i_o}=\Diamond_{\mk (K)}\overline{\{g\}}\bigcap\mathcal A\bigcap \mathcal K_{i_0}$, which is a contradiction with $G\not\in \Diamond_{\mk (K)}\overline{\{g\}}$. Thus $K\in \bigcap\limits_{i\in I}\mathcal K_i\subseteq \mathcal U\subseteq \mk (X)\setminus \mathcal A$, being a contradiction with $K\in \mathcal C$. Therefore, $P_S(X)$  is super {\rm H}-sober by Theorem \ref{super H-sober charact H-sets}.

(2) $\Rightarrow$ (1): By Theorem \ref{super H-sober is H-sober}

\end{proof}

By Remark \ref{C D property M}, Corollary \ref{WF=Smyth d-space} and Theorem \ref{super H-sober=PS super H-sober}, we get the following corollary.

\begin{corollary}\label{WF=Ps WF}\emph{(\cite{xuxizhao})}
For a $T_0$ space $X$, the following conditions are equivalent:
\begin{enumerate}[\rm (1)]
\item $X$ is well-filtered.
\item $P_S(X)$ is a $d$-space.
\item $P_S(P_S(X))$ is a $d$-space.
\item $P_S(X)$ is well-filtered.
\end{enumerate}
\end{corollary}

A topological space $X$ is called \emph{consonant} if for every $\mathcal F\in \sigma (\mathcal O(X))$, there is a family $\{K_i : i\in I\}\subseteq \mk (X)$ with $\mathcal F=\bigcup_{K\in \mathcal A} \Phi (K)$ (see, e.g., \cite{Jean-2013}). Inspired by this concept, we give the following definition.

\begin{definition}\label{H-consonant}  Let ${\rm H} : \mathbf{Top}_0 \longrightarrow \mathbf{Set}$ be an R-subset system. A $T_0$ space $X$ is called \emph{H}-\emph{consonant} if for any $\mathcal F\in \mathrm{OFilt(\mathcal O(X))}$, there is an $\mathcal A\in {\rm H}(P_S(X))$ with $\mathcal F=\mathcal F_{\mathcal A}=\bigcup_{K\in \mathcal A} \Phi (K)$ (cf. Lemma \ref{irr-induced-opne filter}).
\end{definition}

By Lemma \ref{irr-induced-opne filter}, the Hofmann-Mislove Theorem (Theorem \ref{Hofmann-Mislove theorem}), Theorem \ref{super H-sober charact H-sets}, Corollary \ref{WF=Smyth d-space} and Theorem \ref{Heckman-Keimel theorem}, we get the following result.

\begin{proposition}\label{sober H-consonant} Let $H : \mathbf{Top}_0 \longrightarrow \mathbf{Set}$ be an R-subset system and $X$ a $T_0$ space. Then the following conditions are equivalent:
\begin{enumerate}[\rm (1)]
	        \item $X$ is sober.
            \item  $X$ is H-consonant and super H-sober.
            \item $X$ is $\mathcal S$-consonant.
            \item  $X$ is $\mathcal D$-consonant and super $\mathcal D$-sober \emph{(}i.e., well-filtered\emph{)}.
            \item  $X$ is $\mathcal R$-consonant and super $\mathcal R$-sober \emph{(}i.e., super sober\emph{)}.
\end{enumerate}
\end{proposition}

In the following, we shall give some equational characterizations of super ${\rm H}$-spaces.

\begin{theorem}\label{super H-sober cha equa H-set} Let $H : \mathbf{Top}_0 \longrightarrow \mathbf{Set}$ be an R-subset system having property M and $X$ a $T_0$ space. Then the following conditions are equivalent:
\begin{enumerate}[\rm (1)]
	        \item $X$ is super H-sober.
            \item  $X$ has $\mathbf{HIP}$ \emph{(}especially, $X$ is Smyth $H$-complete\emph{)}, and $\ua_{\mk (X)} \left(\mathcal C\cap\bigcap\limits_{K\in \mathcal A} \ua_{\mk (X)} K\right)=\bigcap\limits_{K\in \mathcal A}\ua_{\mk (X)} (\mathcal C\cap \ua_{\mk (X)} K)$ for any $\mathcal A\in H(P_S(X))$ and $\mathcal C\in \Gamma(P_S(X))$.
            \item  $X$ has $\mathbf{HIP}$ \emph{(}especially, $X$ is Smyth $H$-complete\emph{)}, and $\ua_{\mk (X)} \left(\mathcal C\cap\bigcap\limits_{K\in \mathcal A} \ua_{\mk (X)} K\right)=\bigcap\limits_{K\in \mathcal A}\ua_{\mk (X)} (\mathcal C\cap \ua_{\mk (X)} K)$ for any $\mathcal A\in H(P_S(X))$ and $\mathcal C\in \ir_c(P_S(X))$.
            \item  $X$ has $\mathbf{HIP}$ \emph{(}especially, $X$ is Smyth $H$-complete\emph{)}, and $\ua \left(C\cap\bigcap\mathcal A\right)=\bigcap\limits_{K\in \mathcal A}\ua (C\cap K)$ for any $\mathcal A\in H (P_S(X))$ and $C\in \Gamma (X)$.
            \item  $X$ has $\mathbf{HIP}$ \emph{(}especially, $X$ is Smyth $H$-complete\emph{)}, and $\ua \left(C\cap\bigcap\mathcal A\right)=\bigcap\limits_{K\in \mathcal A}\ua (C\cap K)$ for any $\mathcal A\in H (P_S(X))$ and $C\in \ir_c (X)$.
\end{enumerate}
\end{theorem}

\begin{proof} (1) $\Leftrightarrow$ (2): By Proposition \ref{H-sober charact-bounded} and Lemma \ref{Ps bounded=HIP}.

(2) $\Rightarrow$ (3): Trivial.

(3) $\Rightarrow$ (1): Let $\mathcal A\in  {\rm H}(P_S(X))$ and $U\in \mathcal O(X)$ with $\bigcap\limits_{K\in \mathcal A}\ua_{\mk (X)} K\subseteq \Box U$. If $K\nsubseteq U$ for all $K\in \mathcal A$, then by Lemma \ref{t Rudin}, $X\setminus U$ contains a minimal irreducible closed subset $A$ that still meets all members of $\mathcal{A}$. By Lemma \ref{X-Smyth-irr}, $\Diamond A\in \ir_c(P_S(X))$. Since {\rm H} has property M and $X$ has $\mathbf{HIP}$, we have $\bigcap\limits_{K\in \mathcal A}\ua (A\cap K)\neq\emptyset$. Select $x\in \bigcap\limits_{K\in \mathcal A}\ua (A\cap K)$. Then by condition (3), we have $\ua x\in \bigcap\limits_{K\in\mathcal K}\ua_{\mk (X)}(\Diamond A\cap \ua_{\mk (X)}K)=\ua_{\mk (X)} (\Diamond A\cap \bigcap\limits_{K\in \mathcal A}\ua_{\mk (X)} K)$, and hence there is $G\in \mk (X)$ such that $G\in \Diamond A\cap \bigcap\limits_{K\in \mathcal A}\ua_{\mk (X)} K$ and $G\sqsubseteq \ua x$, that is, $\ua x\subseteq G$, $A\bigcap G\neq \emptyset$ and $G\subseteq \bigcap \mathcal A$. It follows that $\emptyset \neq A\bigcap G\subseteq A\cap \bigcap \mathcal A \subseteq A \cap U$, which is in contradiction with $A\subseteq X\setminus U$. By Theorem \ref{super H-sober charact H-sets}, $X$ is super {\rm H}-sober.

(1) $\Rightarrow$ (4): By Proposition \ref{H-sober charact-bounded} and Lemma \ref{Ps bounded=HIP}, $X$ is Smyth ${\rm H}$-complete. Let $\mathcal A\in {\rm H} (P_S(X))$ and $C\in \Gamma (X)$. Obviously, $\ua \left(C\cap\bigcap\mathcal A\right)\subseteq\bigcap\limits_{K\in \mathcal A}\ua (C\cap K)$. On the other hand, if $x\not\in \ua \left(C\cap\bigcap\mathcal A\right)$, then $\da x\cap C\cap \bigcap\mathcal A=\emptyset$, and hence $\bigcap\mathcal A\subseteq X\setminus \da x\cap C$. By Theorem \ref{super H-sober charact H-sets}, there is $G\in \mathcal A$ with $G\subseteq X\setminus \da x\cap C$ or, equivalently, $\da x\cap C\cap G=\emptyset$. Thus $x\not\in \ua (C\cap G)$, and whence $x\not\in \bigcap\limits_{K\in \mathcal A}\ua (C\cap K)$. Therefore, $\bigcap\limits_{K\in \mathcal A}\ua (C\cap K)\subseteq\ua \left(C\cap\bigcap\mathcal A\right)$. The equation $\ua \left(C\cap\bigcap\mathcal A\right)=\bigcap\limits_{K\in \mathcal A}\ua (C\cap K)$ thus holds.

(4) $\Rightarrow$ (5): Trivial.

(5) $\Rightarrow$ (1): Suppose that $\mathcal A\in {\rm H}(P_S(X))$ and $U\in \mathcal O(X)$ with $\bigcap \mathcal A \subseteq U$. If $K\not\subseteq U$ for each $K\in \mathcal A$, then by Lemma \ref{t Rudin}, $X\setminus U$ contains a minimal irreducible closed subset $C$ that still meets all members of $\mathcal{K}$. By condition (5), we have $\{\ua (C\cap K) : K\in \mathcal A\}\in {\rm H}(P_S(X))$, and hence by condition (5), $\emptyset \neq \bigcap\limits_{K\in\mathcal A}\ua (C\cap K)=\ua \left(C\cap\bigcap \mathcal A\right)=\emptyset$ since $\bigcap \mathcal A \subseteq U\subseteq X\setminus C$, a contradiction. Thus $X$ is super {\rm H}-sober by Theorem \ref{super H-sober charact H-sets}.

\end{proof}

By Remark \ref{C D property M}, Lemma \ref{R has property M}, Corollary \ref{WF=Smyth d-space}, Theorem \ref{Heckman-Keimel theorem} and Theorem \ref{super H-sober cha equa H-set}, we get the following two corollaries.

\begin{corollary}\label{WF charact-bounded} \emph{(\cite{xu-shen-xi-zhao1})} For a $T_0$ space $X$, the following conditions are equivalent:
\begin{enumerate}[\rm (1)]
	        \item $X$ is well-filtered.
            \item  $X$ has $\mathbf{FIP}$ \emph{(}especially, $\mk (X)$ is a dcpo\emph{)}, and $\ua_{\mk (X)} \left(\mathcal C\cap\bigcap\limits_{K\in \mathcal K} \ua_{\mk (X)} K\right)=\bigcap\limits_{K\in \mathcal K}\ua_{\mk (X)} (\mathcal C\cap \ua_{\mk (X)} K)$ for any $\mathcal K\in \mathcal D(\mk (X))$ and $\mathcal C\in \Gamma(P_S(X))$.
            \item  $X$ has $\mathbf{FIP}$ \emph{(}especially, $\mk (X)$ is a dcpo\emph{)}, and $\ua_{\mk (X)} \left(\mathcal C\cap\bigcap\limits_{K\in \mathcal K} \ua_{\mk (X)} K\right)=\bigcap\limits_{K\in \mathcal K}\ua_{\mk (X)} (\mathcal C\cap \ua_{\mk (X)} K)$ for any $\mathcal K\in \mathcal D(\mk (X))$ and $\mathcal C\in \ir_c(P_S(X))$.
            \item  $X$ has $\mathbf{FIP}$ \emph{(}especially, $\mk (X)$ is a dcpo\emph{)}, and $\ua \left(C\cap\bigcap\mathcal K\right)=\bigcap\limits_{K\in \mathcal K}\ua (C\cap K)$ for any $\mathcal K\in \mathcal D(\mk (X))$ and $C\in \Gamma (X)$.
            \item  $X$ has $\mathbf{FIP}$ \emph{(}especially, $\mk (X)$ is a dcpo\emph{)}, and $\ua \left(C\cap\bigcap\mathcal K\right)=\bigcap\limits_{K\in \mathcal K}\ua (C\cap K)$ for any $\mathcal K\in \mathcal D(\mk (X))$ and $C\in \ir_c (X)$.
\end{enumerate}
\end{corollary}

\begin{corollary}\label{super sober charact-bounded} \emph{(\cite{xu-shen-xi-zhao1})} For a $T_0$ space $X$, the following conditions are equivalent:
\begin{enumerate}[\rm (1)]
	        \item $X$ is sober.
            \item  $X$ has $\mathbf{RIP}$ \emph{(}especially, $\mk (X)$ is irreducible complete\emph{)}, and $\ua_{\mk (X)} \left(\mathcal C\cap\bigcap\limits_{K\in \mathcal A} \ua_{\mk (X)} K\right)=\bigcap\limits_{K\in \mathcal A}\ua_{\mk (X)} (\mathcal C\cap \ua_{\mk (X)} K)$ for any $\mathcal A\in \ir(P_S(X))$ and $\mathcal C\in \Gamma(P_S(X))$.
            \item  $X$ has $\mathbf{RIP}$ \emph{(}especially, $\mk (X)$ is irreducible complete\emph{)}, and $\ua_{\mk (X)} \left(\mathcal C\cap\bigcap\limits_{K\in \mathcal A} \ua_{\mk (X)} K\right)=\bigcap\limits_{K\in \mathcal A}\ua_{\mk (X)} (\mathcal C\cap \ua_{\mk (X)} K)$ for any $\mathcal A\in \ir_c(P_S(X))$ and $\mathcal C\in \ir_c(P_S(X))$.
            \item  $X$ has $\mathbf{RIP}$ \emph{(}especially, $\mk (X)$ is irreducible complete\emph{)}, and $\ua \left(C\cap\bigcap\mathcal A\right)=\bigcap\limits_{K\in \mathcal A}\ua (C\cap K)$ for any $\mathcal A\in \ir(P_S(X))$ and $C\in \Gamma (X)$.
            \item  $X$ has $\mathbf{RIP}$ \emph{(}especially, $\mk (X)$ is irreducible complete\emph{)}, and $\ua \left(C\cap\bigcap\mathcal A\right)=\bigcap\limits_{K\in \mathcal K}\ua (C\cap K)$ for any $\mathcal A\in \ir(P_S(X))$  and $C\in \ir_c (X)$.
\end{enumerate}
\end{corollary}

The following three corollaries can be directly deduce from Lemma \ref{COMPminimalset}, Theorem \ref{super H-sober cha equa H-set}, Corollary \ref{WF charact-bounded} and Corollary \ref{super sober charact-bounded}.

\begin{corollary}\label{super H-sober min equa} Let $H : \mathbf{Top}_0 \longrightarrow \mathbf{Set}$ be an R-subset system having property M, $X$ a $T_0$ space and $\mathcal A\in H(P_S(X))$. Then $C=\bigcap \mathcal A\in \mk (X)$, and for each $c\in \mathrm{min}(C)$, $\bigcap\limits_{K\in \mathcal A} \ua (\da c\cap K)=\ua \left(\da c\cap\bigcap \mathcal A\right)=\ua c$.
\end{corollary}

\begin{corollary}\label{sober min equa} Let $X$ be a sober space and $\mathcal A\subseteq \ir (P_S(X))$. Then $C=\bigcap \mathcal A\in \mk (X)$, and for each $c\in \mathrm{min}(C)$, $\bigcap\limits_{K\in \mathcal A} \ua (\da c\cap K)=\ua \left(\da c\cap\bigcap \mathcal A\right)=\ua c$.
\end{corollary}

\begin{corollary}\label{WF min equa}\emph{(\cite{wu-xi-xu-zhao-19})} Let $X$ be a well-filtered space and $\mathcal K\in \mathcal D(\mk (X))$. Then $C=\bigcap \mathcal K\in \mk (X)$, and for each $c\in \mathrm{min}(C)$, $\bigcap\limits_{K\in \mathcal K} \ua (\da c\cap K)=\ua \left(\da c\cap\bigcap \mathcal K\right)=\ua c$.
\end{corollary}

\begin{theorem}\label{super H-sober mapping} Let $H : \mathbf{Top}_0 \longrightarrow \mathbf{Set}$ be an R-subset system and $X$ a $T_0$ space. Then the following conditions are equivalent:
\begin{enumerate}[\rm (1)]
		\item $X$ is super H-sober.
        \item For every continuous mapping $f:X\longrightarrow Y$ to a $T_0$ space $Y$ and any $\mathcal A\in H(P_S(X))$, $\ua f\left(\bigcap\mathcal A\right)=\bigcap_{K\in\mathcal A}\ua f(K)$,
        \item For every continuous mapping $f:X\longrightarrow Y$ to an H-sober space $Y$ and any $\mathcal A\in H(P_S(X))$, $\ua f\left(\bigcap\mathcal A\right)=\bigcap_{K\in\mathcal A}\ua f(K)$.
        \item For every continuous mapping $f:X\longrightarrow Y$ to a super H-sober space $Y$ and any $\mathcal A\in H(P_S(X))$, $\ua f\left(\bigcap\mathcal A\right)=\bigcap_{K\in\mathcal A}\ua f(K)$.
        \item For every continuous mapping $f:X\longrightarrow Y$ to a sober space $Y$ and any $\mathcal A\in H(P_S(X))$, $\ua f\left(\bigcap\mathcal A\right)=\bigcap_{K\in\mathcal A}\ua f(K)$.

\end{enumerate}
\end{theorem}
\begin{proof} By Theorem \ref{super H-sober is H-sober} and Theorem \ref{Heckman-Keimel theorem}, we only need to prove the equivalences of conditions (1), (2) and (5).

(1) $\Rightarrow$ (2): It needs only to check $\bigcap_{K\in\mathcal A}\ua f(K)\subseteq \ua f\left(\bigcap\mathcal A\right)$. Let $y\in\bigcap_{K\in\mathcal A}\ua f(K)$. Then for each $K\in\mathcal A$, $\overline{\{y\}}\cap f(K)\neq\emptyset$, that is, $K\cap f^{-1}\left(\overline{\{y\}}\right)\neq\emptyset$. Since $X$ is super {\rm H}-sober, $f^{-1}\left(\overline{\{y\}}\right)\cap\bigcap\mathcal A\neq\emptyset$ (for otherwise we have $\bigcap\mathcal A\subseteq X\setminus f^{-1}\left(\overline{\{y\}}\right)$, and hence by Theorem \ref{super H-sober charact H-sets}, $K\subseteq X\setminus f^{-1}\left(\overline{\{y\}}\right)$ for some $K\in\mathcal A$, a contradiction). It follows that $\overline{\{y\}}\cap f\left(\bigcap\mathcal A\right)\neq\emptyset$. This implies that $y\in\ua f\left(\bigcap\mathcal A\right)$. So $\bigcap_{K\in\mathcal A}\ua f(K)\subseteq \ua f\left(\bigcap\mathcal A\right)$.

(2) $\Rightarrow$ (5): Trivial.

(5) $\Rightarrow$ (1):   Let $\eta_X : X \rightarrow X^s$ ($=P_H(\ir_c(X))$) be the canonical topological embedding from $X$ into its soberification.  Suppose that $\mathcal A\in {\rm H}(P_S(X))$ and $U\in \mathcal O(X)$ with $\bigcap \mathcal A \subseteq U$. If $K\not\subseteq U$ for each $K\in \mathcal A$, then by Lemma \ref{t Rudin}, $X\setminus U$ contains a minimal irreducible closed subset $A$ that still meets all members of $\mathcal{A}$. By condition (5) we have $\bigcap\limits_{K\in\mathcal A}\ua_{\ir_c(X)} \eta_X(K)=\ua_{\ir_c(X)} \eta_X\left(\bigcap\mathcal A\right)\subseteq \ua_{\ir_c(X)}\eta_X(U)=\Diamond_{\ir_c(X)} U$. Clearly, $A\in \bigcap\limits_{K\in\mathcal A}\ua_{\ir_c(X)} \eta_X(K)$, and whence $A\in \Diamond_{\ir_c(X)} U$, that is, $A\cap U\neq\emptyset$, being in contradiction with $A\subseteq X\setminus U$. Thus $X$ is super {\rm H}-sober by Theorem \ref{super H-sober charact H-sets}.
\end{proof}

By Lemma \ref{Ps functor}, Lemma \ref{sups in Smyth}, Remark \ref{H-S H-C H-B}, Lemma \ref{Scott H-cont} and Lemma \ref{Ps bounded=HIP}, Theorem \ref{super H-sober mapping} can be restated as the following one.

\begin{theorem}\label{super H-sober charact-mapping-1} Let $H : \mathbf{Top}_0 \longrightarrow \mathbf{Set}$ be an R-subset system and $X$ a $T_0$ space. Then the following conditions are equivalent:
\begin{enumerate}[\rm (1)]
		\item $X$ is super $H$-sober.
        \item $X$ is Smyth $H$-complete, and for any continuous mapping $f:X\longrightarrow Y$ to a $T_0$ space $Y$, $ P_S(f): P_S(X) \longrightarrow P_S(Y)$ is Scott H-continuous.
        \item $X$ is Smyth $H$-complete, and for any continuous mapping $f:X\longrightarrow Y$ to an H-sober space $Y$, $ P_S(f): P_S(X) \longrightarrow P_S(Y)$ is Scott H-continuous.
        \item $X$ is Smyth $H$-complete, and for any continuous mapping $f:X\longrightarrow Y$ to a super H-sober space $Y$, $ P_S(f): P_S(X) \longrightarrow P_S(Y)$ is Scott H-continuous.
            \item $X$ is Smyth $H$-complete, and for any continuous mapping $f:X\longrightarrow Y$ to a sober space $Y$, $ P_S(f): P_S(X) \longrightarrow P_S(Y)$ is Scott H-continuous.
\end{enumerate}
\end{theorem}

By Corollary \ref{WF=Smyth d-space}, Theorem \ref{Heckman-Keimel theorem}, Theorem \ref{super H-sober mapping} and Theorem \ref{super H-sober charact-mapping-1}, we have the following four corollaries.

\begin{corollary}\label{WF mapping} \emph{(\cite{xu-shen-xi-zhao1})} For a $T_0$ space $X$, the following conditions are equivalent:
\begin{enumerate}[\rm (1)]
		\item $X$ is well-filtered.
        \item For every continuous mapping $f:X\longrightarrow Y$ from $X$ to a $T_0$ space $Y$ and any $\mathcal K\in \mathcal D(\mk (X))$, $\ua f\left(\bigcap\mathcal K\right)=\bigcap_{K\in\mathcal K}\ua f(K)$.
        \item For every continuous mapping $f:X\longrightarrow Y$ from $X$ to a well-filtered space $Y$ and any $\mathcal K\in \mathcal D(\mk (X))$, $\ua f\left(\bigcap\mathcal K\right)=\bigcap_{K\in\mathcal K}\ua f(K)$.
        \item For every continuous mapping $f:X\longrightarrow Y$ from $X$ to a sober space $Y$ and any $\mathcal K\in \mathcal D(\mk (X))$, $\ua f\left(\bigcap\mathcal K\right)=\bigcap_{K\in\mathcal K}\ua f(K)$.
\end{enumerate}
\end{corollary}

\begin{corollary}\label{WF charact-mapping-1} For a $T_0$ space $X$, the following conditions are equivalent:
\begin{enumerate}[\rm (1)]
		\item $X$ is well-filtered.
        \item $\mk (X)$ is a dcpo, and for any continuous mapping $f:X\longrightarrow Y$ to a $T_0$ space $Y$, $ P_S(f): P_S(X) \longrightarrow P_S(Y)$ is Scott continuous.
        \item $\mk (X)$ is a dcpo, and for any continuous mapping $f:X\longrightarrow Y$ to a well-filtered space $Y$, $ P_S(f): P_S(X) \longrightarrow P_S(Y)$ is Scott continuous.
            \item $\mk (X)$ is a dcpo, and for any continuous mapping $f:X\longrightarrow Y$ to a sober space $Y$, $ P_S(f): P_S(X) \longrightarrow P_S(Y)$ is Scott continuous.
\end{enumerate}
\end{corollary}

\begin{corollary}\label{super sober mapping} \emph{(\cite{xu-shen-xi-zhao1})} For a $T_0$ space $X$, the following conditions are equivalent:
\begin{enumerate}[\rm (1)]
		\item $X$ is sober.
        \item For every continuous mapping $f:X\longrightarrow Y$ to a $T_0$ space $Y$ and any $\mathcal A\in \ir(P_S(X))$, $\ua f\left(\bigcap\mathcal K\right)=\bigcap_{K\in\mathcal A}\ua f(K)$.
        \item For every continuous mapping $f:X\longrightarrow Y$ to a sober space $Y$ and any $\mathcal A\in \ir(P_S(X))$, $\ua f\left(\bigcap\mathcal K\right)=\bigcap_{K\in\mathcal A}\ua f(K)$.
\end{enumerate}
\end{corollary}

\begin{corollary}\label{sober charact-mapping-1} For a $T_0$ space $X$, the following conditions are equivalent:
\begin{enumerate}[\rm (1)]
		\item $X$ is sober.
        \item $\mk (X)$ is irreducible complete, and for any continuous mapping $f:X\longrightarrow Y$ to a $T_0$ space $Y$ and any $\mathcal A\in \ir (P_S(X))$, $ P_S(f)(\bigvee_{\mk (X)}\mathcal A)=\bigvee_{\mk (Y)}P_S(f)(\mathcal A)$.
          \item $\mk (X)$ is irreducible complete, and for any continuous mapping $f:X\longrightarrow Y$ to a sober space $Y$ and any $\mathcal A\in \ir (P_S(X))$, $ P_S(f)(\bigvee_{\mk (X)}\mathcal A)=\bigvee_{\mk (Y)}P_S(f)(\mathcal A)$.
\end{enumerate}
\end{corollary}

\section{H-Rudin sets and H-sober determined sets}

\begin{definition}\label{H-Rudin set HD-set}
		Let ${\rm H} : \mathbf{Top}_0 \longrightarrow \mathbf{Set}$ be an R-subset system, $X$ a $T_0$ space and $A$ a nonempty subset of $X$.

\begin{enumerate}[\rm (1)]
\item $A$ is said to have \emph{H}-\emph{Rudin property}, if there exists $\mathcal K\in{\rm H}(P_S(X))$ such that $\overline{A}\in m(\mathcal K)$, that is,  $\overline{A}$ is a minimal closed set that intersects all members of $\mathcal K$. The $\mathcal D$-Rudin sets are shortly called \emph{Rudin sets} (see \cite{Shenchon, xu-shen-xi-zhao1}). Let ${\rm H}^R(X)=\{A\subseteq X : A \mbox{{~has H-Rudin property}}\}$. The sets in ${\rm H}^R(X)$ will also be called \emph{H-Rudin sets}.
    \item $A$ is called \emph{super H}-\emph{sober determined}, if for any continuous mapping $ f:X\longrightarrow Y$
to a super {\rm H}-sober space $Y$, there exists a unique $y_A\in Y$ such that $\overline{f(A)}=\overline{\{y_A\}}$. Denote by ${\rm H}^D(X)$ the set of all super {\rm H}-sober determined subsets of $X$. Let ${\rm H}^D_c(X)={\rm H}^D(X)\bigcap\Gamma (X)$.
\end{enumerate}
\end{definition}

Clearly, $A\in {\rm H}^R(X)$ (resp., $A\in {\rm H}^D(X)$) if{}f $\overline{A}\in {\rm H}^R(X)$ (resp., $\overline{A}\in {\rm H}^D(X)$).

\begin{lemma}\label{super H-Rudin functor}
	Let $H : \mathbf{Top}_0 \longrightarrow \mathbf{Set}$ be an R-subset system having property M. Then $H^R : \mathbf{Top}_0 \longrightarrow \mathbf{Set}$ is an R-subset system, where for any continuous mapping $ f:X\longrightarrow Y$ in $\mathbf{Top}_0$, $H^R(f) : H^R (X)\longrightarrow H^R (Y)$ is defined by $H^R(f)(A)=f(A)$ for each $A\in H^R (X)$.
\end{lemma}
\begin{proof} Suppose that $f : X\longrightarrow Y$ is a continuous mapping in $\mathbf{Top}_0$. We need to show $f({\rm H}^R(X))\subseteq {\rm H}^R(Y)$. Let $A\in H^R(X)$. Then there exists $\mathcal K\in{\rm H}(P_S(X))$ such that $\overline{A}\in m(\mathcal K)$. Let $\mathcal{K}_f=\{\ua f(K\cap \overline{A}) : K\in \mathcal K\}$. Then by the property M of {\rm H}, $\mathcal{K}_f\in {\rm H}(P_S(Y))$. For each $K\in \mathcal K$, since $K\cap \overline{A}\neq\emptyset$, we have $\emptyset\neq f(K\cap \overline{A})\subseteq \ua f(K\cap \overline{A})\cap \overline{f(A)}$. So $\overline{f(A)}\in M(\mathcal{K}_f)$. If $B$ is a closed subset of $\overline{f(A)}$ with $B\in M(\mathcal{K}_f)$, then $B\cap\ua f(K\cap \overline{A})\neq\emptyset$ for each $K\in \mathcal K$. So $K\cap \overline{A}\cap f^{-1}(B)\neq\emptyset$ for all $K\in \mathcal K$. It follows from the minimality of $\overline{A}$ that $\overline{A}=\overline{A}\cap f^{-1}(B)$, and consequently, $\overline{f(A)}\subseteq B$. Therefore, $\overline{f(A)}=B$, and hence $\overline{f(A)}\in m(\mathcal K_f)$. Thus $f(A)\in {\rm H}^R(Y)$.
\end{proof}

\begin{lemma}\label{HD functor} Let $H : \mathbf{Top}_0 \longrightarrow \mathbf{Set}$ be an R-subset system. Then $H^D : \mathbf{Top}_0 \longrightarrow \mathbf{Set}$ is an R-subset system, where for any continuous mapping $ f:X\longrightarrow Y$ in $\mathbf{Top}_0$, $H^D(f) : H^D (X)\longrightarrow H^D (Y)$ is defined by $H^D(f)(A)=f(A)$ for each $A\in H^D (X)$.
\end{lemma}

\begin{proof} Suppose that $f : X\longrightarrow Y$ is a continuous mapping in $\mathbf{Top}_0$ and $A\in H^D(X)$. Let $Z$ be a super {\rm H}-sober space and $g:Y\longrightarrow Z$ a continuous mapping.
Since $g\circ f:X\longrightarrow Z$ is continuous and $A\in {\rm H}^D (X)$, there is $z\in Z$ such that $\overline{g(f(A))}=\overline{g\circ f(A)}=\overline{\{z\}}$. Thus $f(A)\in {\rm H}^D (Y)$.
\end{proof}

\begin{proposition}\label{H Hd HR HD ir} Let $H : \mathbf{Top}_0 \longrightarrow \mathbf{Set}$ be an R-subset system and $X$ a $T_0$ space. Then
\begin{enumerate}[\rm (1)]
\item  $H\leq H^d\leq H^D\leq \mathcal R$.
\item $H(X)\subseteq H^R(X)\subseteq \ir (X)$.
\item If H has property M, then $H\leq H^R\leq H^D\leq \mathcal R$.
\end{enumerate}
\end{proposition}

\begin{proof} (1): By Lemma \ref{H Hd ir}, Corollary \ref{Hd R-subset system}, Theorem \ref{super H-sober is H-sober} and Lemma \ref{HD functor}, we have ${\rm H}\leq{\rm H}^d\leq {\rm H}^D$. By Corollary \ref{Hd R-subset system}, Theorem \ref{Heckman-Keimel theorem} and $H\leq \mathcal R$, we have $H^D\leq \mathcal R^D=\mathcal R^d=\mathcal R$.

(2): For $A\in {\rm H}(X)$, let $\mathcal K_A=\xi_X(A)=\{\ua a : a\in A\}$. Then by Remark \ref{xi continuous}, $\mathcal K_A\in {\rm H}(P_S(X))$ and $\overline{A}\in M(\mathcal K_A)$. If $B\in M(\mathcal K_A)$, then $a\in B$ for all $a\in A$, and hence $\overline{A}\subseteq B$. So $\overline{A}\in m(\mathcal K_A)$. Thus $A\in {\rm H}^R(X)$. Now we show ${\rm H}^R(X)\subseteq \ir (X)$. Let $C\in {\rm H}^R(X)$. Then there exists $\mathcal K\in{\rm H}(P_S(X))$ such that $\overline{C}\in m(\mathcal K)$. For $B_1, B_2\in \Gamma (X)$ with $\overline{C}=B_1\cup B_2$, if $B_1\neq \overline{C}$ and $B_2\neq \overline{C}$, then by the minimality of $\overline{A}$, there exist $K_1, K_2\in\mathcal K$ with $K_1\cap B_1=\emptyset$ and $K_2\cap B_2=\emptyset$, that is, $K_1\in \Box (X\setminus B_1)$ and $K_2\in \Box (X\setminus B_2$, and hence by $\mathcal K\in {\rm H}(P_S(X))\subseteq \ir (P_S(X))$, $\mathcal K\cap \Box (X\setminus C)=\mathcal K\cap \Box (X\setminus B_1)\cap \Box (X\setminus B_2)\neq\emptyset$, which is in contradiction with $\overline{C}\in m(\mathcal K)$. Therefore, $C\in \ir (X)$.

(3): Since {\rm H} has property M, by Lemma \ref{super H-Rudin functor}, ${\rm H}^R$ is an R-subset system. Now we show ${\rm H}^R\leq {\rm H}^D$. Let $A\in {\rm H}^R(X)$ and $ f:X\longrightarrow Y$ a continuous mapping to a super {\rm H}-sober space $Y$. Then there exists $\mathcal K\in{\rm H}(P_S(X))$ such that $\overline{A}\in m(\mathcal K)$, and hence by the property M of {\rm H},we have $\mathcal{K}_f=\{\ua f(K\cap \overline{A}) : K\in \mathcal K\}\in {\rm H}(P_S(Y))$ and $\overline{f(A)}\in M(\mathcal{K}_f)$. By Theorem \ref{super H-sober charact H-sets}, we have $\bigcap\mathcal{K}_f \cap \overline{f(A)}\neq \emptyset$. Select a $y_A\in \bigcap\mathcal{K}_f \cap \overline{f(A)}$. Then $\overline{\{y_A\}}\subseteq\overline{f(A)}$, and $K\cap \overline{A}\cap f^{-1}(\overline{\{y_A\}})\neq\emptyset$ for all $K\in \mathcal K$. It follows that $\overline{A}=\overline{A}\cap f^{-1}(\overline{\{y_A\}})$ by the minimality of $\overline{A}$, and consequently, $\overline{f(A)}\subseteq \overline{\{y_A\}}$. Therefore, $\overline{f(A)}=\overline{\{y_A\}}$. The uniqueness of $y_A$ follows from the $T_0$ separation of $Y$. Thus $A\in {\rm H}^D(X)$.

\end{proof}

\begin{lemma}\label{LHC directed} \emph{(\cite{E_20182})}
	Let $X$ be a locally hypercompact $T_0$ space and $A\in\ir(X)$. Then there exists a directed subset $D\subseteq\da A$ such that $\overline{A}=\overline{D}$.
\end{lemma}

 By Lemma \ref{H Hd ir} and Lemma \ref{LHC directed}, we get the following result.

 \begin{corollary}\label{LHC is d-D}\emph{(\cite{xu-shen-xi-zhao1})} For a locally hypercompact $T_0$ space $X$, $\ir_c (X)=\mathcal D^d_c(X)=\mathcal D_c(X)$.
 \end{corollary}

By Corollary \ref{LHC is d-D}, every locally hypercompact $d$-space is sober. For locally compact $T_0$ spaces and core compact $T_0$ spaces, we have the following similar result.

\begin{proposition}\label{LC rudin}\emph{(\cite{xu-shen-xi-zhao1})} Let $X$ be a $T_0$ space. Then
	\begin{enumerate}[\rm (1)]
\item If $X$ is locally compact, then $\ir_c (X)=\mathcal D^D_c(X)=\mathcal D^R_c(X)$.
\item If $X$ is core compact, then $\ir_c (X)=\mathcal D^D_c(X)$.
\end{enumerate}
\end{proposition}

The following result shows that for an R-subset system H having property M, the super H-sobriety is indeed a special type of $\mathrm{H}$-sobriety.

\begin{theorem}\label{super H-sober HD-sober} Let $H : \mathbf{Top}_0 \longrightarrow \mathbf{Set}$ be an R-subset system having property M and $X$ a $T_0$ space. Then the following conditions are equivalent:
\begin{enumerate}[\rm (1)]
\item  $X$ is super H-sober \emph{(}i.e., $P_S(X)$ is H-sober\emph{)}.
\item $X$ is $H^D$-sober.
\item $X$ is $H^R$-sober.
\end{enumerate}
\end{theorem}

\begin{proof} First, by Corollary \ref{Hd R-subset system}, Proposition \ref{H Hd HR HD ir} and the property M of {\rm H}, we have that ${\rm H}^R$ and ${\rm H}^D$ are R-subset systems, ${\rm H}\leq H^d\leq {\rm H}^D \leq \mathcal R$ and ${\rm H}\leq {\rm H}^R\leq{\rm H}^D$.

(1) $\Rightarrow$ (2): Assume that $X$ is super {\rm H}-sober. For any $A\in{\rm H}^D(X)$, since the identity $id_X : X\longrightarrow X$ is continuous, there is a unique $x\in X$ such that $\overline{A}=\overline{\{x\}}$. Thus $X$ is ${\rm H}^D$-sober.

(2) $\Rightarrow$ (3): By ${\rm H}^R\leq{\rm H}^D$.

(3) $\Rightarrow$ (1): Suppose that $\mathcal A\in {\rm H}(P_S(X))$ and $U\in \mathcal O(X)$ with $\bigcap\mathcal A \subseteq U$. If $\mathcal K\not\subseteq U$ for all $K\in \mathcal A$, then by Lemma \ref{t Rudin}, $X\setminus U$ contains a minimal irreducible closed subset $A$ that still meets all members of $\mathcal A$, and hence $A\in {\rm H}^R(X)$. Since $X$ is  ${\rm H}^R$-sober, there is $x\in X$ such that $A=\overline{\{x\}}$. It follows that $\overline{\{x\}}\in m(\mathcal A)$, and consequently, $x\in \bigcap \mathcal A\subseteq U$, being a contradiction with $x\in A\subseteq X\setminus U$. Therefore, $X$ is super {\rm H}-sober by Theorem \ref{super H-sober charact H-sets}.

\end{proof}

By Theorem \ref{super H-sober=PS super H-sober} and Theorem \ref{super H-sober HD-sober}, we get the following corollary.

\begin{corollary}\label{super super H-sober HR-sober} Let $H : \mathbf{Top}_0 \longrightarrow \mathbf{Set}$ be an R-subset system having property M and $X$ a $T_0$ space. Then the following conditions are equivalent:
\begin{enumerate}[\rm (1)]
\item  $X$ is super H-sober.
\item $X$ is $H^R$-sober.
\item $X$ is $H^D$-sober.
\item  $P_S(X)$ is super H-sober.
\item $P_S(X)$ is $H^R$-sober.
\item $P_S(X)$ is $H^D$-sober.
\end{enumerate}
\end{corollary}

By Remark \ref{C D property M}, Corollary \ref{WF=Smyth d-space} and Theorem \ref{super H-sober HD-sober}, we have the following result.

\begin{corollary}\label{WF=DD-sober} \emph{(\cite{xu-shen-xi-zhao1, xuxizhao})}  For a $T_0$ space $X$, the following conditions are equivalent:
\begin{enumerate}[\rm (1)]
\item  $X$ is well-filtered.
\item $X$ is $\mathcal D^D$-sober.
\item $X$ is $\mathcal D^R$-sober.
\item  $P_S(X)$ is well-filtered.
\item $P_S(X)$ is $\mathcal D^R$-sober.
\item $P_S(X)$ is $\mathcal D^D$-sober.
\end{enumerate}
\end{corollary}

The following result follows directly from Theorem \ref{SoberLC=CoreC}, Proposition \ref{LC rudin} and Corollary \ref{WF=DD-sober}.

\begin{theorem}\label{corecwellf sober}\emph{(\cite{Lawson-Xi, xu-shen-xi-zhao1})} Let $X$ be a a core compact well-filtered space. Then $X$ is sober. Therefore, $X$ is a locally compact sober space.
\end{theorem}

In \cite{Hofmann-Lawson} (or \cite[Exercise V-5.25]{redbook}), Hofmann and Lawson constructed a second-countable core compact $T_0$ space $X$ in which every compact subset has empty interior. By Theorem \ref{corecwellf sober}, $X$ is not well-filtered.

\begin{example}\label{examp1}
	Let $X$ be a countably infinite set and $X_{cof}$ the space equipped with the \emph{co-finite topology} (the empty set and the complements of finite subsets of $X$ are open). Then
\begin{enumerate}[\rm (a)]
    \item $\Gamma(X_{cof})=\{\emptyset, X\}\cup X^{(<\omega)}$, $X_{cof}$ is $T_1$ and hence a $d$-space.
    \item $\mk (X_{cof})=2^X\setminus \{\emptyset\}$.
    \item $X_{cof}$ is locally compact and first countable.
    \item $\mathcal S (X_{cof})=\mathcal D (X_{cof})=\{\{x\} : x\in X\}$,
    \item $\ir_c (X_{cof})=\mathcal D^R(X_{cof})=\mathcal D^D(X_{cof})=\{X\}\cup \{\{x\} : x\in X\}\neq\mathcal D (X_{cof})$. In fact, let $\mathcal{K}_X=\{X\setminus F : F\in X^{(<\omega)}\}$. Then $\mathcal{K}_X\in \mathcal D(P_S(X_{cof}))$, $X\in M(\mathcal{K}_X)$ and $\bigcap \mathcal{K}_X=\emptyset$. For any $A\in \Gamma (X_{cof})$, if $A\neq X$, then $A$ is finite and hence $A\not\in M(\mathcal{K}_X)$ because $A\cap (X\setminus A)=\emptyset$. Thus $X\in \mathcal D^R(X_{cof})$, but $X\not\in \mathcal D(X_{cof})=\mathcal D_c(X_{cof})$.
\item $X_{cof}$ is not well-filtered by Corollary \ref{WF=Smyth d-space} or Theorem \ref{super H-sober HD-sober}.
\end{enumerate}
\end{example}

\begin{example}\label{examp2}
	Let $L$ be the complete lattice constructed by Isbell in \cite{isbell}. Then by \cite[Corollary 3.2]{Xi-Lawson-2017} (or Corollary \ref{complete Soctt is WF} below), $X=\Sigma~\!\!L$ is well-filtered. Note that $\Sigma L$ is not sober. Thus by Theorem \ref{super H-sober HD-sober}, $\mathcal D^R(X)\neq\ir (X)$ and $\mathcal D^D\neq\ir (X)$.
\end{example}

\begin{example}\label{examp3}(\cite{xuzhao}) Let $X$ be an uncountably infinite set and $X_{coc}$ the space equipped with \emph{the co-countable topology} (the empty set and the complements of countable subsets of $X$ are open sets). Then
\begin{enumerate}[\rm (a)]
    \item $\Gamma (X_{coc})=\{\emptyset, X\}\cup X^{(\leqslant\omega)}$, $X_{coc}$ is $T_1$ and hence a $d$-space.
    \item $\mk (X_{coc})=X^{(<\omega)}\setminus \{\emptyset\}$ and $\ii~\!K=\emptyset$ for all $K\in \mk (X_{coc})$.
    \item $X_{coc}$ is not locally compact and not first countable.
    \item $\ir (X_{coc})=\ir_c(X_{coc})=\{X\}\cup \{\{x\} : x\in X\}$, $\mathcal D^D(X_{coc})=\mathcal D^R(X_{coc})=\mathcal D(X_{coc})=\mathcal S(X_{coc})=\{\{x\} : x\in X\}$. Therefore, $\ir (X_{coc})\neq \mathcal D^D(X_{coc})$.
    \item $X_{coc}$ is well-filtered by Corollary \ref{WF=DD-sober}, but it not sober.
\end{enumerate}
\end{example}

By Corollary \ref{Hdd=Hd}, Proposition \ref{H Hd HR HD ir} and Corollary \ref{super super H-sober HR-sober}, we have the following result.

\begin{corollary}\label{HDD=HD}  Let $H : \mathbf{Top}_0 \longrightarrow \mathbf{Set}$ be an R-subset system having property M. Then
\begin{enumerate}[\rm (1)]
\item $D : \mathcal H\longrightarrow \mathcal H$, $H\mapsto H^D$, is a closure operator.
\item $d\leq R\leq D$, where $R : \mathcal H\longrightarrow \mathcal H$ is defined by $R(H)=H^R(H)$ for each $H\in \mathcal H$.
\end{enumerate}
\end{corollary}

In the following, we investigate some basic properties of super {\rm H}-sober spaces.

\begin{proposition}\label{super H-sober closed} Let $X$ be a super H-sober space.
\begin{enumerate}[\rm (1)]
\item If $A$ is a nonempty closed subspace of $X$, then $A$ is super H-sober.
\item If $U$ is a nonempty saturated subspace of $X$, then $U$ is super H-sober.
\end{enumerate}
\end{proposition}
\begin{proof} (1): Let $j_A : A \longrightarrow X$ be the inclusion mapping and $\mathcal B\in {\rm H}(P_S(A))$. Then by Lemma \ref{Ps functor}, $P_S(j_A)(\mathcal B)=\{\ua_X K : K\in \mathcal B\}\in {\rm H}(P_S(X))$, and hence there is $G\in \mk (X)$ such that $\cl_{P_S(X)} P_S(j_A)(\mathcal B)=\cl_{P_S(X)}\{G\}$. Clearly, $G\cap A\in \mk (A)$ (for otherwise, $G\cap A=\emptyset$ implies $G\in \Box_{\mk (X)} (X\setminus A)$, and hence $K\subseteq \ua_X K\subseteq X\setminus A$ for some $K\in \mathcal B$, a contradiction). Now we show that $\cl_{P_S(A)} \mathcal B=\cl_{P_S(A)}\{G\cap A\}$. For any $U\in \mathcal O(X)$, we have
$$\begin{array}{lll}
\emptyset&\neq&\mathcal B\cap \Box_{\mk (A)} (U\cap A)\\
&\Leftrightarrow&\exists K\in \mathcal B\mbox{~with~}K\subseteq U\cap A\\
&\Leftrightarrow&P_S(j_A)(\mathcal B)\cap \Box_{\mk (X)} U\neq \emptyset\\
&\Leftrightarrow&G\subseteq U\\
&\Rightarrow&G\cap A\subseteq U\cap A\\
&\Leftrightarrow&\{G\cap A\}\cap \Box_{\mk (A)} (U\cap A)\neq\emptyset\\.
\end{array}$$
On the other hand, if $G\cap A\subseteq U$, then $G\subseteq U\cup (X\setminus A)$, and hence by $\cl_{P_S(X)} P_S(j_A)(\mathcal B)=\cl_{P_S(X)}\{G\}$, there is $K\in \mathcal B$ such that $\ua_X K\subseteq U\cup (X\setminus A)$. It follows that $K=A\cap \ua_X K\subseteq U$. Since $K\subseteq A$, we have $K\subseteq U\cap A$, and whence $\mathcal B\cap \Box_{\mk (A)} (U\cap A)\neq \emptyset$. Consequently, $\mathcal B\cap \Box_{\mk (A)} (U\cap A)\neq \emptyset \Leftrightarrow \{G\cap A\}\cap \Box_{\mk (A)} (U\cap A)\neq\emptyset$. Thus $\cl_{P_S(A)} \mathcal B=\cl_{P_S(A)}\{G\cap A\}$. Therefore, $A$ is super {\rm H}-sober.

(2): Let $\mathcal A\in {\rm H}(P_S(U))$. For each $K\in \mathcal A$, since $U=\ua_X U$ and $K=\ua_U K$, we have $K=\ua_X K$. Let $j_U : U \longrightarrow X$ be the inclusion mapping. Then $\mathcal A=P_S(j_U)(\mathcal A)\in {\rm H}(P_S(X))$, and whence there is $H\in \mk (X)$ such that $\cl_{P_S(X)}\mathcal A=\cl_{P_S(X)}\{H\}$. Now we show $H\subseteq U$. Assume, on the contrary, that there is $h\in H\setminus U$, then $U\subseteq X\setminus \cl_X \{h\}$ since $U=\ua_X U$. By $\mathcal A\subseteq \Box_{\mk (X)} U\subseteq \Box_{\mk (X)} (X\setminus \cl_X \{h\})$, we have $\{H\}\cap \Box_{\mk (X)}(X\setminus \cl_X \{h\})$ or, equivalently, $H\subseteq X\setminus \cl_X \{h\}$, a contradiction. So $H\in \mk (U)$. For any $W\in \mathcal O(X)$, it is easy to check that $\mathcal A\cap \Box_{\mk (U)}W\cap U\neq\emptyset \Leftrightarrow H\in \Box_{\mk (U)}W\cap U$. Therefore, $\cl_{P_S(U)} \mathcal A=\cl_{P_S(U)}\{H\}$. Thus $U$ is super {\rm H}-sober.
\end{proof}

By Corollary \ref{Ps retract} and Proposition \ref{H-sober retract}, we get the following result.

\begin{proposition}\label{super H-sober retract} A retract of a super  H-sober space is super H-sober.
\end{proposition}

By Theorem \ref{H-sober prod}, Theorem \ref{H-sober function space}, Proposition \ref{H-sober equalizer} and Theorem \ref{super H-sober HD-sober}, we have the following three results.

\begin{theorem}\label{super H-sober prod}
	Let $H : \mathbf{Top}_0 \longrightarrow \mathbf{Set}$ be an R-subset system having property M and  $\{X_i:i\in I\}$ a family of $T_0$ spaces. Then the following two conditions are equivalent:
	\begin{enumerate}[\rm(1)]
		\item The product space $\prod_{i\in I}X_i$ is super H-sober.
		\item For each $i \in I$, $X_i$ is super H-sober.
	\end{enumerate}
\end{theorem}

\begin{theorem}\label{super H-sober function space}
	Let $H : \mathbf{Top}_0 \longrightarrow \mathbf{Set}$ be an R-subset system having property M. If $X$ is a $T_0$ space and $Y$ a super H-sober space, then the function space $\mathbf{Top}_0(X, Y)$ equipped with the topology of pointwise convergence is super H-sober.
\end{theorem}

\begin{proposition}\label{super H-sober equalizer} Let $H : \mathbf{Top}_0 \longrightarrow \mathbf{Set}$ be an R-subset system having property, $X$ a super H-sober space and $Y$ a $T_0$ space. If $f, g: X \longrightarrow Y$ are continuous, then the equalizer
$E(f, g)=\{x\in X : f(x)=g(x)\}$ \emph{(}as a subspace of $X$\emph{)} is super H-sober.
\end{proposition}

By Theorem \ref{super H-sober prod} and  Proposition \ref{super H-sober equalizer}, we get the following corollary.

\begin{corollary}\label{super H-sober complete} For any R-subset system $H : \mathbf{Top}_0 \longrightarrow \mathbf{Set}$ having property M, $\mathbf{SH}$-$\mathbf{Sob}$ is complete.
\end{corollary}

The following conclusion follows directly from Remark \ref{C D property M}, Corollary \ref{WF=Smyth d-space} and Theorem \ref{super H-sober prod}.

\begin{corollary}\label{WF prod} \emph{(\cite{Shenchon, xu-shen-xi-zhao1, xu-shen-xi-zhao2})}
	Let $\{X_i:i\in I\}$ be a family of $T_0$ spaces. Then the following two conditions are equivalent:
	\begin{enumerate}[\rm(1)]
		\item The product space $\prod_{i\in I}X_i$ is well-filtered.
		\item For each $i \in I$, $X_i$ is well-filtered.
	\end{enumerate}
\end{corollary}

By Remark \ref{C D property M}, Theorem \ref{super H-sober function space} and Proposition \ref{super H-sober equalizer}, we have the following two results.

\begin{theorem}\label{WF function space}
	If $X$ is a $T_0$ space and $Y$ a well-filtered space, then the function space $\mathbf{Top}_0(X, Y)$ equipped with the topology of pointwise convergence is well-filtered.
\end{theorem}

\begin{proposition}\label{WF equalizer} Let $X$ be a well-filtered space and $Y$ a $T_0$ space. If $f, g: X \longrightarrow Y$ are continuous, then the equalizer
$E(f, g)=\{x\in X : f(x)=g(x)\}$ \emph{(}as a subspace of $X$\emph{)} is well-filtered.
\end{proposition}

By Corollary \ref{WF prod} and Proposition \ref{WF equalizer}, we get the following corollary.

\begin{corollary}\label{WF cat complete} $\mathbf{Top}_w$ is complete.
\end{corollary}

For a $T_0$ space and $A\subseteq X$, define $\Psi_{\mk (X)}(A)=\{K\in \mk (X) : K\cap A\neq \emptyset\}$. The following theorem provides a new characterization of super {\rm H}-sober spaces.

\begin{theorem}\label{super H-sober Psi} Let $H : \mathbf{Top}_0 \longrightarrow \mathbf{Set}$ be an R-subset system and $X$ a $T_0$ space. Consider the following conditions are equivalent:
 \begin{enumerate}[\rm (1)]
		\item $X$ is super H-sober.
        \item For any $(A, K)\in H^D_c(X)\times\mk (X)$, $\Psi_{\mk (X)}(A)\in \mathrm{Id}(\mk (X))$, $\mathrm{max} (A)\neq\emptyset$ and $\downarrow (A\cap K)\in \Gamma (X)$.
        \item For any $(A, K)\in H^R_c(X)\times\mk (X)$, $\Psi_{\mk (X)}(A)\in \mathrm{Id}(\mk (X))$, $\mathrm{max} (A)\neq\emptyset$ and $\downarrow (A\cap K)\in \Gamma (X)$.
\end{enumerate}
Then \emph{(1)} $\Rightarrow$ \emph{(2)} and \emph{(3)} $\Rightarrow$ \emph{(1)}, and all three conditions are equivalent if $H$ has property M.
\end{theorem}

\begin{proof}  (1) $\Rightarrow$ (2):  Since the identity $id_X : X\longrightarrow X$ is continuous, there is a unique $x\in X$ such that $\overline{A}=\overline{\{x\}}$.  there is $x\in X$ with $A=\overline {\{x\}}$, and hence $\mathrm{max} (A)=\{x\}\neq\emptyset$. Clearly, $\Psi_{\mk (X)}(A)=\{K\in \mk (X) : K\cap \da x\neq \emptyset\}=\{K\in \mk (X) : x\in K\}=\da_{\mk (X)}\ua x$ is a principal ideal of $\mk (X)$ (note that the order on $\mk (X)$ is the reverse inclusion
order). Now we show that $\downarrow (A\cap K)$ is a closed subset of $X$. If $\downarrow x\cap K\neq\emptyset$, then $x\in K$ since $K$ is an upper set. It follows that  $\downarrow (A\cap K)=\downarrow (\downarrow x\cap K)=\downarrow x\in \Gamma (X)$.

(3) $\Rightarrow$ (1): Suppose $\mathcal A\in {\rm H}(P_S(X))$ and $U\in \mathcal O(X)$ with $\bigcap \mathcal A\subseteq U$. If $K\not\subseteq U$ for all $K\in U$, then by Lemma \ref{t Rudin}, $X\setminus U$ contains a minimal irreducible closed subset $A$ that still meets all members of $\mathcal{A}$. For any $\{K_1, K_2\}\subseteq \mathcal A$, since $\Psi_{\mk (X)}(A)\in \mathrm{Id}(\mk (X))$, there is a $K_3\in\Psi_{\mk (X)}(A)$ with $K_3\subseteq K_1\cap K_2$. It follows that $\emptyset \neq A\cap K_3\subseteq \downarrow (A\cap K_1)\cap K_2\neq\emptyset$. By the minimality of $A$, we have $\downarrow (A\cap K_1)=A$ for all $K_1\in\mathcal A$.  By condition (3), $\mathrm{max} (A)\neq \emptyset$. Select an $x\in \mathrm{max} (A)$. Then for each $K\in \mathcal A$, $x\in  \downarrow (A\cap K)$, and consequently, there is $a_k\in A\cap K$ with $x\leq a_k$. By the maximality of $x$ in $A$, we have $x=a_k$. Therefore, $x\in K$ for all $K\in \mathcal A$, and whence $x\in \bigcap \mathcal A \subseteq U\subseteq X\setminus A$, a contradiction. It follows from Theorem \ref{super H-sober charact H-sets} that $X$ is super {\rm H}-sober.

(2) $\Rightarrow$ (3): If H has property M, then by Proposition \ref{H Hd HR HD ir}, we have (2) $\Rightarrow$ (3), and hence all three conditions are equivalent.
\end{proof}

By Lemma \ref{d-space max point} and Theorem \ref{super H-sober Psi}, we get the following result.

\begin{corollary}\label{super H-sober Psi-1} Let $H : \mathbf{Top}_0 \longrightarrow \mathbf{Set}$ be an R-subset system and $X$ a $d$-space space. Consider the following conditions:
 \begin{enumerate}[\rm (1)]
		\item $X$ is super H-sober.
        \item For any $(A, K)\in H^D_c(X)\times\mk (X)$, $\Psi_{\mk (X)}(A)\in \mathrm{Id}(\mk (X))$ and $\downarrow (A\cap K)\in \Gamma (X)$.
        \item For any $(A, K)\in H^R_c(X)\times\mk (X)$, $\Psi_{\mk (X)}(A)\in \mathrm{Id}(\mk (X))$ and $\downarrow (A\cap K)\in \Gamma (X)$.
\end{enumerate}
Then \emph{(1)} $\Rightarrow$ \emph{(2)} and \emph{(3)} $\Rightarrow$ \emph{(1)}, and all three conditions are equivalent if $H$ has property M.
\end{corollary}

\begin{corollary}\label{WF Psi}\emph{(\cite{xuzhao})}
	For any $T_0$ space $X$, the following conditions are equivalent:
 \begin{enumerate}[\rm (1)]
		\item $X$ is well-filtered.
        \item For each $(A, K)\in \mathcal D^D(X)\times\mk (X)$, $\mathrm{max} (A)\neq\emptyset$ and $\downarrow (A\cap K)\in \Gamma (X)$.
        \item For each $(A, K)\in \mathcal D^R(X)\times\mk (X)$, $\mathrm{max} (A)\neq\emptyset$ and $\downarrow (A\cap K)\in \Gamma (X)$.
\end{enumerate}
\end{corollary}

\begin{proof}  (1) $\Rightarrow$ (2): By Corollary \ref{WF=Smyth d-space} and Theorem \ref{super H-sober Psi}.

(2) $\Rightarrow$ (3): By Proposition \ref{H Hd HR HD ir}.

(3) $\Rightarrow$ (1): Suppose that $\mathcal K\in \mathcal D(\mk (X))$ and $U\in \mathcal O(X)$ with $\bigcap \mathcal K \subseteq U$. If $K\not\subseteq U$ for each $K\in \mathcal K$, then by Lemma \ref{t Rudin}, $X\setminus U$ contains a minimal irreducible closed subset $A$ that still meets all members of $\mathcal{K}$, and hence $A\in \mathcal D^R(X)$. For any $\{K_1, K_2\}\subseteq \mathcal K$, we can find $K_3\in \mathcal K$ with $K_3\subseteq K_1\cap K_2$. It follows that $\downarrow (A\cap K_1)\in \mathcal C(X)\Gamma (X)$ and $\emptyset \neq A\cap K_3\subseteq \downarrow (A\cap K_1)\cap K_2\neq\emptyset$. By the minimality of $A$, we have $\downarrow (A\cap K_1)=A$ for all $K_1\in\mathcal K$. Since $\mathrm{max} (A)\neq\emptyset$, we can select an $x\in \mathrm{max} (A)$. For each $K\in \mathcal K$, we have $x\in  \downarrow (A\cap K)$, and consequently, there is $a_k\in A\cap K$ such that $x\leq a_k$. By the maximality of $x$ we have $x=a_k$. Therefore, $x\in K$ for all $K\in \mathcal K$, and so $x\in \bigcap \mathcal K \subseteq U\subseteq X\setminus A$, a contradiction. Thus $X$ is well-filtered.
\end{proof}

By Lemma \ref{d-space max point} and Corollary \ref{WF Psi}, we have the following corollary.

\begin{corollary}\label{WF Psi-1}\emph{(\cite{xuzhao})}
	For a $d$-space $X$, the following conditions are equivalent:
 \begin{enumerate}[\rm (1)]
		\item $X$ is well-filtered.
        \item For each $(A, K)\in \mathcal D^D(X)\times\mk (X)$, $\downarrow (A\cap K)\in \Gamma (X)$.
        \item For each $(A, K)\in \mathcal D^R(X)\times\mk (X)$, $\downarrow (A\cap K)\in \Gamma (X)$.
\end{enumerate}
\end{corollary}

\begin{corollary}\label{xi-lawsonWF}\emph{(\cite{Xi-Lawson-2017})}
	Let $X$ be a $d$-space such that $\downarrow (A\cap K)$ is closed for all
$A\in \Gamma (X)$ and $K\in \mk (X)$. Then $X$ is well-filtered.
\end{corollary}

By Lemma \ref{Scott compact closed}, Corollary \ref{complete lattice Scott compact closed} and Corollary \ref{WF Psi-1} (or Corollary \ref{xi-lawsonWF}), we get the following two conclusions.

\begin{corollary}\label{Scott is WF} For a dcpo $P$, if $\da (\ua x\cap A)\in \Gamma (\Sigma~\!\!P)$ for any $x\in P$ and $A\in \Gamma (\Sigma~\!\!P)$,
	then $\Sigma~\!\!P$ is well-filtered.
\end{corollary}

\begin{corollary}\label{complete Soctt is WF} \emph{(\cite{Xi-Lawson-2017})} For a complete lattice $L$, $\Sigma \!\!~L$ is well-filtered.
\end{corollary}

By Lemma \ref{d-space max point}, Theorem \ref{Heckman-Keimel theorem}, Theorem \ref{H Hd HR HD ir} and Theorem \ref{super H-sober Psi}, we have the following two corollaries.

\begin{corollary}\label{super Psi}
For a $T_0$ space $X$, the following two conditions are equivalent:
 \begin{enumerate}[\rm (1)]
		\item $X$ is sober.
        \item For any $(A, K)\in \ir_c(X)\times\mk (X)$, $\Psi_{\mk (X)}(A)\in \mathrm{Id}(\mk (X))$, $\mathrm{max} (A)\neq\emptyset$ and $\downarrow (A\cap K)\in \Gamma (X)$.
\end{enumerate}
\end{corollary}

\begin{corollary}\label{sober Psi-1}
	For a $d$-space $X$, the following two conditions are equivalent:
 \begin{enumerate}[\rm (1)]
        \item $X$ is sober.
        \item For each $(A, K)\in \ir_c(X)\times\mk (X)$, $\Psi_{\mk (X)}(A)\in \mathrm{Id}(\mk (X))$ and $\downarrow (A\cap K)\in \Gamma (X)$.
\end{enumerate}
\end{corollary}

\begin{remark}\label{sober Psi not sober} Let $L$ be the Isbell's lattice (see Example \ref{examp2}). Then $\Sigma~\!\! L$ is not sober. For each $(A, K)\in \ir_c(\Sigma~\!\! L)\times\mk (\Sigma~\!\! L)$, by Lemma \ref{Scott compact closed} or Corollary \ref{complete lattice Scott compact closed}, we have $\downarrow (A\cap K)\in \Gamma (\Sigma~\!\! L)$. So the condition that $\Psi_{\mk (X)}(A)\in \mathrm{Id}(\mk (X))$ in Corollary \ref{super Psi} and Corollary \ref{sober Psi-1} is essential. Therefore, two answers to \cite[Problem 4.3 and Problem 4.4]{xuzhao} are both negative.

\end{remark}

\section{H-sober reflections and super H-sober reflections of $T_0$ spaces}

In this section, we shall give a direct construction of the H-sobrifcations and super H-sobrifications of $T_0$ spaces, and investigate some basic properties of the H-sober reflections and super H-sober reflections of $T_0$ spaces.

In what follows, $\mathbf{K}$ always refers to a full subcategory $\mathbf{Top}_0$ containing $\mathbf{Sob}$, the objects of $\mathbf{K}$ are called $\mathbf{K}$-\emph{spaces}.

\begin{definition}\label{K-reflection}
	Let $X$ be a $T_0$ space. A $\mathbf{K}$-\emph{reflection} of $X$ is a pair $\langle \widetilde{X}, \mu\rangle$ consisting of a $\mathbf{K}$-space $\widetilde{X}$ and a continuous mapping $\mu :X\longrightarrow \widetilde{X}$ satisfying that for any continuous mapping $f: X\longrightarrow Y$ to a $\mathbf{K}$-space, there exists a unique continuous mapping $f^* : \widetilde{X}\longrightarrow Y$ such that $f^*\circ\mu=f$, that is, the following diagram commutes.\\
\begin{equation*}
	\xymatrix{
		X \ar[dr]_-{f} \ar[r]^-{\mu}
		&\widetilde{X}\ar@{.>}[d]^-{f^*}\\
		&Y}
	\end{equation*}
For $\mathbf{K}=\mathbf{H}$-$\mathbf{Sob}$ (resp., $\mathbf{K}=\mathbf{SH}$-$\mathbf{Sob}$), the $\mathbf{K}$-\emph{reflection} of $X$ is called the \emph{H}-\emph{sober reflection} (resp., \emph{super H}-\emph{sober reflection}) of $X$, or the \emph{H}-\emph{sobrification} (resp., \emph{super H}-\emph{sobrification}) of $X$.
\end{definition}

By a standard argument, $\mathbf{K}$-reflections, if they exist, are unique up to homeomorphism. We shall use $X^k$ to denote the space of the $\mathbf{K}$-reflection of $X$ if it exists, and $X^h$ (resp., $X^{H}$) to denote the space of the {\rm H}-sobrification (resp., super {\rm H}-sobrification) of $X$ if it exists.

\begin{definition}\label{K closed with respect to homeomorphisms} A full subcategory $\mathbf{K}$ of $\mathbf{Top}_0$ is said to be \emph{closed with respect to homeomorphisms} if homeomorphic copies of $\mathbf{K}$-spaces are $\mathbf{K}$-spaces.
\end{definition}

Let ${\rm H} : \mathbf{Top}_0 \longrightarrow \mathbf{Set}$ be an R-subset system. Then by Lemma \ref{Ps functor} and Theorem \ref{Heckman-Keimel theorem}, $\mathbf{H}$-$\mathbf{Sob}$ and $\mathbf{SH}$-$\mathbf{Sob}$ are full subcategories $\mathbf{Top}_0$ containing $\mathbf{Sob}$ and are closed with respect to homeomorphisms.

\begin{definition}\label{K-determined set} (\cite{xu20})
	 A subset $A$ of a $T_0$ space $X$ is called $\mathbf{K}$-\emph{determined} provided for any continuous mapping $ f:X\longrightarrow Y$
to a $\mathbf{K}$-space $Y$, there exists a unique $y_A\in Y$ such that $\overline{f(A)}=\overline{\{y_A\}}$. Clearly, a subset $A$ of a space $X$ is a $\mathbf{K}$-determined set if{}f $\overline{A}$ is a $\mathbf{K}$-determined set.
Denote by $\mathbf{K}(X)$ the set of all closed $\mathbf{K}$-determined sets of $X$.
\end{definition}

One can easily show that if $f:X\longrightarrow Y$ is a continuous mapping in $\mathbf{Top}_0$, then $\overline{f(A)}\in \mathbf{K} (Y)$ for all $A\in \mathbf{K} (X)$ (see \cite[Lemma 3.11]{xu20}).

\begin{remark}\label{Sob WF d-space determined set} Let $X$ be a $T_0$ space. Then by Lemma \ref{sobd=irr} and Proposition \ref{H Hd HR HD ir}, we have the following conclusions:
 \begin{enumerate}[\rm (1)]
 \item $\mathbf{Sob}(X)=\mathcal R^d_c(X)=\mathcal R^R_c(X)=\mathcal R^D_c(X)=\ir_c (X)$.
 \item $\mathcal S_c(X)\subseteq\mathbf{K}(X)\subseteq\mathbf{Sob}(X)=\ir_c (X)$.
 \item $\mathcal S_c(X)\subseteq {\rm H}_c(X)\subseteq\mathbf{H}$-$\mathbf{Sob}(X)={\rm H}^d_c(X)\subseteq\mathbf{SH}$-$\mathbf{Sob}(X)={\rm H}^D_c(X)\subseteq\mathbf{Sob}(X)=\ir_c (X)$.
 \item $\mathcal S_c(X)\subseteq\mathcal D_c(X)\subseteq\mathbf{Top}_d(X)=\mathcal D^d_c(X)\subseteq \mathcal D^R_c(X)\subseteq \mathcal D^D_c(X)=\mathbf{Top}_w(X)\subseteq\ir_c (X)$.
 \end{enumerate}
 \end{remark}

\begin{lemma}\label{K-set prod}\emph{(\cite{xu20})}
	Let	$\{X_i: 1\leq i\leq n\}$ be a finite family of $T_0$ spaces and $X=\prod\limits_{i=1}^{n}X_i$ the product space. For $A\in\ir (X)$, the following conditions are equivalent:
\begin{enumerate}[\rm (1)]
	\item $A$ is $\mathbf{K}$-determined.
	\item $p_i(A)$ is $\mathbf{K}$-determined for each $1\leq i\leq n$.
\end{enumerate}
\end{lemma}

\begin{corollary}\label{K-closed set prod}\emph{(\cite{xu20})}
	Let	$X=\prod\limits_{i=1}^{n}X_i$ be the product of a finite family $\{X_i: 1\leq i\leq n\}$ of $T_0$ spaces. If $A\in\mathbf{K} (X)$, then $A=\prod\limits_{i=1}^{n}p_i(X_i)$, and $p_i(A)\in \mathbf{K} (X_i)$ for all $1\leq i \leq n$.
\end{corollary}

By Lemma \ref{K-set prod} and Corollary \ref{K-closed set prod}, we get the following four corollaries.

\begin{corollary}\label{H-sober d-set prod}
	Let $H : \mathbf{Top}_0 \longrightarrow \mathbf{Set}$ be an R-subset system and $X=\prod\limits_{i=1}^{n}X_i$ the product space of a finite family $\{X_i: 1\leq i\leq n\}$ of $T_0$ spaces. For $A\in\ir (X)$, the following conditions are equivalent:
\begin{enumerate}[\rm (1)]
	\item $A$ is H-sober determined.
	\item $p_i(A)$ is H-sober determined for each $1\leq i\leq n$.
\end{enumerate}
\end{corollary}

\begin{corollary}\label{H-sober d-closed set prod}
	Let $H : \mathbf{Top}_0 \longrightarrow \mathbf{Set}$ be an R-subset system and	$X=\prod\limits_{i=1}^{n}X_i$ the product of a finite family $\{X_i: 1\leq i\leq n\}$ of $T_0$ spaces. If $A\in H^d_c(X)$, then $A=\prod\limits_{i=1}^{n}p_i(X_i)$, and $p_i(A)\in H^d_c (X_i)$ for all $1\leq i \leq n$.
\end{corollary}

\begin{corollary}\label{super H-sober D-set prod}
	Let $H : \mathbf{Top}_0 \longrightarrow \mathbf{Set}$ be an R-subset system and $X=\prod\limits_{i=1}^{n}X_i$ the product space of a finite family $\{X_i: 1\leq i\leq n\}$ of $T_0$ spaces. For $A\in\ir (X)$, the following conditions are equivalent:
\begin{enumerate}[\rm (1)]
	\item $A$ is super H-sober determined.
	\item $p_i(A)$ is super H-sober determined for each $1\leq i\leq n$.
\end{enumerate}
\end{corollary}

\begin{corollary}\label{super H-sober D-closed set prod}
	Let $H : \mathbf{Top}_0 \longrightarrow \mathbf{Set}$ be an R-subset system and $X=\prod\limits_{i=1}^{n}X_i$ the product space of a finite family $\{X_i: 1\leq i\leq n\}$ of $T_0$ spaces. If $A\in H^D_c(X)$, then $A=\prod\limits_{i=1}^{n}p_i(X_i)$, and $p_i(A)\in H^D_c (X_i)$ for all $1\leq i \leq n$.
\end{corollary}

\begin{remark}\label{eta k embedding}(\cite{xu20})  The canonical mapping $\eta_X^k : X\longrightarrow P_H(\mathbf{K}(X))$ is a topological embedding, where $\eta_X^k(x)=\overline {\{x\}}$ for all $x\in X$.
\end{remark}

For the discussion of $\mathbf{K}$-reflections of $T_0$ spaces, we need the following key lemma.

\begin{lemma}\label{K-lemma f star} \emph{(\cite{xu20})}
Let $X$ be a $T_0$ space and $f:X\longrightarrow Y$ a continuous mapping to a $\mathbf{K}$-space $Y$. Then there exists a unique continuous mapping $f^* :P_H(\mathbf{K}(X))\longrightarrow Y$ such that $f^*\circ\eta_X^k=f$, that is, the following diagram commutes.
\begin{equation*}
\xymatrix{
	X \ar[dr]_-{f} \ar[r]^-{\eta_X^k}
	&P_H(\mathbf{K}(X))\ar@{.>}[d]^-{f^*}\\
	&Y}
\end{equation*}	

\noindent The unique continuous mapping $f^* :P_H(\mathbf{K}(X))\longrightarrow Y$ is defined by

$$\forall A\in\mathbf{K}(X),\ \  f^*(A)=y_A,$$

\noindent where $y_A\in Y$ is the unique point such that $\overline{f(A)}=\overline{\{y_A\}}$.
\end{lemma}

From Lemma \ref{K-lemma f star} we immediately deduce the following result.

\begin{theorem}\label{K-reflection theorem}\emph{(\cite{xu20})}
	Let $X$ be a $T_0$ space. If $P_H(\mathbf{K}(X))$ is a $\mathbf{K}$-space,  then the pair $\langle X^k=P_H(\mathbf{K}(X)), \eta_X^k\rangle$ is the $\mathbf{K}$-reflection of $X$, where $\eta_X^k :X\longrightarrow X^k$ is the canonical topological embedding.
\end{theorem}

\begin{definition}\label{K-adequate} (\cite{xu20}) $\mathbf{K}$ is called \emph{adequate} if for any $T_0$ space $X$, $P_H(\mathbf{K}(X))$ is a $\mathbf{K}$-space.
\end{definition}

\begin{corollary}\label{K-adequate reflective}\emph{(\cite{xu20})}
	If $\mathbf{K}$ is adequate, then $\mathbf{K}$ is reflective in $\mathbf{Top}_0$.
\end{corollary}

\begin{corollary}\label{K-fuctor}\emph{(\cite{xu20})}
	If $\mathbf{K}$ is adequate, then for any  $T_0$ spaces $X,Y$ and any continuous mapping $f:X\longrightarrow Y$, there exists a unique continuous mapping $f^k:X^k\longrightarrow Y^k$ such that $f^k\circ \eta_X^k=\eta_Y^k\circ f$, that is, the following diagram commutes.
		\begin{equation*}
	\xymatrix{
		X \ar[d]_-{f} \ar[r]^-{\eta_X^k} &X^k\ar[d]^-{f^k}\\
		Y \ar[r]^-{\eta_Y^k} &Y^k
	}
	\end{equation*}
For each $A\in \mathbf{K} (X)$, $f^k(A)=\overline{f(A)}$.
\end{corollary}

Corollary \ref{K-fuctor} defines a functor $K : \mathbf{Top}_0 \longrightarrow \mathbf{K}$, which is the left adjoint to the inclusion functor $I : \mathbf{K} \longrightarrow \mathbf{Top}_0$.

\begin{theorem}\label{K-reflectionprod}\emph{(\cite{xu20})}
	For an adequate $\mathbf{K}$ and a finite family $\{X_i: 1\leq i\leq n\}$ of $T_0$ spaces,  $(\prod\limits_{i=1}^{n}X_i)^k=\prod\limits_{i=1}^{n}X_i^k$ \emph{(}up to homeomorphism\emph{)}.
\end{theorem}

\begin{lemma}\label{Ph functor} Let $H : \mathbf{Top}_0 \longrightarrow \mathbf{Set}$ be an R-subset system. Then $P_H^d : \mathbf{Top}_0 \longrightarrow \mathbf{Top}_0$ is a covariant functor, where for any $f : X \longrightarrow Y$ in $\mathbf{Top}_0$, $P_H^d(f) : P_H(H^d_c(X)) \longrightarrow P_H(H^d_c(Y))$ is defined by $P_H^d(f)(A)=\overline{f(A)}$ for all $A\in H^d_c(X)$.
\end{lemma}

\begin{proof} For any $T_0$ space $X$ and any continuous mapping $f : X \longrightarrow Y$ to a $T_0$ space $Y$, one can easily deduce that $P_H({\rm H}^d_c(X))$ and $P_H({\rm H}^d_c(Y))$ are $T_0$. By Lemma \ref{Hd image}, $P_H^d(f) : P_H({\rm H}^d_c(X)) \longrightarrow P_H({\rm H}^d_c(Y))$ is well-defined. Let $\eta_X^h : X \longrightarrow P_H({\rm H}^d_c(X))$ be the canonical mapping, defined by $\eta_X^h(x)=\overline{\{x\}}$. It is easy to check that $\eta_X^h$ is an order and topological embedding (cf. Remark \ref{eta continuous}).

	{Claim 1:} $P_H^d(f)\circ\eta_X^h=\eta_Y^h\circ f$.
	
	For each $ x\in X$, we have
	$$P_H^d(f)(\eta_X^h(x))=P_H^d(f)(\overline{\{x\}})=\overline{f(\overline{\{x\}})}=\overline{f(\{x\})}=\overline{\{f(x)\}}=\eta_Y^h(f(x)),$$
	that is, the following diagram commutes.
	\begin{equation*}
	\xymatrix{
		X \ar[d]_-{f} \ar[r]^-{\eta_X^h} &P_H(({\rm H}^d_c(X)))\ar[d]^-{P_H^d(f)}\\
		Y \ar[r]^-{\eta_Y^h} &P_H({\rm H}^d_c(Y))	}
	\end{equation*}

{Claim 2:} $P_H^d(f): P_H({\rm H}^d_c(X))\longrightarrow P_H({\rm H}^d_c(Y))$ is continuous.

For $V\in\mathcal O(Y)$, we have
$$\begin{array}{lll}
P_H^d(f)^{-1}(\Diamond V)&=&\{B\in {\rm H}^d_c(X): P_H^d(f)(B)\in \Diamond V\}\\
&=&\{B\in {\rm H}^d_c(X): \overline{f(B)}\cap V\neq\emptyset \}\\
&=&\{B\in {\rm H}^d_c(X):  B\cap f^{-1}(V)\neq\emptyset\}\\
&=&\Diamond f^{-1}(V),
\end{array}$$
which is open in $P_H({\rm H}^d_c(X))$. This implies that $P_H^d(f)$ is continuous.

{Claim 3:} $P_H^d(id_X)=id_{P_H^d(X)}$

$P_H^d(id_X)(A)=\overline{A}=A$ for all $A\in {\rm H}^d_c(X)$.

{Claim 4:} For any continuous mapping $g : Y \longrightarrow Z $ in $\mathbf{Top}_0$, $P_H^d(g\circ f)=P_H^d(g)\circ P_H^d(f)$.

For each $A\in {\rm H}^d_c(X)$, $P_H^d(g\circ f)(A)=\overline{g(f(A))}=\overline{g(\overline{f(A)})}=P_H^d(g)\circ(P_H^d(f)(A)$.

\noindent Thus $P_H^d : \mathbf{Top}_0 \longrightarrow \mathbf{Top}_0$ is a covariant functor.
\end{proof}

\begin{remark}\label{fd continuous} Let $X$ be a $T_0$ space and $f : X \longrightarrow Y$ a continuous mapping to an {\rm H}-sober space $Y$. Then by Lemma \ref{K-lemma f star}, the mapping $f^d : P_H({\rm H}^d_c(X)) \longrightarrow Y$ is continuous, where $f^d$ is defined by
 $$\forall A\in{\rm H}^d_c(X),\ \  f^d(A)=y_A,$$
here $y_A\in Y$ is the unique point such that $\overline{f(A)}=\overline{\{y_A\}}$.
\end{remark}

We can directly check the continuity of $f^d$. In fact, for any  $V\in\mathcal O(Y)$, we have $(f^d)^{-1}(V)=\{A\in {\rm H}^d_c(X): f^d(A)\in V\}=\{A\in {\rm H}^d_c(X): \overline{\{f^d(A)\}}\cap  V\neq\emptyset\}=\{A\in {\rm H}^d_c(X): \overline{f(A)}\cap V\neq\emptyset \}=\{A\in {\rm H}^d_c(X):  A\cap f^{-1}(V)\neq\emptyset\}=\Diamond f^{-1}(V)\in \mathcal O(P_H({\rm H}^d_c(X)))$, proving that $P_H^d(f)$ is continuous.

\begin{theorem}\label{H-sober reflection}  Let $H : \mathbf{Top}_0 \longrightarrow \mathbf{Set}$ be an R-subset system. Then $\mathbf{H}$-$\mathbf{Sob}$ is adequate. Therefore, for any $T_0$ space $X$, $X^h=P_H(H^d_c(X))$ with the canonical topological embedding $\eta_{X}^h: X\longrightarrow X^h$ is the H-sobrification of $X$, where $\eta_{X}^h(x)=\overline{\{x\}}$ for all $x\in X$.
\end{theorem}

\begin{proof}
	Suppose that $X$ be a $T_0$ space. We show that $P_H({\rm H}^d_c(X))$ is an ${\rm H}$-sober space. Since $X$ is $T_0$, one can directly deduce that $P_H({\rm H}^d_c(X))$ is $T_0$. Let $\{
	A_i : i\in I\}\in {\rm H}(P_H({\rm H}^d_c(X)))$ and $A=\overline{\bigcup_{i\in I}A_i}$. We check that $A\in {\rm H}^d_c(X)$. For any continuous mapping $f : X \longrightarrow Y$ to an {\rm H}-sober space $Y$ and $i\in I$, by $A_i\in {\rm H}^d_c(X)$, there is a unique $y_i\in Y$ such that $\overline{f(A_i)}=\overline{\{y_i\}}$. By Remark \ref{fd continuous}, $f^d(\{
	A_i : i\in I\})=\{y_i : i\in I\}\in {\rm H}(Y)$, and hence by the {\rm H}-sobriety of $Y$, there is a unique $y\in Y$ such that $\overline{\{y_i : i\in I\}}=\overline{\{y\}}$.  Therefore, we have
 $$\begin{array}{lll}
           \overline{f(A)} & =\overline{f(\overline{\bigcup_{i\in I}A_i})}\\
	                       & =\overline{f(\bigcup_{i\in I}A_i)}\\
	                       & =\overline{\bigcup_{i\in I}f(A_i)}\\
	                       & =\overline{\bigcup_{i\in I}\overline{f(A_i)}}\\
                           & =\overline{\bigcup_{i\in I}\overline{\{y_i\}}}\\
                           & =\overline{\{y_i : i\in I\}}\\
                           & =\overline{\{y\}}.\\
	\end{array}$$
Thus $A\in {\rm H}^d_c(X)$. We check that $\overline{\{
	A_i : i\in I\}}=\overline{\{A\}}$ in $P_H({\rm H}^d_c(X))$. For any open set $\Diamond U$ ($U\in \mathcal O(X)$) of $P_H({\rm H}^d_c(X))$, we have

$$\begin{array}{lll}
\emptyset&\neq&\overline{\{A\}}\bigcap \Diamond U\\
&\Leftrightarrow&A\in\Diamond U\\
&\Leftrightarrow&A\bigcap U\neq \emptyset\\
&\Leftrightarrow&A_i\bigcap U\neq \emptyset \mbox{~for some~} i\in I\\
&\Leftrightarrow&A_i\in \Diamond U \mbox{~for some~} i\in I\\
&\Leftrightarrow&\{A_i : i\in I\}\bigcap \Diamond U\neq\emptyset\\.
\end{array}$$

\noindent It follows that $\overline{\{
	A_i : i\in I\}}=\overline{\{A\}}$ in $P_H({\rm H}^d_c(X))$. Thus $P_H({\rm H}^d_c(X))$ is {\rm H}-sober.

By Lemma \ref{K-lemma f star} or Theorem \ref{K-reflection theorem}, $\mathbf{H}$-$\mathbf{Sob}$ is adequate. Therefore, for any $T_0$ space $X$, the pair $\langle X^h=P_H({\rm H}^d_c(X)), \eta_{X}^h\rangle$ is the {\rm H}-sobrification of $X$.
\end{proof}

By Corollary \ref{K-adequate reflective}, Corollary \ref{K-fuctor} (or Lemma \ref{Ph functor}) and Theorem \ref{H-sober reflection}, we have the following corollary.

\begin{corollary}\label{H-fuctor} Let $H : \mathbf{Top}_0 \longrightarrow \mathbf{Set}$ be an R-subset system. Then $\mathbf{H}$-$\mathbf{Sob}$ is reflective in $\mathbf{Top}_0$. The functor $H^s : \mathbf{Top}_0 \longrightarrow \mathbf{H}$-$\mathbf{Sob}$ is a left adjoint to the inclusion functor $I : \mathbf{H}$-$\mathbf{Sob} \longrightarrow \mathbf{Top}_0$, where for any continuous mapping $f:X\longrightarrow Y$ in $\mathbf{Top}_0$, $H^s(X)=X^h=P_H(H^d_c(X))$, $H^s(Y)=Y^h=P_H(H^d_c(Y))$, and $H^s(f)=f^h :X^h\longrightarrow Y^h$ is the unique continuous mapping such that $f^h\circ \eta_X^h=\eta_Y^h\circ f$, that is, the following diagram commutes.
		\begin{equation*}
	\xymatrix{
		X \ar[d]_-{f} \ar[r]^-{\eta_X^h} &X^h\ar[d]^-{f^h}\\
		Y \ar[r]^-{\eta_Y^h} &Y^h
	}
	\end{equation*}
For each $A\in {\rm H}^d_c(X)$, $f^h(A)=\overline{f(A)}$.
\end{corollary}

By Corollary \ref{SH-sober is subcat of H-sober}, Theorem \ref{super H-sober HD-sober}, Theorem \ref {H-sober reflection} and Corollary \ref{H-fuctor}, we get the following two results.

\begin{theorem}\label{super H-sober reflection}  Let $H : \mathbf{Top}_0 \longrightarrow \mathbf{Set}$ be an R-subset system having property M. Then $\mathbf{SH}$-$\mathbf{Sob}$ is adequate. Therefore, for any $T_0$ space $X$, $X^{H}=P_H(H^D_c(X))$ with the canonical mapping $\eta_{X}^{H}: X\longrightarrow X^{H}$ is the super H-sobrification of $X$, where $\eta_{X}^{H}(x)=\overline{\{x\}}$ for all $x\in X$.
\end{theorem}

\begin{corollary}\label{super H-fuctor} Let $H : \mathbf{Top}_0 \longrightarrow \mathbf{Set}$ be an R-subset system having property M. Then $\mathbf{SH}$-$\mathbf{Sob}$ is reflective in $\mathbf{Top}_0$ and $\mathbf{SH}$-$\mathbf{Sob}$ is reflective in $\mathbf{H}$-$\mathbf{Sob}$. The functor $H^S : \mathbf{Top}_0 \longrightarrow \mathbf{SH}$-$\mathbf{Sob}$ is a left adjoint to the inclusion functor $I : \mathbf{SH}$-$\mathbf{Sob} \longrightarrow \mathbf{Top}_0$, where for any continuous mapping $f:X\longrightarrow Y$ in $\mathbf{Top}_0$, $H^S(X)=X^{H}=P_H({\rm H}^D_c(X))$, $H^S(Y)=Y^{H}=P_H({\rm H}^D_c(Y))$, and $H^S(f)=f^{H} :X^{H}\longrightarrow Y^{H}$ is the unique continuous mapping such that $f^{H}\circ \eta_X^h=\eta_Y^h\circ f$, that is, the following diagram commutes.
		\begin{equation*}
	\xymatrix{
		X \ar[d]_-{f} \ar[r]^-{\eta_X^{H}} &X^{H}\ar[d]^-{f^{H}}\\
		Y \ar[r]^-{\eta_Y^{H}} &Y^{H}
	}
	\end{equation*}
For each $A\in H^D_c(X)$, $f^{H}(A)=\overline{f(A)}$.
\end{corollary}

By Theorem \ref{K-reflectionprod}, Theorem \ref{H-sober reflection} and Theorem \ref{super H-sober reflection}, we get the following two corollaries.

\begin{corollary}\label{H-sober reflection prod}  Let $H : \mathbf{Top}_0 \longrightarrow \mathbf{Set}$ be an R-subset system and $\{X_i: 1\leq i\leq n\}$ a finite family of $T_0$ spaces. Then $(\prod\limits_{i=1}^{n}X_i)^h=\prod\limits_{i=1}^{n}X_i^h$ \emph{(}up to homeomorphism\emph{)}.
\end{corollary}

\begin{corollary}\label{super H-sober reflection prod}  Let $H : \mathbf{Top}_0 \longrightarrow \mathbf{Set}$ be an R-subset system having property M and $\{X_i: 1\leq i\leq n\}$ a finite family of $T_0$ spaces. Then $(\prod\limits_{i=1}^{n}X_i)^{H}=\prod\limits_{i=1}^{n}X_i^{H}$ \emph{(}up to homeomorphism\emph{)}.
\end{corollary}

We can directly apply the main results of this section to the R-subset systems $\mathcal D$ and $\mathcal R$, and obtain the following known results: (1) $\mathbf{Top}_d$, $\mathbf{Sob}$ and $\mathbf{Top}_w$ are all reflective in $\mathbf{Top}_0$; and (2) the $d$-space reflection and the well-filtered reflection preserve finite products of $T_0$ spaces (cf. \cite{xu20}).

\section{Conclusion}

In this paper, we have provided a uniform approach to $d$-spaces, sober spaces and well-filtered spaces, and have developed a general framework for dealing with H-sober spaces and super H-spaces for an R-subset system H. One can continue some further investigations on the H-sober spaces, super H-spaces and other related structures. For example, for an an R-subset system H and a $T_0$ space $X$, we may call $X$ a \emph{H}-\emph{Rudin space} if any closed irreducible subset of $X$ is an H-Rudin set, and $X$ a \emph{H}-\emph{sober determined space} if any closed irreducible subset of $X$ is a  H-sober determined set. These two kinds of spaces are closely related to sober spaces, H-sober spaces, super H-sober spaces, locally compact spaces and core compact spaces (cf. \cite{xu-shen-xi-zhao1}). Their Smyth power spaces and categorical structures deserve further investigation.

We now close our paper with several questions.

\begin{question}\label{HR=HD question} Let ${\rm H} : \mathbf{Top}_0 \longrightarrow \mathbf{Set}$ be an R-subset system having property M. Does ${\rm H}^R(X)={\rm H}^D(X)$ hold for any $T_0$ space $X$?
\end{question}

\begin{question}\label{HRR=HR question}  Let ${\rm H} : \mathbf{Top}_0 \longrightarrow \mathbf{Set}$ be an R-subset system having property M. Is $R : \mathcal H\longrightarrow \mathcal H$, ${\rm H}\mapsto {\rm H}^R$, a closure operator?
\end{question}

\begin{question}\label{super H-sober prod question}
	Let ${\rm H} : \mathbf{Top}_0 \longrightarrow \mathbf{Set}$ be an R-subset system and $\{X_i:i\in I\}$ a family of super H-spaces. Is the product space $\prod_{i\in I}X_i$ super H-sober?
\end{question}

\begin{question}\label{super H-sober function space question}
	Let ${\rm H} : \mathbf{Top}_0 \longrightarrow \mathbf{Set}$ be an R-subset system, $X$ a $T_0$ space and $Y$ a super H-sober space. Is the function space $\mathbf{Top}_0(X, Y)$ equipped with the topology of pointwise convergence super H-sober?
\end{question}

\begin{question}\label{super H-sober equalizer question} Let ${\rm H} : \mathbf{Top}_0 \longrightarrow \mathbf{Set}$ be an R-subset system, $X$ a super H-sober space and $Y$ a $T_0$ space. For a pair of continuous mappings $f, g: X \longrightarrow Y$, is the equalizer
$E(f, g)=\{x\in X : f(x)=g(x)\}$ (as a subspace of $X$) super H-sober?
\end{question}

\begin{question}\label{super H-sober complete question} For an R-subset system ${\rm H} : \mathbf{Top}_0 \longrightarrow \mathbf{Set}$, is $\mathbf{SH}$-$\mathbf{Sob}$ complete?
\end{question}

\begin{question}\label{Hd HR HD property M question}
	If an R-subset system ${\rm H} : \mathbf{Top}_0 \longrightarrow \mathbf{Set}$ has property M, do the induced R-subset systems ${\rm H}^d$, ${\rm H}^R$ and ${\rm H}^D$ have property M?
\end{question}

\begin{question}\label{H-sober reflection-infinite-prod} Does the H-sobrification preserves arbitrary products of $T_0$ spaces? Or equivalently, does $(\prod\limits_{i\in I}X_i)^h=\prod\limits_{i\in I}X_i^h$ (up to homeomorphism) hold for any family $\{X_i : i\in I\}$ of $T_0$ spaces?
\end{question}

\begin{question}\label{super H-sober reflection-infinite-prod} For an R-subset system H having property M, does the super H-sobrification preserves arbitrary products of $T_0$ spaces? Or equivalently, does $(\prod\limits_{i\in I}X_i)^H=\prod\limits_{i\in I}X_i^H$ (up to homeomorphism) hold for any family $\{X_i : i\in I\}$ of $T_0$ spaces?
\end{question}

\begin{question}\label{super H-sobrification question}
	For an R-subset system ${\rm H} : \mathbf{Top}_0 \longrightarrow \mathbf{Set}$ and a $T_0$ space $X$, does the super H-sobrification of $X$ exist?
\end{question}

\begin{question}\label{super H-sober reflective question}
	For an R-subset system ${\rm H} : \mathbf{Top}_0 \longrightarrow \mathbf{Set}$, is $\mathbf{SH}$-$\mathbf{Sob}$ reflective in $\mathbf{Top}_0$?
\end{question}

\end{document}